\documentclass[11pt]{amsart}
\usepackage{amsmath, amsfonts, amssymb}
\usepackage{graphicx, amsthm, color}
\usepackage{color}

\usepackage{ucs}
\usepackage{hyperref}

\newtheorem{theorem}{Theorem}[section]
\newtheorem{lemma}[theorem]{Lemma}

\newtheorem{corollary}[theorem]{Corollary}
\newtheorem{proposition}[theorem]{Proposition}
\theoremstyle{definition}

\newtheorem{example}[theorem]{Example}

\theoremstyle{remark}
\newtheorem{remark}[theorem]{Remark}

\numberwithin{equation}{section}
\newcommand{\R}{\mathbb R}

\newcommand{\N}{\mathbb N}

\def\e{\varepsilon}

\begin{document}

\title[Approximation of an optimal control problem on a network]{Approximation of an optimal control problem on a network with a perturbed problem
in the whole space}

\date{\today}

\author[M. Camar-Eddine, M. Chuberre, M. Haddou, O. Ley]{Mohamed Camar-Eddine \and M\'eriadec Chuberre \and Mounir Haddou \and Olivier Ley}

\address{Univ Rennes, INSA Rennes, CNRS, IRMAR - UMR 6625, F-35000 Rennes, France}
\email{\{mcamared, meriadec.chuberre, mounir.haddou, olivier.ley\}@insa-rennes.fr}

\begin{abstract}
A classical optimal control problem posed in the whole space~$\mathbb{R}^2$ is perturbed 
by a singular term of magnitude $\varepsilon^{-1}$ aimed at driving the trajectories
to a prescribed network $\Gamma$. We are interested in the link between the limit problem, as~$\varepsilon\to0$, and some optimal control problems on networks studied in the literature. We prove that the sequence of trajectories admits a subsequential limit evolving on~$\Gamma$.  Moreover, in the case of the Eikonal equation,
we show that the sequence of value functions associated with the perturbed optimal control problems converges to a limit which, in particular, coincides with the value function of the expected optimal control problem set on the network~$\Gamma$.

\end{abstract}


\maketitle

\section{Introduction}
\label{sec:intro}

There are a lot of recent developments on the study of optimal control problems
and Hamilton-Jacobi-Bellman (HJB) equations on networks or stratified structures,
see~\cite{sc13, acct13, imz13, bbc13, bbc14} for the seminal works and the book~\cite{bc24}
for a recent overview of the theory.
This subject is beyond the classical theory of viscosity solutions since
the problem is not set in an open domain (with possibly boundary conditions)
as usual. A new theory has been developed to tackle the issue of networks.
The common approach is to consider controls and costs which are designed
to constrain the trajectories to remain on the network.
Then, define a HJB
equation on the network and, in particular, give a meaning of what happens
at the vertices.

Here, we follow a different approach by considering a network $\Gamma$
as a closed subset of $\R^2$. In the whole space $\R^2$,  we perturb a classical optimal 
control problem with an additional
singular term of magnitude $\frac{1}{\varepsilon}$
aimed at driving the trajectories closer and closer
to~$\Gamma$. The goal is to make the link between the limit problem
and some related optimal control problems studied in the
literature.
In this way, we encode  all the geometry of
the network  in the perturbation and we can approach some
optimal control problems on networks by studying classical
(singular perturbed) optimal control problems.
Many limits of the value functions of singular perturbed optimal control problems have been 
studied (see, for instance, \cite{ab03,ab10,ag00,bb98}). However, to our knowledge,
this kind of perturbation is new in the context of HJB
equation on networks.
There is a related work
by Achdou \& Tchou~\cite{at15}, where the control problem on a junction is
approximated by considering a state constraint
problem in an $\varepsilon$-neighborhood of the junction. 

In this work, which was initiated in Chuberre~\cite{chuberre23}, we consider a very simple model
case we will describe more precisely in the next section.
In the future,
we would like to investigate more general networks,
more general control problems and to apply our results
to numerical approximation of optimal control problems
on networks.

The structure of the paper is the following: In Section~\ref{nota-stat-results} we set up some general notations and state the main results (Theorems \ref{cv-traj}, \ref{edo-interieur} and \ref{limit-vf}). Section~\ref{sec:ode-eps} deals with the asymptotics of the perturbed ODE, whereas Section~\ref{sec:sko} is dedicated to the study of the qualitative properties of the trajectories. In Section~\ref{sec:control}, we investigate the limit behavior of the sequence of value functions associated with the perturbed ODE studied in Section~\ref{sec:ode-eps}. Finally, the appendix collects some useful properties of the gradient descent (Appendix~\ref{prop-dist}) and the proofs of some intermediate results used in the previous sections (Appendices~\ref{proof-beh-thr-origin} and~\ref{sec-proof-lemma123}).
\smallskip

\noindent{\bf Acknowledgements:}
This work is partially supported by ANR COSS (COntrol on Stratified Structures) ANR-22-CE40-0010
and the Centre Henri Lebesgue ANR-11-LABX-0020-01.

\section{Statement of the problem and main results}\label{nota-stat-results} 

We start by defining a classical optimal control problem
in $\R^2$.
Let $A$ be a compact metric space (the set of control values)
and $f:\R^2\times A\to \R^2$,
$\ell:\R^2\times A\to \R$ be continuous functions satisfying, for some $M>0$ and all $a\in A$, $x,y\in \R^2$,
\begin{eqnarray}\label{hyp-f-ell}
\begin{array}{ll}
|f(x,a)| \leq M, &  |f(x,a)-f(y,a)| \leq M|x-y|,\\
|\ell(x,a)|\leq M, & |\ell(x,a)-\ell(y,a)|\leq m_\ell (|x-y|),
\end{array}
\end{eqnarray}
(here and throughout the paper, $|\cdot|$ denotes indifferently the Euclidean norm in $\R$ and $\R^2$
and $m_\ell: [0,\infty)\to [0,\infty)$ is a modulus of continuity).
We consider the infinite horizon control problem given by the value function
\begin{eqnarray}\label{vf-classique}
&& V(x):= \inf_{\alpha \in \mathcal{A}} \int_0^\infty e^{-\lambda t} \ell (X^{x,\alpha}(t),\alpha(t))dt,
\end{eqnarray}
where $\lambda >0$ is a discount parameter,
$\mathcal{A}:= L^\infty ((0,\infty);A)$ is the set of controls,
and
$X^{x,\alpha}$ is the controlled trajectory of $\R^2$ given by the ODE
\begin{eqnarray}\label{traj}
&&  X^{x,\alpha}(0)=x\in\R^2, \quad   \dot{X}^{x,\alpha}(t)=f(X^{x,\alpha}(t), \alpha(t)), \  t >  0.
\end{eqnarray}

To study the control problem (\ref{vf-classique})-(\ref{traj}) on some network $\Gamma$, the by-now classical way is to
restrict the class of controls to those leading to trajectories~\eqref{traj}
remaining on $\Gamma$, that is, for $x\in\Gamma$,
\begin{eqnarray}\label{def-Ax}
&&  \mathcal{A}_x:= \{\alpha \in \mathcal{A} :  X^{x,\alpha} \text{ is a solution of~\eqref{traj} and }
  X^{x,\alpha}(t)\in \Gamma \text{ for all $t\geq 0$}\}.
\end{eqnarray}
Under some additional assumptions on the set of controls, it is proved (\cite{acct13}) that
the value function~\eqref{vf-classique} restricted to trajectories living on $\Gamma$,
\begin{eqnarray}\label{vf-gamma}
&& V_\Gamma(x):= \inf_{\alpha \in  \mathcal{A}_x} \int_0^\infty e^{-\lambda t} \ell (X^{x,\alpha}(t),\alpha(t))dt,
\end{eqnarray}
is the unique viscosity solution to a HJB equation set on $\Gamma$.
We do not give the precise assumptions and definitions here, we refer the reader to
Achdou \emph{et al.}~\cite{acct13, aot15}.
This notion of solution is, in many usual cases, equivalent to those
developed in~\cite{imz13, ls17}, see~\cite{bc24} and references therein for details.

In this work,
we follow another approach. To simplify the exposition,
we consider a very simple  junction $\Gamma\subset \R^2$
consisting of the union
of the two axes,
\begin{eqnarray}\label{reseau}
&&  \Gamma = \{O \} \cup \bigcup_{i\in \{E,N,W,S\}} (0,\infty) e_i,
\end{eqnarray}
where $O=(0,0)$ is the origin, called the junction, and the $e_i$'s are the four
unit vector pointing to the direction
of the cardinal points $e_E=(1,0)$, $e_N=(0,1)$, $e_W=(-1,0)$ and $e_S=(0,-1)$.
For further purpose, we set $e_i^\perp = \mathcal{R}(e_i)$, where $\mathcal{R}$
is the rotation by angle $\frac{\pi}{2}$ in the plane.
Let us consider a $C^{1,1}$ function $d_\Gamma :\R^2\to \R$
such that
\begin{eqnarray}\label{cond-d}
\Gamma = \{ d_\Gamma = 0\} = \{ \nabla  d_\Gamma = (0,0) \}.
\end{eqnarray}
In the sequel, we will simply write  $d_\Gamma= d$ and choose
\begin{eqnarray}\label{choix-d}
  d(x)= x_1^2 x_2^2, \quad \text{for any $x=(x_1,x_2)\in\R^2$}.
\end{eqnarray}
It is obvious that $d$ fulfills~\eqref{cond-d}. The reason for
choosing $\Gamma$ given by~\eqref{reseau} and the above function $d$
is that  $d$ is of polynomial type and the computations are simple.
Actually, our framework should be easily generalized to
the case of general junctions with a finite number of edges (or branches), that is,
\begin{eqnarray}\label{jonction-generale}
\Gamma = \{O \} \cup \bigcup_{i\in \mathcal{I}} (0,\infty) e_i,
\end{eqnarray}
where $\{e_i\}_{i\in \mathcal{I}}$ is a finite set of distinct unit vectors in $\R^2$. In this case it is always possible to
choose a subanalytic function $d_\Gamma$ satisfying~\eqref{cond-d} and sharing many of the properties
of the particular function $d$ defined by~\eqref{choix-d} (see~\cite[Chapter 6]{chuberre23}
and Appendix~\ref{prop-dist}).
We leave this for a future work. 
Let us mention that, to illustrate our proposed approach in a very simple framework,
the straightforward case of the ``line network'' $\Gamma= \{O \} \cup  (0,\infty)e_E\cup  (0,\infty)e_W$ is described in Appendix~\ref{line-junction}.

\smallskip
We now introduce a $\frac{1}{\varepsilon}$-perturbed optimal control problem 
\begin{eqnarray}\label{val-eps}
&& V^\varepsilon (x):= \inf_{\alpha \in \mathcal{A}} \int_0^\infty e^{-\lambda t} \ell (X^{x,\alpha, \varepsilon}(t),\alpha(t))dt,
\end{eqnarray}
where $X^{x,\alpha, \varepsilon}$ is the solution to the ODE
\begin{eqnarray}\label{traj-eps}
  && \dot{X}^{x,\alpha, \varepsilon}(t)= F^\varepsilon (X^{x,\alpha, \varepsilon}(t), \alpha(t))
  := f(X^{x,\alpha, \varepsilon}(t), \alpha(t)) -\frac{1}{\varepsilon} \nabla d (X^{x,\alpha, \varepsilon}(t)), \ t>0,
\end{eqnarray}
satisfying  $X^{x,\alpha, \varepsilon}(0)=x$.

We describe our main results. Let
\begin{eqnarray}\label{def-k}
&& k^{x,\alpha, \varepsilon}(t)= \frac{1}{\varepsilon}\int_0^t \nabla d (X^{x,\alpha, \varepsilon}(s)) ds, \quad t\geq0.
\end{eqnarray}  

\begin{theorem}[Convergence of trajectories]  \label{cv-traj}
Consider~\eqref{traj-eps} where $f$ satisfies~\eqref{hyp-f-ell} and $d$ is given by~\eqref{choix-d}.
  For every $x\in\R^2$ and $\alpha\in \mathcal{A}$,
\begin{enumerate}
\item There exists a unique solution $X^{x,\alpha, \varepsilon}\in W_{\rm loc}^{1,\infty}[0,\infty)$
  of~\eqref{traj-eps}.
\item Up to extraction, $(X^{x,\alpha, \varepsilon}, k^{x,\alpha, \varepsilon})$ converges locally uniformly on $(0,\infty)$
  to some $(X^{x,\alpha},k^{x,\alpha})\in H_{\rm loc}^1(0,\infty)\times H_{\rm loc}^1(0,\infty)$ when $\varepsilon\to 0$, such that
\begin{eqnarray}
&& \text{for all $t\geq 0$, }    X^{x,\alpha}(t)=x+\int_0^t f(X^{x,\alpha}(s),\alpha(s))ds - k^{x,\alpha}(t),\label{ppte-skorokhod}\\
&& \text{for all $t > 0$, } X^{x,\alpha}(t) \in \Gamma, \nonumber\\
&&  X^{x,\alpha}(0^+)= \overline{x}, \ k^{x,\alpha}(0^+)= x-\overline{x},\nonumber
\end{eqnarray} 
with
\begin{eqnarray}\label{proj-gamma}
&& \overline{x}:= \lim_{t\to \infty} Z^x (t),
\end{eqnarray}
where $Z^x$ is the unique solution of
\begin{eqnarray}\label{edo-grad1}
&& \dot{Z}^x= -\nabla d(Z^x), \ Z^x(0)=x.
\end{eqnarray}

\item If $x\in\Gamma$, then, up to extraction,
  $(X^{x,\alpha, \varepsilon}, k^{x,\alpha, \varepsilon})$ converges locally uniformly
  on $[0,\infty)$ to $(X^{x,\alpha},k^{x,\alpha})\in H_{\rm loc}^{1} [0,\infty)\times  H_{\rm loc}^{1} [0,\infty)$ when $\varepsilon\to 0$.

\item $(X^{x,\alpha},k^{x,\alpha})\in W_{\rm loc}^{1,\infty} (0,\infty)\times  W_{\rm loc}^{1,\infty} (0,\infty)$ and
\begin{eqnarray}\label{estim-derivees}
&& |\dot{X}^{x,\alpha}(t)| , |\dot{k}^{x,\alpha}(t)| \leq |f|_\infty \qquad \text{for a.e. $t\in (0,\infty)$.}
\end{eqnarray}

\end{enumerate}
\end{theorem}

Following the above results, we introduce the $\omega$-limit set of all possible trajectories,
\begin{eqnarray}\label{omega-lim}
&&  \omega (x,\alpha)= \{ X^{x,\alpha} \in H_{\rm loc}^1(0,\infty): \text{$(X^{x,\alpha},k^{x,\alpha})$ is a subsequential limit} \\
&&\hspace*{7cm}  \text{of $(X^{x,\alpha, \varepsilon},k^{x,\alpha, \varepsilon})$ in
    \eqref{traj-eps}-\eqref{def-k}}\}.\nonumber
\end{eqnarray}
Theorem~\ref{cv-traj} implies that $\omega (x,\alpha)\not=\emptyset$
but we are not able to prove that $\omega (x,\alpha)$ reduces to one element.
More precisely, we cannot characterize completely a solution $(X^{x,\alpha},k^{x,\alpha})$
of the limit problem. Property~\eqref{ppte-skorokhod} means that
$(X^{x,\alpha},k^{x,\alpha})$
is a solution of a Skorokhod problem~\cite{ls84}. But, such a problem on~$\Gamma$ faces non-uniqueness
and instability. See Section \ref{effetzenon} for discussion.
Later on, writing $\omega (x,\alpha)=\{X^{x,\alpha}\}$ means that the whole sequence
$(X^{x,\alpha, \varepsilon}, k^{x,\alpha, \varepsilon})_\varepsilon$ converges to a unique element $(X^{x,\alpha},k^{x,\alpha})$.

We obtain a partial characterization result inside the edges.
Let $X^{x,\alpha}\in  \omega (x,\alpha)$. By continuity of the trajectory for $t>0$,
we can write
\begin{eqnarray}\label{junctionandedges}
(0,\infty)= \mathcal{Z}(X^{x,\alpha}) \cup \bigcup_{0 \leq {n}\leq l} (a_{n}(X^{x,\alpha}), b_{n}(X^{x,\alpha})),  
\end{eqnarray}
where $l=l(X^{x,\alpha})$ is possibly infinite, $\mathcal{Z}(X^{x,\alpha}):=\{ t\in (0,\infty) : X^{x,\alpha}(t)=O\}$
is the closed set of $(0,\infty)$ when the trajectory is at the junction,
and $(a_{n}(X^{x,\alpha}), b_{n}(X^{x,\alpha}))$ are the time intervals when the trajectory is inside a branch,
\emph{i.e.}, for every ${n}\in\N$, there exists $i=i({n})\in \{E,N,W,S\}$ such that, for all $t\in (a_{n}(X^{x,\alpha}), b_{n}(X^{x,\alpha}))$,
$X^{x,\alpha}(t)\in (0,\infty)e_{i({n})}$.
The intervals depend on $X^{x,\alpha}$ but we omit to write the dependence when there is no ambiguity.
\begin{theorem}[Dynamic on $\Gamma$ outside $O$]\label{edo-interieur}
Let $X^{x,\alpha}\in  \omega (x,\alpha)$.
\begin{enumerate}

\item If $a_{n}=0$ is the initial time, then, for all $t\in [0,b_{n}]$,
\begin{eqnarray*}
  \left\{
  \begin{array}{l}
\displaystyle X^{x,\alpha}(t)=\overline{x}+\int_0^t \langle f(X^{x,\alpha}(s),\alpha(s)),e_{i({n})}\rangle e_{i({n})} ds,\\
\displaystyle k^{x,\alpha}(t)=x-\overline{x}  +\int_0^t \langle f(X^{x,\alpha}(s),\alpha(s)),e_{i({n})}^\perp\rangle e_{i({n})}^\perp ds,
  \end{array}
  \right.
\end{eqnarray*}
where $\overline{x}$ is defined by~\eqref{proj-gamma}.

\item If $a_{n}\not= 0$, then, for all $t\in [a_{n},b_{n}]$,
\begin{eqnarray*}
  \left\{
  \begin{array}{l}
\displaystyle X^{x,\alpha}(t)=\int_{a_{n}}^t \langle f(X^{x,\alpha}(s),\alpha(s)),e_{i({n})}\rangle e_{i({n})} ds,\\
\displaystyle k^{x,\alpha}(t)= k^{x,\alpha}(a_{n})  +\int_{a_{n}}^t \langle f(X^{x,\alpha}(s),\alpha(s)),e_{i({n})}^\perp\rangle e_{i({n})}^\perp ds.
  \end{array}
  \right.
\end{eqnarray*}

\end{enumerate}
\end{theorem}
\noindent This result gives a characterization of the limits $X^{x,\alpha}$ and $k^{x,\alpha}$
inside the edges in terms of projections of the given dynamic $f$ as expected.
A consequence of this result is that $X^{x,\alpha}$ (and so $k^{x,\alpha}$) is uniquely determined 
once we know the sequences $(a_n)_n$ and $(b_n)_n$, which seems to be hard to establish
for general $f$. Indeed, it is easy to build examples with Zeno effects, for which
the trajectory $X^{O,\alpha}$ visits all the edges for arbitrary small times.
See~\cite{jz23} and Lemma~\ref{lem:traj-visite-tout}.

We now turn to the natural asymptotic optimal control problem on $\Gamma$
induced by~\eqref{val-eps}-\eqref{traj-eps}, by
setting, for $x\in\R^2$,
\begin{eqnarray}\label{vf-naturelle}
&& \overline{V}(x):= \inf_{\begin{array}{c}\scriptstyle\alpha \in \mathcal{A}\\[-1mm]\scriptstyle X^{x,\alpha}\in  \omega (x,\alpha)\end{array}} \int_0^\infty e^{-\lambda t} \ell (X^{x,\alpha}(t),\alpha(t))dt.
\end{eqnarray}

We introduce two additional assumptions:
\begin{eqnarray}\label{calc-var}
&& A= \overline{B}(0,1) \text{ and } f(x,a)=a \text{ for all $x\in\R^2$, $a\in A$,}
\end{eqnarray}
and
\begin{eqnarray}\label{cout-indep}
  && \ell (x,a)=\ell(x)  \text{ for all $x\in\R^2$, $a\in A$.}
\end{eqnarray}
These assumptions appear to be very restrictive since, in this case,
the HJB equation,
which corresponds to the optimal control problem~\eqref{vf-classique},
is the classical Eikonal equation,
\begin{eqnarray}\label{hjb-eik} 
&& \lambda u(x) + |\nabla u (x)| = \ell (x), \quad x\in \R^2,
\end{eqnarray}
the unique viscosity solution of which is
the value function $V$ defined by \eqref{vf-classique}.

In this simple setting,
the assumptions of~\cite[Section 2]{acct13} hold, and we know
that $V_\Gamma$ in~\eqref{vf-gamma} is the unique viscosity solution of~\eqref{hjb-eik} 
{\em on $\Gamma$} in the sense of~\cite{acct13}, and we are in position
to make the link between the value function $\overline{V}$ and $V_\Gamma$.

\begin{theorem}[Convergence to the control problem on the network]\label{limit-vf}
Assume~\eqref{hyp-f-ell},~\eqref{calc-var} and~\eqref{cout-indep}. Then,
  
\begin{enumerate}

\item  $\overline{V}=V_\Gamma$ on $\Gamma$, where $V_\Gamma$ is defined by~\eqref{vf-gamma}.
  
\item $V^\varepsilon$ converges locally uniformly on $\R^2$ to $\overline{V}\circ\phi_d$,
where $\phi_d: x\in\R^2\mapsto \overline{x}\in \Gamma$,
with $\overline{x}$ defined by~\eqref{proj-gamma}.
\end{enumerate}
\end{theorem}

This result brings a complete answer to our problem
in a simple setting, which intends to be a first step toward more general
ones.
Assumption~\eqref{calc-var} implies that the system is controllable at
the junction. This controllability assumption is in force in most of the papers
on networks (\cite{acct13, imz13, sc13, aot15}).
It ensures, in particular, the continuity of the value function.
Regarding~\eqref{cout-indep}, we provide a counterexample
(Example~\ref{contre-ex-la})
to the equality $\overline{V}=V_\Gamma$
when~\eqref{cout-indep} does not hold. It follows that this assumption  
cannot be easily removed without restricting ourselves to controls designed
for keeping the trajectories on  $\Gamma$.
For the reader's convenience, we describe all the above results
in the straightforward case of the line network, see Appendix~\ref{line-junction},
which can be read separately as a starting point.

The proof of Theorem \ref{limit-vf} relies on some optimal control techniques.
It is an intriguing question to know whether it is possible to
obtain the result by PDE techniques, \emph{i.e.}, by passing to the limit
directly in the viscosity inequalities satisfied by $V^\varepsilon$.
All these issues will be the aim of a future work.

\section{Study of the perturbed ODE~\eqref{traj-eps}}
\label{sec:ode-eps}

In this section, we study the well-posedness and the convergence, as $\varepsilon\to 0$, of
the $\frac{1}{\varepsilon}$-perturbed ODE~\eqref{traj-eps}.
In the whole section, we will assume that $f$ satisfies~\eqref{hyp-f-ell} and that
$\Gamma$ and $d$ are given by~\eqref{reseau} and~\eqref{choix-d}.
We are looking for estimates which are independent of $\alpha\in \mathcal{A}$
and $\e >0$.

Some elementary calculations show that the function $d$ satisfies, for any $x\in\R^2$, 
\begin{equation}\label{propri-d}
  \langle x,\nabla d(x)\rangle=4d(x),\
   |\nabla d(x)|=2d^\frac{1}{2}(x) |x|,\
  |x|\geq\sqrt{2}d^\frac{1}{4}(x),
  |\nabla d(x)|\geq2\sqrt{2}d^\frac{3}{4}(x),\
\end{equation}
and, some useful properties of~\eqref{traj-eps} in the case $f\equiv 0$ are collected in Theorem~\ref{thm-loja} and Proposition~\ref{traj-d1}. 

\begin{lemma}[Well-posedness of~\eqref{traj-eps} for $\varepsilon >0$] \label{lem-exis}
  Assume~\eqref{hyp-f-ell}. Then, for
  every $\e >0$, $\alpha\in \mathcal{A}$ and $x\in\R^2$, there exists a unique solution
  $X^{x,\alpha, \varepsilon}\in W_{\rm loc}^{1,\infty}[0,\infty)$ of~\eqref{traj-eps}. Moreover, $X^{x,\alpha, \varepsilon}$ satisfies
\begin{eqnarray}\label{bornXeps}
&& |X^{x,\alpha, \varepsilon}(t)| \leq |x|+\sqrt{2}|f|_\infty t, \quad \text{for all $t\geq 0$.}
\end{eqnarray}
\end{lemma}

\begin{proof}
By~\eqref{hyp-f-ell} and the smoothness of $d$, the map $(x,a)\in \R^2\times A\mapsto F^\varepsilon(x,a)$,
defined in~\eqref{traj-eps},
is locally Lipschitz continuous with respect to $x$ and measurable with respect to $a$.
(Note that the Lipschitz constant of $F^\varepsilon$ is of order $1/\e$, which is the main
difficulty in the study.)
Thanks to Cauchy-Lipschitz theorem (e.g.,~\cite{hale80} for a measurable in $t$ version),
we have existence and uniqueness of a maximal solution $X^{x,\alpha, \varepsilon}$ in any
interval $[0,T),\, T>0$. Using~\eqref{hyp-f-ell} and~\eqref{propri-d}, we have, for a.e. $0\leq t<T$,
\begin{eqnarray}\label{calc-norm2} 
 \frac{1}{2}\frac{d}{dt}|X^{x,\alpha, \varepsilon}(t)|^2
  &=&\langle X^{x,\alpha, \varepsilon}(t), F^\varepsilon(X^{x,\alpha, \varepsilon}(t),\alpha(t))\rangle\\ \nonumber
&= &\langle X^{x,\alpha, \varepsilon}(t), f(X^{x,\alpha, \varepsilon}(t),t) \rangle -\frac{4}{\varepsilon}d(X^{x,\alpha, \varepsilon}(t)) \\ \nonumber
&\leq& \sqrt{2}|f|_\infty |X^{x,\alpha, \varepsilon}(t)|.
\end{eqnarray}
It follows that $|X^{x,\alpha, \varepsilon}(t)|^2\leq h(t):=|x|^2+ 2\sqrt{2}|f|_\infty \int_0^t |X^{x,\alpha, \varepsilon}(s)|ds$,
from which we infer that~$h'(t)\leq  2\sqrt{2}|f|_\infty \sqrt{h(t)}$ and then $h(t)\leq (\sqrt{h(0)}+\sqrt{2}|f|_\infty t)^2$.
We end up with inequality~\eqref{bornXeps} in $[0,T)$, which implies that the solution is global, \emph{i.e.}, $T=\infty$. 
\end{proof}

For any $\e >0$ and $\alpha\in \mathcal{A}$, we write
\begin{eqnarray}\label{decom-eps}
(0,\infty)= \mathcal{Z}(X^{x,\alpha,\e}) \cup \bigcup_{0 \leq {n}\leq N} (a_{n}(X^{x,\alpha,\e}), b_{n}(X^{x,\alpha,\e})),  
\end{eqnarray}
where $\mathcal{Z}(X^{x,\alpha, \e}):=\{ t\in (0,\infty) : d(X^{x,\alpha,\e}(t))=0\}$
and $(a_{n}(X^{x,\alpha,\e}), b_{n}(X^{x,\alpha,\e}))$ (denoted by $(a_n^\e,b_n^\e)$ for short)
are the time intervals when the trajectory is out of the (closed) set $\Gamma$.

\begin{lemma}\label{estim-int-dist}
Assume~\eqref{hyp-f-ell}.  
Let  $\e >0$, $\alpha\in \mathcal{A}$, $x\in\R^2$ and $T>0$. Then, there exists a positive constant~$C=C(T,|f|_\infty, |x|)$ such that, for any  $0\leq t_1\leq t_2\leq T$, we have 
\begin{equation}\label{estimintegrdxepsilon}
  \frac{1}{\varepsilon}\int_{{t_1}}^{{t_2}}d(X^{x,\alpha, \varepsilon}(t))dt
  \leq C(t_2-t_1)+  \frac{1}{4}\left( d^{\frac{1}{2}}(X^{x,\alpha, \varepsilon}(t_1))-d^{\frac{1}{2}}(X^{x,\alpha, \varepsilon}(t_2) \right).
\end{equation}
\end{lemma}

\begin{proof}
Let $a<b$ such that $[a,b]\subset (a_n^\e,b_n^\e)$ for some $n$.
Since for all $t\in (a_n^\e,b_n^\e)$, $d(X^{x,\alpha, \varepsilon}(t))\not= 0$, one can write,
using~\eqref{propri-d},
\begin{eqnarray*} 
-\frac{d}{dt}d^{\frac{1}{2}}(X^{x,\alpha, \varepsilon}(t))
&=&-\frac{1}{2}\langle \nabla d(X^{x,\alpha, \varepsilon}(t)), f(X^{x,\alpha, \varepsilon}(t),t)\rangle d^{-\frac{1}{2}}(X^{x,\alpha, \varepsilon}(t))
\\
& &\hspace{2cm}+ \frac{1}{2\varepsilon}|\nabla d(X^{x,\alpha, \varepsilon}(t)|^2d^{-\frac{1}{2}}(X^{x,\alpha, \varepsilon}(t))\\
&\geq& -\sqrt{2}|f|_\infty |X^{x,\alpha, \varepsilon}(t)| +\frac{4}{\varepsilon}d(X^{x,\alpha, \varepsilon}(t)).
\end{eqnarray*}
It follows, thanks to~\eqref{bornXeps}, that
\begin{eqnarray}\label{ab365}
  && \frac{1}{\varepsilon}\int_{a}^{b}d(X^{x,\alpha, \varepsilon}(t))dt \leq C(b-a)
  +\frac{1}{4}\left( d^{\frac{1}{2}}(X^{x,\alpha, \varepsilon}(a)) -  d^{\frac{1}{2}}(X^{x,\alpha, \varepsilon}(b))\right).
\end{eqnarray} 
By passing to the limit $a\to (a_n^\e)^+$ and $b\to (b_n^\e)^-$, we deduce that, for every
$(a_n^\e, b_n^\e)$,
\begin{eqnarray}\label{ab123}
  && \frac{1}{\varepsilon}\int_{a_n^\e}^{b_n^\e}d(X^{x,\alpha, \varepsilon}(t))dt \leq C(b_n^\e- a_n^\e).
\end{eqnarray}

Let now $0\leq t_1\leq t_2\leq T$ and denote by $I$ the subset of indices $n$ such that
$(a_n^\e, b_n^\e)\subset [t_1,t_2]$ for all $n\in I$. Set $\underline{a} =\inf_{n\in I}a_n^\e$
and $\overline{b} =\sup_{n\in I}b_n^\e$ and notice that, by continuity,
$d(X^{x,\alpha, \varepsilon}(\underline{a}))=d(X^{x,\alpha, \varepsilon}(\overline{b}))=0$.
We write
\begin{eqnarray}\label{equa-4-terms}
  \int_{t_1}^{t_2}d(X^{x,\alpha, \varepsilon}(t))dt
  &=&  \int_{t_1}^{\underline{a}}d(X^{x,\alpha, \varepsilon}(t))dt
  + \sum_{n\in I}\int_{a_n^\e}^{b_n^\e}d(X^{x,\alpha, \varepsilon}(t))dt
  \\
  && \hspace*{.cm} + \int_{\overline{b}}^{t_2}d(X^{x,\alpha, \varepsilon}(t))dt + \int_{\mathcal{Z}(X^{x,\alpha, \e})\cap [a,b]}d(X^{x,\alpha, \varepsilon}(t))dt\nonumber
\end{eqnarray}
and we estimate the terms in the right hand side.

Obviously, the last integral is zero and, from~\eqref{ab123},
\begin{eqnarray*}
&& \sum_{n\in I} \frac{1}{\varepsilon}\int_{a_n^\e}^{b_n^\e}d(X^{x,\alpha, \varepsilon}(t))dt\leq  C \sum_{n\in I} (b_n^\e- a_n^\e).
\end{eqnarray*}

To estimate the first integral of the right hand side in (\ref{equa-4-terms}), we consider two cases:
either $[t_1,\underline{a}]\subset \mathcal{Z}(X^{x,\alpha, \e})$
and the integral is zero,
or there exists $n_0$ such that $a_{n_0}^\e\leq t_1\leq b_{n_0}^\e=\underline{a}$. From~\eqref{ab365}, we obtain
\begin{eqnarray*}
  \frac{1}{\varepsilon}\int_{t_1}^{\underline{a}}d(X^{x,\alpha, \varepsilon}(t))dt
  &\leq & C(\underline{a}-t_1)
  +\frac{1}{4}\left( d^{\frac{1}{2}}(X^{x,\alpha, \varepsilon}(t_1)) -  d^{\frac{1}{2}}(X^{x,\alpha, \varepsilon}(\underline{a}))\right)\\
   &\leq & C(\underline{a}-t_1)
  +\frac{1}{4}d^{\frac{1}{2}}(X^{x,\alpha, \varepsilon}(t_1)).
\end{eqnarray*} 
With similar arguments, we obtain
\begin{eqnarray*}
  \frac{1}{\varepsilon}\int_{\overline{b}}^{t_2}d(X^{x,\alpha, \varepsilon}(t))dt
   &\leq & C(t_2-\overline{b})
  -\frac{1}{4}d^{\frac{1}{2}}(X^{x,\alpha, \varepsilon}(t_2)).
\end{eqnarray*} 

Putting all the estimates together and noticing that
$\underline{a}-t_1 +  \sum_{n\in I} (b_n^\e- a_n^\e) + t_2-\overline{b} \leq t_2- t_1$, we
finally obtain~\eqref{estimintegrdxepsilon}.
\end{proof}

A subset $S\subset\R^2$ is said to be {\em invariant for \eqref{traj-eps}} if,
for any~$t\geq0$, $X^{x,\alpha, \varepsilon}(t)\in S$. We recall that a sufficient condition to be invariant for
a subset $S$ with $C^1$ boundary is 
\begin{equation}\label{suff-inv}
\langle F^\varepsilon(x,a),n_{{\partial S}}(x)\rangle\leq 0, \quad \text{for all $x\in\partial S$, $a\in A$,}
\end{equation}
where $F^\varepsilon$ is defined in~\eqref{traj-eps} and 
$n_{{\partial S}}(x)$ denotes the outward unit normal vector to $\partial S$ at $x$.

We first investigate invariance properties of the sets
\begin{equation}\label{sublevel-set}
Z(\lambda):=\{x\in\R^2: d(x)\leq\lambda\},\ \lambda>0,
\end{equation}
and define the {\em entry time in~$Z(\lambda)$} from an initial position~$x\in\R^2$ by
\begin{equation}\label{entry-time}
t^{x,\alpha,\varepsilon}(\lambda):=\inf\{t\geq0: X^{x,\alpha, \varepsilon}(t)\in Z(\lambda)\}
\end{equation}
if the trajectory $X^{x,\alpha, \varepsilon}(\cdot)$ reaches $Z(\lambda)$, and $t^{x,\alpha,\varepsilon}(\lambda)=\infty$ otherwise.

\begin{proposition}[Invariant subsets and entry times]\label{prop-inv}
  Assume~\eqref{hyp-f-ell} and let
$\varepsilon >0$, $\alpha\in \mathcal{A}$ and $x\in\R^2$.
\begin{enumerate}

\item For any $0\leq t_1\leq t_2$, we have 
\begin{equation}\label{ineq123}
d(X^{x,\alpha, \varepsilon}(t_2))-d(X^{x,\alpha, \varepsilon}(t_1))\leq(t_2-t_1)\frac{|f|_\infty^2}{2}\varepsilon.
\end{equation}

\item If
\begin{eqnarray*} 
&& \lambda\geq\kappa \varepsilon^{4/3}, \text{ with $\kappa:=2^{-4/3}|f|_\infty^{4/3}$,}
\end{eqnarray*}
then $Z(\lambda)$ 
is invariant for~\eqref{traj-eps}. In particular,
\begin{eqnarray}\label{d4tiers} 
  && d(X^{x,\alpha, \varepsilon}(t))\leq \kappa\varepsilon^{4/3},
  \qquad \text{for all $x\in\Gamma$ and $t\geq 0$.}
\end{eqnarray}
Moreover, if $d(x)>\kappa\varepsilon^{4/3}$, then,
for any $0\leq t_1 < t_2 \leq t^{x,\alpha,\varepsilon}(\kappa\varepsilon^{4/3})$,  we have
\begin{equation}\label{decroiss1}
d(X^{x,\alpha, \varepsilon}(t_2)) \leq  d(X^{x,\alpha, \varepsilon}(t_1)).
\end{equation}

\item
For any $\gamma <1$ and
$\varepsilon \leq (4|f|_\infty/7)^{-1/(1-\gamma)}$, one has
\begin{eqnarray}\label{reach-4-3} 
&&  t^{x,\alpha,\varepsilon}( \varepsilon^{4\gamma/3})\leq 4 d(x)^{1/4} \varepsilon^{1-\gamma}.
\end{eqnarray}


\item For any $\lambda >0$ and $T>0$, there exists $C=C(T,|f|_\infty,|x|)$ such that
\begin{eqnarray}\label{reach-l} 
&&  t^{x,\alpha,\varepsilon}(\lambda)\leq C \frac{\varepsilon}{\lambda}.
\end{eqnarray}
  
\end{enumerate}
\end{proposition}

\begin{remark} \
\begin{itemize}

\item[(i)]
  In the particular case when $X^{x,\alpha, \varepsilon}(t_1)\in \Gamma$,
  Estimate~\eqref{ineq123} allows to recover an estimate similar to~\eqref{estimintegrdxepsilon}.

\item[(ii)] 
  The results of (2) mean that, if we start inside $Z(\kappa\varepsilon^{4/3})$, we stay in $Z(\kappa\varepsilon^{4/3})$ forever, otherwise, the ``distance'' $t\mapsto d(X^{x,\alpha, \varepsilon}(t))$ is non-increasing  
  until the trajectory $t\to X^{x,\alpha, \varepsilon}(t)$ reaches  $Z(\kappa\varepsilon^{4/3})$.
  The dynamics in the layer $Z(\kappa\varepsilon^{4/3})$ is
  more complicated and depends strongly on the control $\alpha$ through $f$. In particular,
  when $\lambda < \kappa\varepsilon^{4/3}$, $Z(\lambda)$ may not be   invariant anymore.

\item[(iii)] By enlarging $Z(\kappa\varepsilon^{4/3})$
  a little bit into $Z(\varepsilon^{4\gamma/3})$, $\gamma <1$,
we are able to estimate precisely the entry time.

\end{itemize}

Finally, note that all these estimates are independent of the control.
They only depend on the magnitude of $f$.
\end{remark}

\begin{proof}
For a.e. $t>0$, we have,
using~\eqref{traj-eps},
\begin{eqnarray} 
\label{egal-dd}
\frac{d}{dt}d(X^{x,\alpha, \varepsilon})
&=&\langle \nabla d(X^{x,\alpha, \varepsilon}), f(X^{x,\alpha, \varepsilon},\alpha)\rangle
- \frac{1}{\varepsilon}|\nabla d(X^{x,\alpha, \varepsilon})|^2\\ \nonumber
&\leq& \sqrt{2}|f|_\infty | \nabla d(X^{x,\alpha, \varepsilon})| - \frac{1}{\varepsilon}|\nabla d(X^{x,\alpha, \varepsilon})|^2\\  \nonumber
&\leq & \sup_{r\geq 0} \{ \sqrt{2}|f|_\infty r - \frac{r^2}{\varepsilon}\} =\frac{|f|_\infty^2}{2}\varepsilon.
\end{eqnarray}
Hence, \eqref{ineq123} is obtained by integrating the above inequality on the interval  $[t_1,t_2]$.

We turn to the proof of (2). To prove that $Z(\lambda)$ is invariant for $\lambda$ big enough, we check~\eqref{suff-inv}.
For any $x\in \partial Z(\lambda)$ and $a\in A$, using~\eqref{traj-eps} and~\eqref{propri-d},
we have 
\begin{eqnarray}\label{ineg478}
  \langle F^\varepsilon(x,a),n_{{\partial Z(\lambda)}}(x)\rangle
  &=&  \langle F^\varepsilon(x,a),\frac{\nabla d(x)}{|\nabla d(x)|}\rangle\nonumber\\
  &=&  \langle f(x,a),\frac{\nabla d(x)}{|\nabla d(x)|}\rangle -\frac{1}{\e} |\nabla d(x)| \nonumber\\
  &\leq & \sqrt{2}|f|_\infty  -\frac{2}{\e} \sqrt{d(x)} |x| \nonumber\\
  &\leq & \sqrt{2}|f|_\infty  -\frac{2\sqrt{2}}{\e} \lambda^{3/4},
\end{eqnarray}
the last inequality following the implication
\begin{eqnarray}\label{boule-level-set}
d(x)\geq \lambda \ \Longrightarrow \ |x|\geq \sqrt{2}\lambda^{1/4}.
\end{eqnarray}
The right-hand side in \eqref{ineg478} is nonpositive provided
$\lambda \geq 2^{-4/3}|f|_\infty^{4/3} \varepsilon^{4/3} =\kappa\varepsilon^{4/3}$,
which proves that~$Z(\lambda)$ is invariant.

The proof of~\eqref{decroiss1} is a straightforward consequence of~\eqref{ineg478} for $\lambda >\kappa\varepsilon^{4/3}$,
since $\frac{d}{dt} d(X^{x,\alpha, \varepsilon})
= \langle F^\varepsilon(X^{x,\alpha, \varepsilon},\alpha),\nabla d(X^{x,\alpha, \varepsilon})\rangle$.

We now prove (3). Let $\gamma <1$. For
$0\leq t < t^{x,\alpha, \varepsilon}(\e^{4\gamma/3})$,
we have $d(X^{x,\alpha, \varepsilon}(t)) > \e^{4\gamma/3}$. Hence, from~\eqref{propri-d},
we get
\begin{eqnarray*} 
&& |\nabla d(X^{x,\alpha, \varepsilon})|\geq 2\sqrt{2}\, d(X^{x,\alpha, \varepsilon})^{3/4} \geq 2\sqrt{2}\, \e^{\gamma}.
\end{eqnarray*}
Therefore, from~\eqref{egal-dd},
\begin{eqnarray*}
-\frac{d}{dt}d(X^{x,\alpha, \varepsilon})
&\geq & |\nabla d(X^{x,\alpha, \varepsilon})| (\frac{|\nabla d(X^{x,\alpha, \varepsilon})|}{\e} - \sqrt{2} |f|_\infty )\\
&\geq &  2\sqrt{2}\, d(X^{x,\alpha, \varepsilon})^{3/4} (2\sqrt{2}\,\e^{\gamma -1}  - \sqrt{2} |f|_\infty )\\
&\geq &  \e^{\gamma -1} d(X^{x,\alpha, \varepsilon})^{3/4}
\end{eqnarray*}
provided $\e\leq (4 |f|_\infty/7)^{-1/(\gamma -1)}$. Integrating the previous ordinary differential inequation,
we get~\eqref{reach-4-3}.

\smallskip

We turn to the proof of~\eqref{reach-l}.
It is possible to deduce the result using computations as in
the proof of~(3) but we provide
an alternative proof, which is interesting in itself.
From Lemmas~\ref{lem-exis} and~\ref{estim-int-dist}, we infer
\begin{equation*}
\displaystyle\frac{1}{\varepsilon}\int_{{0}}^{{T}}d(X^{x,\alpha, \varepsilon}(t))dt\leq C:= (|x|+|f|_\infty)T+\frac{1}{4}  d^{\frac{1}{2}}(x).
\end{equation*}
For $\lambda >0$, set $I=I(x,\alpha,\e,T,\lambda):=\{t\in[0,T] : d(X^{x,\alpha, \varepsilon}(t))\geq\lambda \}$.
We have 
\begin{equation}\label{boundemeasure}
\text{meas}(I)\leq \frac{1}{\lambda}\int_{I}d(X^{x,\alpha, \varepsilon}(t))dt \leq C \frac{\varepsilon}{\lambda},
\end{equation}
where $\text{meas}(I)$ denotes the Lebesgue measure of the Borel set $I$.
It follows that $t^{x,\alpha,\varepsilon}(\lambda)\leq C \varepsilon\lambda^{-1}$.
\end{proof}

\begin{lemma}[Boundedness of the sequence $(X^{x,\alpha, \varepsilon}, k^{x,\alpha, \varepsilon})$
in $H_{\rm loc}^1(0,\infty)\times H_{\rm loc}^1(0,\infty)$]\label{local-unif-bound}
Assume~\eqref{hyp-f-ell}. Let $\varepsilon >0$, $\alpha\in \mathcal{A}$,  
$X^{x,\alpha, \varepsilon}$ be the solution of~\eqref{traj-eps}
and recall that $k^{x,\alpha, \varepsilon}$ is defined by~\eqref{def-k}. Then,
for all $0<\eta \leq T$, there exists $C=C(\eta, T, |f|_\infty)$ (independent of $\e$ and $\alpha$) such that
\begin{eqnarray*} 
&& \|X^{x,\alpha, \varepsilon}\|_{H^1 [\eta,\,T]},  \|k^{x,\alpha, \varepsilon}\|_{H^1 [\eta,\,T]} \leq C.
\end{eqnarray*}
If $x\in\Gamma$, then $C$ is also independent of $\eta$.
\end{lemma}

\begin{proof}
The boundedness of $X^{x,\alpha, \varepsilon}$ in $L^\infty [0,T]$ for every $T>0$ is given by~\eqref{bornXeps}.
It follows from the ODE~\eqref{traj-eps} and the estimate~\eqref{bornXeps} that, for all $t\geq 0$,
\begin{eqnarray*} 
  |k^{x,\alpha, \varepsilon}(t)|&\leq&  |k^{x,\alpha, \varepsilon}(0)|+ \int_0^t |f(X^{x,\alpha, \varepsilon}(s),\alpha(s))|ds + |X^{x,\alpha, \varepsilon}(t)| + |X^{x,\alpha, \varepsilon}(0)|\\
  &\leq& 2|x|+(1+\sqrt{2})|f|_\infty t,
\end{eqnarray*}
which gives an $L^\infty$-bound, and therefore an $L^2$-bound for $k^{x,\alpha, \varepsilon}$ in~$[0,T]$ which is independent of $\alpha$ and $\varepsilon$.

We turn to the estimates of the derivatives.
From Equation~\eqref{traj-eps}, we have, for a.e. $0\leq t\leq T$,
\begin{eqnarray*}
|\dot{X}^{x,\alpha, \varepsilon}(t)|^2
&=& \langle \dot{X}^{x,\alpha, \varepsilon}(t), f(X^{x,\alpha, \varepsilon}(t),t) \rangle
  -\frac{1}{\varepsilon}\langle \dot{X}^{x,\alpha, \varepsilon}(t), \nabla d(X^{x,\alpha, \varepsilon}(t)) \rangle  \\
&\leq &\sqrt{2}|f|_\infty |\dot{X}^{x,\alpha, \varepsilon}(t)|-\frac{1}{\varepsilon}\frac{d}{dt}d(X^{x,\alpha, \varepsilon}(t)).
\end{eqnarray*}
Using Young's inequality and
integrating on any interval $[t_1,T]\subset [0,T]$ leads to
\begin{eqnarray*}
  && \frac{1}{2} \int_{t_1}^{T} |\dot{X}^{x,\alpha, \varepsilon}(t)|^2dt
  \leq |f|_\infty^2(T-t_1)+ \frac{1}{\varepsilon}\left( d(X^{x,\alpha, \varepsilon}(t_1))- d(X^{x,\alpha, \varepsilon}(T))\right).
\end{eqnarray*}

First, if $x\in \Gamma$, then the previous inequality with $t_1=0$ gives the uniform bound
\begin{eqnarray*}
  && \frac{1}{2} \int_{0}^{T} |\dot{X}^{x,\alpha, \varepsilon}(t)|^2dt
  \leq |f|_\infty ^2T
\end{eqnarray*}
on the whole interval $[0,T]$.

When  $x\not\in \Gamma$ \emph{i.e.}, $d(x)>0$, for any $0<\eta \leq T$, we set $\lambda=C\e \eta^{-1}$ and
choose $t_1= t^{x,\alpha,\varepsilon}(\lambda)\leq C\e \lambda^{-1}=\eta$ (see Proposition~\ref{prop-inv}),
which gives
\begin{eqnarray*}
  \frac{1}{2} \int_{\eta}^{T} |\dot{X}^{x,\alpha, \varepsilon}(t)|^2dt
  &\leq & \frac{1}{2} \int_{t^{x,\alpha,\varepsilon}(\lambda)}^{T} |\dot{X}^{x,\alpha, \varepsilon}(t)|^2dt\\
  &\leq& |f|_\infty^2 T+ \frac{1}{\varepsilon} d(X^{x,\alpha, \varepsilon}(t^{x,\alpha,\varepsilon}(\lambda)))\\
  &\leq & |f|_\infty^2 T+\frac{C}{\eta}.
\end{eqnarray*}
The bound for $\dot{k}^{x,\alpha, \varepsilon}$ comes from  Equation~\eqref{traj-eps},
$\dot{X}^{x,\alpha, \varepsilon}=f-\dot{k}^{x,\alpha, \varepsilon}$.
Lemma \ref{local-unif-bound} is proved.
\end{proof}

\begin{lemma}[Compactness of the sequence of trajectories]\label{compactnesslemma}
Assume~\eqref{hyp-f-ell}. Then, for every
$\alpha\in \mathcal{A}$ and $x\in\R^2$,
there exists some $(X^{x,\alpha},k^{x,\alpha})\in H_{\rm loc}^1(0,\infty)\times H_{\rm loc}^1(0,\infty)$
such that, up to extraction, the sequence  $(X^{x,\alpha, \varepsilon}, k^{x,\alpha, \varepsilon})$ converges
pointwisely on $[0,\infty)$ and
locally uniformly on~$(0,\infty)$, to~$(X^{x,\alpha},k^{x,\alpha})$, as $\varepsilon\to 0$.
When $x\in\Gamma$,  $(X^{x,\alpha},k^{x,\alpha})\in H_{\rm loc}^1[0,\infty)\times H_{\rm loc}^1[0,\infty)$
and the convergence is  locally uniform on~$[0,\infty)$.
\end{lemma}

\begin{proof}
Thanks to Lemma~\ref{local-unif-bound}, we use Rellich Theorem (see~\cite{brezis83})
together with a diagonal extraction procedure in the interval~$[n^{-1},n]$, $n\geq 1$.
This gives the local uniform convergence of a subsequence of
$(X^{x,\alpha, \varepsilon_n}, k^{x,\alpha, \varepsilon_n})$ to some
$(X^{x,\alpha},k^{x,\alpha})\in H_{\rm loc}^1(0,\infty)\times H_{\rm loc}^1(0,\infty)$
on $(0,\infty)$, as $\e_n\to 0$.
The pointwise convergence at $t=0$ comes from the fact that $X^{x,\alpha, \varepsilon}(0)=x$
and $k^{x,\alpha, \varepsilon}(0)=0$ for all $\e >0$.
Note that, since $H_{\rm loc}^1(0,\infty)\subset C(0,\infty)$, the limit is continuous
on $(0,\infty)$ with a possible discontinuity at $0^+$.

When $x\in \Gamma$, the $\varepsilon$-uniform estimates of Lemma~\ref{local-unif-bound}
hold up to $0$, so the convergence is locally uniform in $[0,\infty)$ (we perform
the diagonal extraction in  $[0,n]$, $n\geq 1$).
The limit is then in $H_{\rm loc}^1[0,\infty)\times H_{\rm loc}^1[0,\infty)$ and, in particular, is
continuous at $t=0$.
\end{proof}

\begin{lemma}\label{saut-sur-reseau}
Let $x\in\R^2$ and $X^{x,\alpha}\in \omega(x,\alpha)$, where~$\omega(x,\alpha)$ is defined by~\eqref{omega-lim}. 
\begin{enumerate}
\item For all $t>0$, $X^{x,\alpha}(t)\in\Gamma$ and~\eqref{ppte-skorokhod} holds true.
\item If $(X^{x,\alpha},k^{x,\alpha})$ is a subsequential limit of $(X^{x,\alpha, \varepsilon}, k^{x,\alpha, \varepsilon})$, then
\begin{equation}
\lim_{t\to0^+}X^{x,\alpha}(t)=\overline{x}
\quad \text{ and } \quad
\lim_{t\to0^+}k^{x,\alpha}(t)=x-\overline{x},
\end{equation}
where $\overline{x}$ is defined by~\eqref{proj-gamma}.
\end{enumerate}
\end{lemma}

\begin{proof} We prove (1).
Let $\gamma <1$. From Proposition~\ref{prop-inv},
we get that $Z(\e^{4\gamma/3})$ is invariant for $\e$ small
enough, and that
$t^{x,\alpha,\e}(\e^{4\gamma/3})= C\e^{1-\gamma}$ for some constant~$C$.
It follows that $d(X^{x,\alpha, \varepsilon}(t))\leq \e^{4\gamma/3}$,
for all $t\geq  C\e^{1-\gamma}$.
Keeping the same notation $\e$ for the extraction and
passing to the limit $\e\to 0$, we obtain $d(X^{x,\alpha}(t))=0$
for all $t>0$. This gives the result since $d$ satisfies~\eqref{cond-d}
by construction.
The proof of~\eqref{ppte-skorokhod} is a straightforward application
of Lebesgue's dominated convergence theorem to the integral
formulation of~\eqref{traj-eps}.

We turn to the proof of (2). Note that, once the first limit is established,
the second one is an immediate consequence of~\eqref{ppte-skorokhod}. 
We divide the proof in four steps.

{\it Step 1. $|X^{x,\alpha}(t)|$ has a limit $\ell\geq 0$ as $t\to 0^+$.}
From~\eqref{calc-norm2}, for every $T>0$, there exists a constant $C=C(T,|f|_\infty, |x|)$
such that, for all $0\leq t_1\leq t_2\leq T$,
\begin{eqnarray*}
&& \frac{1}{2} \left( |X^{x,\alpha, \varepsilon}(t_2)|^2- |X^{x,\alpha, \varepsilon}(t_1)|^2 \right) 
  = \int_{{t_1}}^{{t_2}} \langle X^{x,\alpha, \varepsilon}, f(X^{x,\alpha, \varepsilon},\alpha) \rangle dt
 - \frac{4}{\varepsilon}\int_{{t_1}}^{{t_2}}d(X^{x,\alpha, \varepsilon})dt.
\end{eqnarray*}
Thanks to Lemma~\ref{estim-int-dist} and~\eqref{bornXeps}, up to increasing $C$ if needed, it follows that 
\begin{eqnarray}\label{ineg592}
 && \left| |X^{x,\alpha, \varepsilon}(t_2)|^2- |X^{x,\alpha, \varepsilon}(t_1)|^2 \right|\\  \nonumber
  &\leq&   2\sqrt{2}|f|_\infty \int_{{t_1}}^{{t_2}}|X^{x,\alpha, \varepsilon}|dt
  + 8C(t_2-t_1) + 2(d^{\frac{1}{2}}(X^{x,\alpha, \varepsilon}(t_1))-d^{\frac{1}{2}}(X^{x,\alpha, \varepsilon}(t_2)))\\  \nonumber
 &\leq& C(t_2-t_1)  + 2(d^{\frac{1}{2}}(X^{x,\alpha, \varepsilon}(t_1))-d^{\frac{1}{2}}(X^{x,\alpha, \varepsilon}(t_2))).
\end{eqnarray}
Recalling that  $X^{x,\alpha, \varepsilon}(t)\to X^{x,\alpha}(t)\in\Gamma$ as $\e\to 0$, for $t>0$,
we arrive at
\begin{eqnarray*}
&& \left| |X^{x,\alpha}(t_2)|^2- |X^{x,\alpha}(t_1)|^2 \right|
\leq  C(t_2-t_1) \quad \text{ for all $0<t_1\leq t_2\leq T$.}
\end{eqnarray*}
We conclude that $(|X^{x,\alpha}(t)|)_{t>0}$ is a Cauchy sequence as $t\to 0^+$, thus
\begin{eqnarray}\label{limit-norme}
\lim_{t\to 0^+} |X^{x,\alpha}(t)|=:\ell 
\end{eqnarray}
exists, which may be different from $|X^{x,\alpha}(0)|$
since $X^{x,\alpha}(t)$ may have a discontinuity at $t=0$.

{\it Step 2. $X^{x,\alpha}(t)$ has a limit $\hat{x}\in\Gamma$ as $t\to 0^+$.}
First, since  $(X^{x,\alpha}(t))_{t>0}$ is bounded and
$X^{x,\alpha}(t)$ belongs to the closed set $\Gamma$, for all $t>0$, the set of cluster points
of $(X^{x,\alpha}(t))_{t>0}$ as $t\to 0^+$ is non empty and is contained in $\Gamma$.
To prove the result, it is enough to prove that there is only one cluster point.
If $\ell =0$ in~\eqref{limit-norme}, then the result holds true with $\hat{x}=O$.
We now deal with the case $\ell >0$.

We argue by contradiction, assuming that there exist two sequences $t_n, t_n' \to 0$
such that $X^{x,\alpha}(t_n)\to \hat{x}^{(1)}$, $X^{x,\alpha}(t_n')\to \hat{x}^{(2)}$
with $\hat{x}^{(1)}\not=  \hat{x}^{(2)}$. Without loss of generality, we may assume that
$0 < t_{n+1}' <  t_{n+1} < t_n' < t_n$ for all $n\geq 1$.
From~\eqref{limit-norme}, we infer that $\hat{x}^{(1)}, \hat{x}^{(2)}\in \Gamma\cap \mathcal{C}(O,\ell)$,
where $\mathcal{C}(O,\ell)$ is the circle of center $O$ and radius~$\ell >0$.
This means that  $\hat{x}^{(1)}$ and $\hat{x}^{(2)}$ are located in different edges
and so is for $X^{x,\alpha}(t_n)$ and~$X^{x,\alpha}(t_n')$ for $n$ large enough.
It follows, due to the continuity of  the map~$t\in (0,\infty)\mapsto X^{x,\alpha}(t)\in \Gamma$, that it is possible to
find $t_n' < t_n'' < t_n$ such that  $X^{x,\alpha}(t_n'')=O$ for all $n$.
Therefore, $|X^{x,\alpha}(t_n'')|\to 0$, which is in contradiction with~\eqref{limit-norme} when $\ell >0$.
The result is proved.

{\it Step 3. $|\hat{x}|=|\overline{x}|$, where $\overline{x}$ is defined by~\eqref{proj-gamma}.}
Let $T>0$. From Lemma~\ref{lem-exis}, for all $t\in [0,T]$, $X^{x,\alpha, \varepsilon}(t), Z^x(t)\in \overline{B}(0,R)$,
where $R=|x|+\sqrt{2}|f|_\infty T$, and recall that $Z^x$ is the solution of~\eqref{edo-grad1}.
We define the Lipschitz constant $L=L(T,|f|_\infty,|x|):= {\rm Lip}_{\overline{B}(0,R)}(\nabla d)$
of $\nabla d$ in $\overline{B}(0,R)$.

Set $Y^{x,\alpha, \varepsilon}(t):=X^{x,\alpha, \varepsilon}(\e t)$ for $t\in [0,\e^{-1}T]$.
From~\eqref{traj-eps}, we obtain that  $Y^{x,\alpha, \varepsilon}(t)$ is the unique solution of
\begin{eqnarray*}
  && \dot{Y}^{x,\alpha, \varepsilon}(t)=\e f(Y^{x,\alpha, \varepsilon}(t), \alpha(\e t))-\nabla d(Y^{x,\alpha, \varepsilon}(t)),
  \quad Y^{x,\alpha, \varepsilon}(0)=x.
\end{eqnarray*}
For a.e. $0<t\leq \e^{-1}T$, we have
\begin{eqnarray*}
  |\dot{Y}^{x,\alpha, \varepsilon}-\dot{Z}|
  &\leq & \e |f|_\infty + |\nabla d(Y^{x,\alpha, \varepsilon}) - \nabla d(Z^x)|\\
  &\leq & \e |f|_\infty + L |Y^{x,\alpha, \varepsilon} - Z^x|.
\end{eqnarray*}
Gr\"onwall inequality yields
\begin{eqnarray}\label{gron123}
  && |Y^{x,\alpha, \varepsilon}(t) - Z^x(t)|= |X^{x,\alpha, \varepsilon}(\e t) - Z^x(t)|\leq \frac{\e |f|_\infty}{L} e^{Lt}
  \quad \text{for all $0\leq t\leq  \e^{-1}T$.}
\end{eqnarray}
We have
\begin{eqnarray}\label{teps23}
  && t_\e:= -\frac{1}{L+1} \ln\e \mathop{\to}_{\e\to 0} \infty, \qquad \e t_\e \mathop{\to}_{\e\to 0} 0.
\end{eqnarray}
It follows from Theorem~\ref{thm-loja} and Proposition~\ref{traj-d1} that
$Z^x(t_\e)\to \overline{x}$ as $\e\to 0$, and, from~\eqref{gron123}, that
\begin{eqnarray*}
  && |X^{x,\alpha, \varepsilon}(\e t_\e) - Z^x(t_\e)|\leq \frac{\e |f|_\infty}{L}  e^{Lt_\e}= \frac{|f|_\infty}{L} \e^{\frac{1}{L+1}}\to 0
\end{eqnarray*}
Hence
\begin{eqnarray}\label{limteps}
X^{x,\alpha, \varepsilon}(\e t_\e) \mathop{\to}_{\e\to 0}  \overline{x}.
\end{eqnarray}
Writing~\eqref{ineg592} with $0<t_1=\e t_\e < t_2=t\leq T$, we obtain
\begin{eqnarray*}
  \left| |X^{x,\alpha, \varepsilon}(t)|^2- |X^{x,\alpha, \varepsilon}(\e t_\e)|^2 \right|
 &\leq& C(t-\e t_\e)  + 2(d^{\frac{1}{2}}(X^{x,\alpha, \varepsilon}(\e t_\e))-d^{\frac{1}{2}}(X^{x,\alpha, \varepsilon}(t))).
\end{eqnarray*}
Sending $\e\to 0$, we get $\left| |X^{x,\alpha}(t)|^2- |\overline{x}|^2 \right|\leq Ct$,
from which we conclude $|X^{x,\alpha}(t)|\to |\overline{x}|$ as $t\to 0^+$. Thus $|\hat{x}|=|\overline{x}|$.

{\it Step 4. $\hat{x}=\overline{x}$.}
The key idea is the same as in Step 2.
First, if $\overline{x}=0$, then, the result follows from Step 3.

We then consider the case
$|\overline{x}|=|\hat{x}|=\ell >0$,
and argue by contradiction assuming $\hat{x}\not=\overline{x}$.
Since  $\hat{x}, \overline{x} \in \Gamma\cap \mathcal{C}(O,\ell)$,
they are in different edges, say $\hat{x}= \ell e_N$
and $\overline{x}= \ell e_E$. Therefore, we can find $\eta >0$
small enough such that $B(\overline{x}, \eta)\subset \{ x_1 > x_2\}$ and $B(\hat{x}, \eta)\subset \{ x_1 < x_2\}$
are disjoint neighborhoods of $\overline{x}$ and $\hat{x}$, respectively.

Since $X^{x,\alpha}(t)\to \hat{x}$ as $t\to 0^+$, there exists $t_\eta >0$ such that
for every $0<t<t_\eta$, $X^{x,\alpha}(t)\in B(\hat{x}, \eta/2)$. We fix such a $t$.
From the pointwise convergence $X^{x,\alpha, \varepsilon}(t)\to X^{x,\alpha}(t)$, we deduce
the existence of $\e_t$ such that, for all $0 <\e <\e_t$,
$X^{x,\alpha, \varepsilon}(t)\in  B(\hat{x}, \eta)$.

From~\eqref{teps23} and~\eqref{limteps}, we infer the existence of $\e_{t_\eta} \leq\e_t$ such that,
for all $\e < \e_{t_\eta}$, $0 < \e t_\e < t$ and $X^{x,\alpha, \varepsilon}(\e t_\e)\in B(\overline{x}, \eta)$.
By continuity of the map $\tau\in [\e t_\e, t]\mapsto X^{x,\alpha, \varepsilon}(\tau)$,
there exists $0< \e t_\e < \tau_\e < t$ such that
$X^{x,\alpha, \varepsilon}(\tau_\e)\in \{x_1=x_2\}$, where
$\{x_1=x_2\}$ is the line separating the half-spaces $\{ x_1 > x_2\}$ and $\{ x_1 < x_2\}$. 
Writing~\eqref{ineq123} for $t_1=\e t_\e \leq t_2=\tau_\e < t$, we have
\begin{eqnarray*}
  0\leq d(X^{x,\alpha, \varepsilon}(\tau_\e))\leq d(X^{x,\alpha, \varepsilon}(\e t_\e)) +\frac{|f|_\infty^2}{2}\varepsilon (\tau_\e -\e t_\e)
   \mathop{\to}_{\e\to 0} 0.
\end{eqnarray*}
Since $\{ x_1 = x_2\}\cap \{d=0\} =\{O\}$, we obtain that $X^{x,\alpha, \varepsilon}(\tau_\e)\to O$
as $\e\to 0$.
Now, writing~\eqref{ineg592} with $0<t_1=\tau_\e < t_2=t\leq T$, we obtain
\begin{eqnarray*}
  \left| |X^{x,\alpha, \varepsilon}(t)|^2- |X^{x,\alpha, \varepsilon}(\tau_\e)|^2 \right|
 &\leq& C(t-\tau_\e)  + 2(d^{\frac{1}{2}}(X^{x,\alpha, \varepsilon}(\tau_\e))-d^{\frac{1}{2}}(X^{x,\alpha, \varepsilon}(t))),
\end{eqnarray*}
which yields, when $\e\to 0$, $|X^{x,\alpha}(t)|^2\leq Ct$.
All these computations being true for every $0<t<t_\eta$, we conclude that $X^{x,\alpha}(t)\to O$
as $t\to 0^+$, which is in contradiction with $\hat{x}\not= 0$.
This completes the proof of Lemma \ref{saut-sur-reseau}.
\end{proof}

We end this section with the proof of Theorem~\ref{cv-traj}.

\begin{proof}[Proof of Theorem~\ref{cv-traj}]
Parts (1),(2),(3) are straightforward consequence of Lemmas~\ref{lem-exis},~\ref{compactnesslemma}
and~\eqref{saut-sur-reseau}.
The proof of (4) is postponed to the end of Section~\ref{sec:proof-dyn-interieure}.
\end{proof}

\section{Qualitative properties of the trajectories $X^{x,\alpha}$}
\label{sec:sko}

\subsection{Dynamic of trajectories inside the branches}
\label{sec:proof-dyn-interieure}

We start with the proof of Theorem~\ref{edo-interieur}, which gives a characterization
of the trajectories when they are inside the edges.

\begin{proof}[Proof of Theorem~\ref{edo-interieur}] Using the notations~\eqref{junctionandedges},
let $n$ be such that for all $t\in (a_{n}, b_{n})$, $X^{x,\alpha}(t)\in (0,\infty)e_{i({n})}$
and consider $X^{x,\alpha, \varepsilon}\to X^{x,\alpha}$ as $\varepsilon\to 0$.
From~\eqref{traj-eps}-\eqref{def-k}, for all $a_n < a \leq t < b_n$, we have
\begin{eqnarray*}
  && X^{x,\alpha, \varepsilon}(t) =  X^{x,\alpha, \varepsilon}(a)+ \int_a^t f(X^{x,\alpha, \varepsilon}(s),\alpha(s))ds
  -(k^{x,\alpha, \varepsilon}(t) - k^{x,\alpha, \varepsilon}(a)).
\end{eqnarray*}
Writing $e:=e_{i(n)}$ for simplicity, it follows that
\begin{eqnarray}\label{scal-e}
  && \langle X^{x,\alpha, \varepsilon}(t), e\rangle
  =  \langle X^{x,\alpha, \varepsilon}(a), e\rangle
  + \int_a^t  \langle f(X^{x,\alpha, \varepsilon}(s),\alpha(s)), e\rangle ds\\ \nonumber
  && \hspace*{6cm}- \langle k^{x,\alpha, \varepsilon}(t) - k^{x,\alpha, \varepsilon}(a), e\rangle.
\end{eqnarray}
Using the uniform convergence $X^{x,\alpha, \varepsilon}\to X^{x,\alpha}$ in $[a,b_n]$ as $\varepsilon\to 0$,
we easily obtain the convergence of the first three terms above. It remains to deal
with
\begin{eqnarray}\label{terme-en-k}
  &&
  \langle k^{x,\alpha, \varepsilon}(t) - k^{x,\alpha, \varepsilon}(a), e\rangle
  = \int_a^t   \langle \dot{k}^{x,\alpha, \varepsilon}(s), e\rangle ds
  =  \int_a^t \frac{1}{\varepsilon}  \langle \nabla d(X^{x,\alpha, \varepsilon}(s)), e\rangle ds.
\end{eqnarray}
For this, we use the explicit form of $d$ (see~\eqref{choix-d}), which leads to 
\[
\begin{array}{lll}
  \langle \nabla d(x), e\rangle \!= \!2\sqrt{d(x)}  \langle x, e^\perp\rangle {\mbox{ and }  \langle \nabla d(x), e^\perp\rangle  \!= \! 2\sqrt{d(x)}  \langle x, e\rangle, \mbox{ for }  i(n)\in\{E, W\},} 
\\*[.5em]
\langle \nabla d(x), e\rangle  \!= \! -2\sqrt{d(x)}  \langle x, e^\perp\rangle {\mbox{ and }  \langle \nabla d(x), e^\perp\rangle  \!= \! -2\sqrt{d(x)}  \langle x, e\rangle, \mbox{ for } i(n)\in\{N, S\},}
\end{array}
\]
and therefore,
\begin{eqnarray*}
  \langle \dot{k}^{x,\alpha, \varepsilon}(s), e\rangle
  = \frac{ \langle X^{x,\alpha, \varepsilon}(s), e^\perp\rangle}{ \langle X^{x,\alpha, \varepsilon}(s), e\rangle} \langle \dot{k}^{x,\alpha, \varepsilon}(s), e^\perp\rangle.
\end{eqnarray*}
Note that, since $X^{x,\alpha, \varepsilon}\to X^{x,\alpha}$ as $\varepsilon\to 0$, and  $X^{x,\alpha}(s)\neq O$ for all $s\in (a_{n}, b_{n})$, we have $\langle X^{x,\alpha, \varepsilon}(s), e_{i(n)}\rangle\neq0$, for $\varepsilon>0$ small enough.
Recalling that $k^{x,\alpha, \varepsilon}$ is bounded in $H^1_{\rm loc}(0,\infty)$, from Cauchy-Schwarz inequality, we get
\begin{eqnarray*}
  &&
  \int_a^t   \langle \dot{k}^{x,\alpha, \varepsilon}(s), e\rangle ds
  \leq \left( \int_a^t  \frac{ \langle X^{x,\alpha, \varepsilon}(s), e^\perp\rangle^2}{ \langle X^{x,\alpha, \varepsilon}(s), e\rangle^2}  ds \right)^{1/2}
   \|k^{x,\alpha, \varepsilon}\|_{H^1 [a,b_n]}.
\end{eqnarray*}
But, since $X^{x,\alpha}\in (0,\infty)e$ in $[a,t]$, we have
$\langle X^{x,\alpha, \varepsilon}(s), e^\perp\rangle^2\to 0$ and
$\langle X^{x,\alpha, \varepsilon}(s), e\rangle^2\to  \langle X^{x,\alpha}(s), e\rangle^2\not= 0$,
uniformly on $[a,t]$. It follows that the sequence in~\eqref{terme-en-k} converges to~$0$, as $\varepsilon\to 0$.

Thus, passing to the limit  $\varepsilon\to 0$ in~\eqref{scal-e}, we get
\begin{eqnarray*}
  && \langle X^{x,\alpha}(t), e\rangle
  =  \langle X^{x,\alpha}(a), e\rangle
  + \int_a^t  \langle f(X^{x,\alpha}(s),\alpha(s)), e\rangle ds.
\end{eqnarray*}
Notice that, as a  by-product, we obtain
\begin{eqnarray*}
\langle k^{x,\alpha}(t), e\rangle= \langle k^{x,\alpha}(a), e\rangle \qquad
\text{for all $t\in [a,b_n)$}.
\end{eqnarray*}

We now turn to the limit of the $e^\perp$-component. We have,
\begin{eqnarray}\label{scal-eperp}
  && \langle X^{x,\alpha, \varepsilon}(t), e^\perp\rangle
  =  \langle X^{x,\alpha, \varepsilon}(a), e^\perp\rangle
  + \int_{a}^t  \langle f(X^{x,\alpha, \varepsilon}(s),\alpha(s)), e^\perp\rangle ds\\ \nonumber
  && \hspace*{7cm}
  - \langle k^{x,\alpha, \varepsilon}(t) - k^{x,\alpha, \varepsilon}(a), e^\perp\rangle.
\end{eqnarray}
By the uniform convergence
$(X^{x,\alpha, \varepsilon},  k^{x,\alpha, \varepsilon})\to (X^{x,\alpha},  k^{x,\alpha})$
on $[a,b_n]$ with $ \langle X^{x,\alpha, \varepsilon},  e^\perp\rangle  \to 0$, we obtain
\begin{eqnarray*}
&&  \langle k^{x,\alpha}(t),  e^\perp\rangle =
  \langle  k^{x,\alpha}(a), e^\perp\rangle
  + \int_{a}^t  \langle f(X^{x,\alpha}(s),\alpha(s)), e^\perp\rangle ds.
\end{eqnarray*}

Collecting the previous results we arrive at
\begin{eqnarray*}
&& X^{x,\alpha}(t) =  X^{x,\alpha}(a) + \int_{a}^t  \langle f(X^{x,\alpha}(s),\alpha(s)), e\rangle e\, ds,\\
&&  k^{x,\alpha}(t)=  k^{x,\alpha}(a) + \int_{a}^t  \langle f(X^{x,\alpha}(s),\alpha(s)), e^\perp\rangle e^\perp ds,
\end{eqnarray*}
for all $a_n<a\leq t< b_n$.
Recalling that
$X^{x,\alpha}(0^+)=\overline{x}$ and $k^{x,\alpha}(0^+)=x-\overline{x}$, this proves the result in the case $a_n=0$. When $a_n\not= 0$, the result follows from the continuity of~$(X^{x,\alpha},  k^{x,\alpha})$ on~$(0,\infty)$.
\end{proof}  

\begin{proof}[Proof of Theorem~\ref{cv-traj} (4)]
From Theorem~\ref{cv-traj} (2), we know that $X^{x,\alpha}, k^{x,\alpha}\in H_{\rm loc}^1(0,\infty)$,
thus they are differentiable almost everywhere on~$(0,\infty)$.

From~\eqref{ppte-skorokhod}, it is enough to prove~\eqref{estim-derivees}
for $X^{x,\alpha}$ and then, deduce the result for~$k^{x,\alpha}$.
Below we use the time partition~\eqref{junctionandedges} associated
with  $X^{x,\alpha}$.

From Theorem~\ref{edo-interieur}, for a.e. $t\in (a_n, b_n)$, we have
$\dot{X}^{x,\alpha}(t) = \langle f({X}^{x,\alpha}(t), \alpha(t)), e_{i(n)}\rangle e_{i(n)}$,
which yields~\eqref{estim-derivees}.

It remains to deal with the case $t\in \mathcal{Z}(X^{x,\alpha})$.
We argue by contradiction assuming  $X^{x,\alpha}$ is differentiable at $t$
with $|\dot{X}^{x,\alpha}(t)|\geq |f|_\infty+2\eta$ for some $\eta >0$. On the one hand,
from the Taylor expansion $X^{x,\alpha}(t+h)=X^{x,\alpha}(t)+ h\dot{X}^{x,\alpha}(t) +o(h)$,
we deduce the existence of $h_0>0$ such that,
for all $h\in (0,h_0)$, 
$|o(h)|\leq \eta h$  and
\begin{eqnarray}\label{minordotX}
&& |X^{x,\alpha}(t+h)|\geq (|f|_\infty+\eta)h.
\end{eqnarray}
This means that $(0,h_0)\subset (a_n,b_n)$ for some $n$ with $a_n=t$.
On the other hand, by Theorem~\ref{edo-interieur}, we have
\begin{eqnarray*}
  && |X^{x,\alpha}(t+h)| = \left|\int_t^{t+h}  \langle f({X}^{x,\alpha}(t), \alpha(t)), e_{i(n)}\rangle e_{i(n)}ds  \right|
  \leq |f|_\infty h,
\end{eqnarray*}
which is in contradiction with~\eqref{minordotX}.
Eventually, we conclude~\eqref{estim-derivees} for all differentiability points
$t\in\mathcal{Z}(X^{x,\alpha})$, \emph{i.e.}, a.e. in $\mathcal{Z}(X^{x,\alpha})$. This ends the proof of Theorem~\ref{cv-traj}~(4).
\end{proof}

\subsection{Behavior of trajectories with locally constant control near the origin}\label{effetzenon}

When a trajectory reaches $O$, it is delicate to determine what happens later
(how long does the trajectory remain at $O$? in which branch does it enters? etc.).
Before giving a complete description of the behavior for locally constant controls,
we first give a pathological example to illustrate the difficulty.
More precisely,
we construct a control~$\alpha\in \mathcal{A}_O$ such that one cannot
determine in which branch the trajectory enters.

\begin{lemma}[A trajectory visiting all branches instantaneously]
\label{lem:traj-visite-tout}
There exists a control $\alpha\in \mathcal{A}_O$ such that
$\omega(O,\alpha)=\{X^{O,\alpha}\}$  and, for all $\delta >0$ and $i\in \{E,N,W,S\}$,
there exists $t\in (0,\delta)$ such that $X^{O,\alpha}(t)\in (0,\infty) e_i$.
\end{lemma}

\begin{proof}
Let $(i(n))_{n\in\N}$ be any sequence such that $i(n)\in \{E,N,W,S\}$
and define the measurable function $\beta_n : [0,1]\to \overline{B}(O,1)$,
\begin{eqnarray*}
  && \beta_n(t)=
  \left\{
  \begin{array}{cl}
    0 & \text{if $t\in [0,\frac{1}{2^{n+1}}]$,}\\[1mm]
    e_{i(n)} &   \text{if $t\in (\frac{1}{2^{n+1}}, \frac{1}{2^{n+1}}+\frac{1}{2^{n+2}} )$,}\\[1mm]
    -e_{i(n)} &   \text{if $t\in [\frac{1}{2^{n+1}}+\frac{1}{2^{n+2}}, \frac{1}{2^{n}})$,}\\[1mm]
    0 &   \text{if $t\in [\frac{1}{2^{n}}, 1]$.}
  \end{array}
  \right.
\end{eqnarray*}
Setting $f(x,\alpha):=\alpha$, we notice that $f$ satisfies~\eqref{hyp-f-ell}.
Moreover, the solution $X^{O,\beta_n,\varepsilon}$ of~\eqref{traj-eps} is exactly
the solution of the ODE $\dot{X}(t)=\beta_n(t)$, $X(0)=O$, that is,
\begin{eqnarray*}
  && X^{O,\beta_n,\varepsilon}(t)=
  \left\{
  \begin{array}{cl}
    O & \text{if $t\in [0,\frac{1}{2^{n+1}}]$,}\\[1mm]
    (t-\frac{1}{2^{n+1}}) e_{i(n)} &   \text{if $t\in (\frac{1}{2^{n+1}}, \frac{1}{2^{n+1}}+\frac{1}{2^{n+2}} )$,}\\[1mm]
    (\frac{1}{2^{n}}-t) e_{i(n)} &   \text{if $t\in [\frac{1}{2^{n+1}}+\frac{1}{2^{n+2}}, \frac{1}{2^{n}})$,}\\[1mm]
    O &   \text{if $t\in [\frac{1}{2^{n}}, 1]$.}
  \end{array}
  \right.
\end{eqnarray*}
Notice that $\beta_n$ is designed so that $X^{O,\beta_n,\varepsilon}(t)\in \Gamma$ for all $t\in [0,1]$.
In particular, the penalization in~\eqref{traj-eps} is always 0 and $X^{O,\beta_n,\varepsilon}\equiv X^{O,\beta_n}$.

Next, we set $\alpha_n:=\sum_{k=0}^n \beta_k$ and we consider the solution $X^{O,\alpha_n,\varepsilon}$
of~\eqref{traj-eps}. 
Since~$\beta_k\beta_m=0$ for $k\not= m$ and $X^{O,\beta_n,\varepsilon}(2^{-n-1})=X^{O,\beta_n,\varepsilon}(2^{-n})=O$, we have
$$
X^{O,\alpha_n,\varepsilon} = \sum_{k=0}^n X^{O,\beta_k,\varepsilon}=  \sum_{k=0}^n X^{O,\beta_k} = X^{O,\alpha_n}.
$$

Sending $n$ to $\infty$, we obtain a measurable function $\alpha:= \sum_{k=0}^\infty \beta_k\in \mathcal{A}_O$
such that the solution  $X^{O,\alpha,\varepsilon}$
of~\eqref{traj-eps} is exactly
$$
X^{O,\alpha} = \sum_{k=0}^\infty X^{O,\beta_k},
$$
which does not depend on $\varepsilon$. Thus, $\omega(O,\alpha)=\{X^{O,\alpha}\}$.
Moreover, by construction, $X^{O,\alpha}(2^{-n})=O$ for all $n$ and
$X^{O,\alpha}(t)\in (0,\infty) e_{i(n)}$ for $t\in (2^{-n-1}, 2^{-n})$.
Choosing, for instance, the periodic sequence $E,N,W,S,E,N,W,S,\cdots$, we obtain
the result of Lemma \ref{lem:traj-visite-tout}.
\end{proof}  

The aim of the following results is to establish the behavior of the trajectories
passing through the origin $O$ with constant controls.

We introduce the unit vector $e_\theta= (\cos\theta, \sin\theta)$
(note that $e_E= e_0$, $e_N= e_{\pi/2}$, etc.) and consider
a trajectory $t\mapsto X^{x,\alpha}(t)$ such that
\begin{eqnarray}\label{traj_hitO}
X^{x,\alpha}(t_0)=O,
\end{eqnarray}
with $\alpha(t)\equiv e_\theta$ in some interval $[t_0-\tau, t_0+\tau]$.

We want to describe all possible situations.
We distinguish the case when the trajectory starts at the origin, $x=O$,
from the case when it starts outside the origin. In this latter case, up to scalings
(Lemma~\ref{lem-scaling}) and
permutations, we may assume~$x=(0,1)=e_N$.
Next, since  we are interested in a single passage at $O$, we can assume
without loss of generality that $\alpha(t)=e_\theta$ for all $t\geq 0$.

\begin{proposition}[Trajectories starting from $O$] \label{behav-O}
Let $\alpha(t)\equiv e_\theta$ for $\theta\in [0,2\pi]$ and $x=O$.
Then, $\omega (O,e_\theta)=\{ X^{O,e_\theta}\}$ and: 
\begin{enumerate}
\item (Non bisector case,
  $\theta\not\in \{\frac{\pi}{4}, \frac{3\pi}{4} ,  \frac{5\pi}{4} ,  \frac{7\pi}{4}\}$)
If  there exists $i\in \{E,N,W,S\}$ such that
  $\langle e_\theta, e_i\rangle > \max_{j\not= i}\langle e_\theta, e_j\rangle$,
then, the trajectory enters the branch $i$; that is,
$X^{O,e_\theta}(t)= \langle e_\theta, e_i\rangle t e_i$ for all $t\geq 0$.

\item (Bisector case, $\theta\in \{\frac{\pi}{4}, \frac{3\pi}{4} ,  \frac{5\pi}{4} ,  \frac{7\pi}{4}\}$)
If  there exist $i\not= j$ such that
$\langle e_\theta, e_i\rangle =\langle e_\theta, e_j\rangle$, then, the trajectory remains at $O$, \emph{i.e.},
$X^{O,e_\theta}(t)= O$  for all $t\geq 0$.

\end{enumerate}
\end{proposition}

\begin{proposition}[Trajectories starting inside a branch] \label{behav-1}
Let $\alpha(t)\equiv e_\theta$ for $\theta\in [0,2\pi]$ and $x=e_N$.
Then, $\omega (e_N , e_\theta)=\{ X^{e_N ,e_\theta}\}$ and: 
\begin{enumerate}
\item If $\theta\in [0,\pi]$, then, the trajectory never reaches $O$.

\item If $\theta\in (\pi,  \frac{5\pi}{4}]$, then, the trajectory
    enters the branch $W$ after passing through~$O$, \emph{i.e.},
  \begin{eqnarray}\label{traj-W}
  && X^{e_N ,e_\theta}(t)=
  \left\{
  \begin{array}{cl}
    \big(0,1 +(\sin\theta) t\big) & \text{if \ $t\in [0, (-\sin\theta)^{-1}]$,}\\*[.6em]
   \left(\frac{\cos\theta}{\sin\theta}+(\cos\theta)t, 0\right) &   \text{if \ $t\in [(-\sin\theta)^{-1},\infty)$.}
  \end{array}
  \right.
  \end{eqnarray}    

\item If $\theta\in (\frac{5\pi}{4}, \frac{7\pi}{4})$, then the trajectory
  continues into the branch $S$  after passing through $O$, \emph{i.e.},
 $X^{e_N ,e_\theta}(t)= \big(0,1 +(\sin\theta) t\big)$ for all $t\geq 0$.

\item If $\theta\in [\frac{7\pi}{4}, 2\pi)$, then the trajectory
    enters the branch $E$ after passing through $O$ (symmetric to the second case). 
  
\end{enumerate}
\end{proposition}


The proofs of these propositions are postponed to the Appendix. The proof
of Proposition~\ref{behav-O} is easy, whereas that of Proposition~\ref{behav-1}
is involved and very technical. The difficult (and non-intuitive) behavior is the one of
part (2).

As a consequence, we obtain some pathological examples showing that
the limit trajectories $X^{x,\alpha}$ do not satisfy the semigroup property and
are not stable.

\begin{corollary}\label{cor-instable}
  Let $x\in\Gamma$ and $\alpha\in \mathcal{A}_x$. 

\begin{enumerate}
\item  (Failure of the semigroup property) We may have $X^{x,\alpha}\in \omega(x,\alpha)$, $X^{X^{x,\alpha}(t),\alpha (t+\cdot)}\in \omega(X^{x,\alpha}(t),\alpha (t+\cdot))$ for some $t>0$,
  and  $X^{x,\alpha}(t+s)\not= X^{X^{x,\alpha}(t),\alpha (t+\cdot)}(s)$.

\item (Instability) It can happen that $x_n\to x$ and $X^{x_n,\alpha}$ converges locally uniformly
  in $(0,\infty)$ to some $Y\not= X^{x,\alpha}$.

\end{enumerate}
  
\end{corollary}

\begin{proof}[Proof of Corollary~\ref{cor-instable}]
  Take $\alpha (t)\equiv e_{5\pi/4}$ and $x=e_N$. Then, $X^{x,\alpha}(t)$ is given
  by~\eqref{traj-W}. In particular, for $s\geq 0$, 
 \[
  X^{x,\alpha}(\sqrt{2}+s)= (-{2}^{-1/2} s,0)
  \not= O\equiv  X^{O, \alpha}(s)= X^{X^{x,\alpha}(\sqrt{2}),\alpha (\sqrt{2}+\cdot)}(s).
  \]
  Replacing the starting point $x=e_N$ above with $x_n=\frac{1}{n}e_N\to O$ as $n\to \infty$, we obtain
  a sequence of trajectories $X^{x_n,\alpha}$ still going to the $W$ branch whereas
  $X^{O,\alpha}\equiv O$ is stuck at the origin.
\end{proof}
  
\section{Convergence of the approximated control problem
to the control problem on $\Gamma$}\label{sec:control}

This section is devoted to the proof of Theorem~\ref{limit-vf}.
We recall that in this section, we assume~\eqref{calc-var}-\eqref{cout-indep}, \emph{i.e.},
\begin{eqnarray*}
  && f(x,a)= a \in A=\overline{B}(0,1) \quad \text{for all $x\in\R^2$, $a\in A$,}\\
  && \ell (x,a)= \ell (x).
\end{eqnarray*}
In particular, $|f|_\infty=1$.

We introduce some notation. For $x\in\R^2$ and $\alpha\in \mathcal{A}=L^\infty([0,\infty), \overline{B}(0,1))$, set
\begin{eqnarray*}
&&  J^\e (x,\alpha):= \int_0^\infty e^{-\lambda t} \ell (X^{x,\alpha,\e}(t))dt, \quad \text{for all $\e>0$,}
\end{eqnarray*}
and
\begin{eqnarray*}
&&  J (x,\alpha, X^{x,\alpha}):= \int_0^\infty e^{-\lambda t} \ell (X^{x,\alpha}(t))dt,
\quad \text{for all $X^{x,\alpha}\in \omega(x,\alpha)$.}
\end{eqnarray*}
With these notations,
$\displaystyle V^\e(x)=\inf_{\alpha\in\mathcal{A}} J^\e (x,\alpha)$ and
$\displaystyle \overline{V}(x)=\inf_{\alpha\in\mathcal{A}, \, X^{x,\alpha}\in \omega(x,\alpha)} J (x,\alpha, X^{x,\alpha})$.

Let us start with the following technical lemma.

\begin{lemma}\label{techn1}
Assume~\eqref{calc-var}. 
Let $x\in\R^2$ and $\overline{x}=\phi_{d_\Gamma}(x)\in\Gamma$
(see~\eqref{x-proj}). Then,

\begin{enumerate}

\item If $\alpha\in \mathcal{A}_{\overline{x}}$ (see~\eqref{def-Ax}), then
$X^{\overline{x},\alpha,\e}=: X^{\overline{x},\alpha}$ does not depend on $\e$, and thus~$\omega( \overline{x}, \alpha)=\{ X^{\overline{x},\alpha}\}$.

\item If $\alpha\in \mathcal{A}$ and
$X^{x,\alpha}\in\omega (x,\alpha)$, then 
$\overline{\alpha}:= (\alpha -\dot{k}^{x,\alpha})\in \mathcal{A}_{\overline{x}}$, $\omega( \overline{x}, \overline{\alpha})=\{ X^{\overline{x},\overline{\alpha}}\}$
with $X^{\overline{x},\overline{\alpha}}(0)=\overline{x}$ and $X^{\overline{x},\overline{\alpha}}(t)=X^{x,\alpha}(t)$
for $t>0$.


\item (Controllability on $\Gamma$) For all $\overline{y}\in \Gamma$ and $\alpha\in \mathcal{A}$,
there exists $\alpha_{\overline{y}}\in \mathcal{A}$
and $0\leq\tau\leq C|x-\overline{y}|$ such that,
for all $X^{\overline{x},\alpha_{\overline{y}}}\in \omega (\overline{x},\alpha_{\overline{y}})$,
$X^{\overline{x},\alpha_{\overline{y}}}\in\Gamma$ on $[0,\tau]$ with $X^{\overline{x},\alpha_{\overline{y}}}(\tau)=\overline{y}$,
and there exists $X^{\overline{y},\alpha}\in \omega (\overline{y},\alpha)$ satisfying
$X^{\overline{x},\alpha_{\overline{y}}}(t)= X^{\overline{y},\alpha}(t-\tau)$ for $t\geq \tau$.
  
\end{enumerate}
\end{lemma}

\begin{remark}
The fact that $\mathcal{A}_{\overline{x}}\not=\emptyset$ is proved in
~\cite[Theorem 2.3]{aot15} the assumptions of which
are satisfied due to~\eqref{calc-var}. The proof of Statement (2) above provides a
concrete way to build such a control~$\overline{\alpha}$. It is used to prove that $\overline{V}\geq V_\Gamma$
on $\Gamma$ in the proof of Proposition~\ref{barV-Vgamma}.
\end{remark}
  
\begin{proof}[Proof of Lemma~\ref{techn1}]\ \\
(1)
If  $\alpha\in \mathcal{A}_{\overline{x}}$, then, by definition, there exists a unique solution
$X^{\overline{x},\alpha}$ of~\eqref{traj}. Since $X^{\overline{x},\alpha}(t)\in\Gamma$ for all $t\geq 0$,
$X^{\overline{x},\alpha}$ is also the unique solution of~\eqref{traj-eps} for all $\e>0$, 
which gives the conclusion.\\
(2)
From Theorem~\ref{cv-traj} and \eqref{calc-var}, any $X^{x,\alpha}\in\omega (x,\alpha)$ is a solution
to $\dot{X}^{x,\alpha}=\overline{\alpha}\in \overline{B}(0,1)$ a.e. $t\in (0,\infty)$, with $X^{x,\alpha}(t)\in \Gamma$
for $t>0$ and $X^{x,\alpha}(0^+)=\overline{x}$. Setting $Y(t):= X^{x,\alpha}(t)$ for $t>0$ and
$Y(0)= \overline{x}$, one has $Y\in\omega( \overline{x}, \overline{\alpha})$. By~(1), we infer
that $Y= X^{\overline{x},\overline{\alpha}}$.\\
(3)
Let $\overline{x},\overline{y}\in \Gamma$, with $\overline{x}\neq\overline{y}$.
Assume that $\overline{x}\in [0,\infty)e_i$ and $\overline{y}\in [0,\infty)e_j$.
We set $\tau:= |\overline{x}-\overline{y}|$ if $i=j$ and $\tau:= |\overline{x}| +|\overline{y}|$ if $i\not= j$. 
Note that $\tau$ is nothing but the geodesic distance between  $\overline{x}$ and $\overline{y}$,
which is equivalent to the Euclidean distance in $\R^2$. In particular, there exists $C>0$ such that $\tau\leq C|\overline{x}-\overline{y}|$
(in our simple case, $C=\sqrt{2}$).
For any $\alpha\in \mathcal{A}$,
we define $\alpha_{\overline{y}}$ in the following way: 
if $i=j$, we set $\alpha_{\overline{y}}(t):= |\overline{y}-\overline{x}|^{-1}(\overline{y}-\overline{x})$
for $t\in [0,\tau]$, otherwise, we set
$\alpha_{\overline{y}}(t):= -|\overline{x}|^{-1}\overline{x}$
for $t\in [0,|\overline{x}|]$ and $\alpha_{\overline{y}}(t)= |\overline{y}|^{-1}\overline{y}$
for $t\in [|\overline{x}|, \tau]$. For $t\geq \tau$, in both case we set $\alpha_{\overline{y}}(t)= \alpha(t-\tau )$.
Let  $X^{\overline{x},\alpha_{\overline{y}}}\in \omega (\overline{x},\alpha_{\overline{y}})$ and $\e_k\to 0$ such that
$X^{\overline{x},\alpha_{\overline{y}},\e_k}\to X^{\overline{x},\alpha_{\overline{y}}}$ locally uniformly on $[0,\infty)$.
By construction of $\alpha_{\overline{y}}$, we have, for all $k$,
$X^{\overline{x},\alpha_{\overline{y}}, \e_k}= X^{\overline{x},\alpha_{\overline{y}}}$ on $[0,\tau]$
with $X^{\overline{x},\alpha_{\overline{y}}, \e_k}(\tau)= X^{\overline{x},\alpha_{\overline{y}}}(\tau)=\overline{y}$.
Moreover, by the semigroup property for~\eqref{traj-eps},
we have $X^{\overline{x},\alpha_{\overline{y}},\e_k}(t)=X^{\overline{y},\alpha,\e_k}(t-\tau)$
for $t\geq\tau$. Using Theorem~\ref{cv-traj}, it follows that there exists
$X^{\overline{y},\alpha}\in \omega(\overline{y},\alpha)$ such that
$X^{\overline{x},\alpha_{\overline{y}}}(t)=X^{\overline{y},\alpha}(t-\tau)$
for $t\geq\tau$.
\end{proof}  

Two first results about the value functions can be deduced from Lemma~\ref{techn1}.

\begin{proposition}\label{barV-Vgamma}
Assume~\eqref{hyp-f-ell},~\eqref{calc-var}-\eqref{cout-indep}.
Then $\overline{V}= V_\Gamma$ on $\Gamma$.
\end{proposition}

\begin{proof}[Proof of Proposition~\ref{barV-Vgamma}]
Let $x\in\R^2$ and $\eta >0$. There exists $\alpha_\eta\in\mathcal{A}$
and $X^{x,\alpha_{\eta}}\in \omega (x,\alpha_{\eta})$, which
are $\eta$-optimal for $\overline{V}(x)$, \emph{i.e.},
\begin{eqnarray*}
&& \overline{V}(x)+\eta
\geq J(x,\alpha_\eta, X^{x,\alpha_{\eta}}).
\end{eqnarray*}
Applying Lemma~\ref{techn1}\,(2), we obtain that
$X^{x,\alpha_{\eta}}=X^{\overline{x},\overline{\alpha}_{\eta}}$
with $\overline{\alpha}_{\eta}=\alpha_{\eta}-\dot{k}^{x,\alpha_{\eta}}\in \mathcal{A}_{\overline{x}}$.
It follows that
\begin{eqnarray*}
J(x,\alpha_\eta, X^{x,\alpha_{\eta}})=J(x,\alpha_\eta, X^{\overline{x},\overline{\alpha}_{\eta}})
= J(x,\overline{\alpha}_\eta, X^{\overline{x},\overline{\alpha}_{\eta}}) \geq V_\Gamma (\overline{x}),
\end{eqnarray*}
where the second equality comes from the
fact that
$\ell$ does not depend directly on the control by~\eqref{cout-indep},
and the inequality comes from the definition of $V_\Gamma$, see~\eqref{vf-gamma}.
Collecting the previous inequalities and using the fact that $\eta >0$ is arbitrary,
we conclude that~$\overline{V}(x)\geq V_\Gamma (\overline{x})$.

We turn to the opposite inequality. Let $\overline{x}\in\Gamma$ and
$\alpha\in \mathcal{A}_{\overline{x}}$.
By Lemma~\ref{techn1}\,(1), we have $\omega(\overline{x}, \alpha)=\{X^{\overline{x},\alpha}\}$
and it follows that 
$\overline{V}(\overline{x})\leq J(\overline{x}, \alpha, X^{\overline{x},\alpha})$.
Taking the infimum over $\alpha\in \mathcal{A}_{\overline{x}}$, we obtain
$\overline{V}(\overline{x})\leq V_\Gamma (\overline{x})$ as desired.
\end{proof}  

\begin{example}[Counterexample to $\overline{V}=V_\Gamma$ when~\eqref{cout-indep} does not hold]
\label{contre-ex-la}
If the instantaneous cost $\ell$ depends directly on the control~$a$, then one may have $\overline{V}(x)<V_\Gamma(x)$.

Consider the case where $\ell$ is defined by
\[
\ell (x,a)= 2+a_1+a_2+|x_2|,\quad \text{for }  x=(x_1,x_2)\in\R^2,  a=(a_1,a_2)\in\overline{B}(0,1).
\]
Note that $\ell$ is not bounded as required in~\eqref{hyp-f-ell}. We choose not to truncate $\ell$ to simplify
the computations.
We estimate the value functions for $x=e_N=(0,1)$.

First, choosing $\alpha(t)\equiv e_{\frac{5\pi}{4}}$, from 
Proposition~\ref{behav-1} (2) one has $\omega (e_N , \alpha)=\{ X^{e_N , e_{\frac{5\pi}{4}}}\}$ with
\begin{eqnarray*}
  && X^{e_N ,  e_{\frac{5\pi}{4}}  }(t)=
  \left\{
  \begin{array}{cl}
    \big(0,1 -t/\sqrt{2}\big) & \text{if \ $t\in [0, \sqrt{2}]$,}\\
   \left(1 -t/\sqrt{2}, 0\right) &   \text{if \ $t\geq\sqrt{2}$,}
  \end{array}
  \right.
  \end{eqnarray*}    
and a straightforward computation leads to
\begin{eqnarray*}
  &&   \overline{V}(e_N)\leq J(e_N, \alpha, X^{e_N , \alpha})= \frac{\lambda\sqrt{2}(3-\sqrt{2})-1
    +e^{-\lambda\sqrt{2}}}{\lambda^2\sqrt{2}}.
\end{eqnarray*}

Now, for any control $\beta\in \mathcal{A}_{e_N}$,
by Lemma~\ref{techn1}, $\omega (e_N,\beta)= \{X^{e_N,\beta}\}$ with $X^{e_N,\beta}(t)=e_N+\int_0^t\beta(s)ds\in\Gamma$,
and $\beta_1(t)\beta_2(t)=0$ a.e. $t\geq0$. Moreover, since~$\beta(t)\in\overline{B}(0,1)$ a.e. $t\geq0$,
on the one hand we have
$$
\ell (X^{e_N,\beta}(t), \beta(t))=2+\beta_1(t)+\beta_2(t)+|X_2^{e_N,\beta}(t)|\geq 1+|X_2^{e_N,\beta}(t)|, \quad \text{for all $t\geq 0.$}
$$
On the other hand,
$\inf\{t\geq0 :X^{e_N,\beta}=O\}\geq 1$ and $|X^{e_N,\beta}_2|\geq1-t$, for all $t\in[0,1]$.
It follows that
$$
J(e_N,\beta, X^{e_N,\beta})\geq
\int_0^1 e^{-\lambda t}(2-t)dt +\int_1^\infty  e^{-\lambda t}dt
= \frac{2\lambda-1+e^{-\lambda}}{\lambda^2}.
$$
Taking the infimum with respect to $\beta\in \mathcal{A}_{e_N}$, we finally obtain the desired strict inequality
$$
\overline{V}(e_N)\leq  \frac{\lambda\sqrt{2}(3-\sqrt{2})-1+e^{-\lambda\sqrt{2}}}{\lambda^2\sqrt{2}} <
 \frac{2\lambda-1+e^{-\lambda}}{\lambda^2} \leq V_\Gamma (e_N).
 $$

\end{example}


\begin{lemma}\label{Veps-lip-gamma}
Assume~\eqref{hyp-f-ell}. Then, $V^\e\in C(\R^2)$ for each $\e >0$ and
$|V^\e|\leq \lambda^{-1}M$. If, in addition, we assume~\eqref{calc-var}, then, $V^\e$ is
uniformly Lipschitz on $\Gamma$, \emph{i.e.}, there exists $C>0$ independent of $\e$ such that
\begin{eqnarray*}
  &&  |V^\e (x)-V^\e (y)|\leq C|x-y| \quad \text{for all $x,y\in\Gamma$.}
\end{eqnarray*}
\end{lemma}

\begin{proof}[Proof of Lemma~\ref{Veps-lip-gamma}]
The first property is classical, see~\cite{barles94, bcd97}. The uniform boundedness
of $V^\e$ comes from the boundedness of $\ell$.

We turn to the proof of the uniform Lipschitz bound. For the sake of notations, we
prove the result under Assumption~\eqref{cout-indep}. Choose $x,y\in\Gamma$, any $\alpha\in \mathcal{A}$, and consider~$\alpha_y\in\mathcal{A}_x$ given by
Lemma~\ref{techn1}\,(3). We have
\begin{eqnarray*}
  V^\e (x) \leq  J^{\e}(x,\alpha_y)
  &=& \int_0^{\tau}e^{-\lambda t}\ell(X^{x,\alpha_y ,\e}(t))dt
  + e^{-\lambda \tau} \int_{0}^{\infty} e^{-\lambda t}\ell(X^{x,\alpha_y ,\e}(\tau+t))dt
\\
&=& \int_0^{\tau}e^{-\lambda t}\ell(X^{x,\alpha_y ,\e}(t))dt
+ e^{-\lambda \tau} \int_{0}^{\infty} e^{-\lambda t}\ell(X^{y,\alpha ,\e}(t))dt
\\
&\leq&
C|x-y|+ e^{-\lambda \tau} J^{\e}(y,\alpha).
\end{eqnarray*}
Taking the infimum over $\alpha\in \mathcal{A}$, we infer (for a positive constant $C$,
which may vary from line to line) that
\begin{eqnarray*}
V^\e (x)
&\leq& C|x-y|+e^{-\lambda \tau}  V^\e (y)\leq  C|x-y|+|1-e^{-\lambda \tau}| \lambda^{-1}M +  V^\e (y)\\
&\leq& C|x-y|+ V^\e (y).
\end{eqnarray*}
This ends the proof.
\end{proof}  

To obtain further properties of the value functions, we need the following crucial result which, roughly speaking, states that any trajectory can be closely followed by a trajectory remaining on the network $\Gamma$.

\begin{proposition}[Trajectory tracking]
\label{suivi-traj-1}
Assume~\eqref{hyp-f-ell}.
For all $x\in\R^2$, $\alpha\in \mathcal{A}$, $0< \gamma <1$
and $\e >0$ small enough,
there exists $\displaystyle \Gamma\ni\overline{x}_\e \mathop{\to}_{\e\to 0} \overline{x}$
and $\beta=\beta_{\alpha, \e}\in \mathcal{A}_{\overline{x}_\e}$ such that
\begin{eqnarray}\label{estim-traj-111}
  && |X^{x,\alpha,\e}(t)-X^{\overline{x}_\e,\beta,\e}(t)|\leq  C( \e^{\gamma/8} +  \e^{5\gamma/24}t)
  \quad \text{for all $t\geq C\e^{1-\gamma}$,}
\end{eqnarray}
where $C$ depends only on $x$ and $|f|_\infty$.
\end{proposition}

\begin{remark} 
Since the control $\beta$ keeps the trajectory on $\Gamma$,
we have $\omega (\overline{x}_\e, \beta)=\{X^{\overline{x}_\e,\beta}\}$
with $X^{\overline{x}_\e,\beta}=X^{\overline{x}_\e,\beta,\e}$. 
\end{remark}

\begin{proof}[Proof of Proposition~\ref{suivi-traj-1}]
We fix  $x\in\R^2$, $\alpha\in \mathcal{A}$,  $0< \gamma <1$,
$\e >0$,  $T\geq 4d(x)^{1/4}\e^{1-\gamma}$ (see~\eqref{reach-4-3}),
and aim at establishing~\eqref{estim-traj-111}. In the whole proof, $C$ is a positive constant which may change from line
to line, but is independent of $\alpha, \e$ and $T$.

Consider the following  subsets, see Figure~\ref{dess-suivi-traj}:
\begin{eqnarray*}
&& \mathcal{C}_O := \{ 0\leq x_1\leq \e^{\gamma/8}, \ 0\leq x_2\leq \e^{\gamma/8}, \
    d(x)\leq \e^{4\gamma/3}\},\\
&& \mathcal{C}_N := \{ x_1\geq 0, \ x_2>  \e^{\gamma/8}, \ d(x)\leq \e^{4\gamma/3}\},\\
&& \mathcal{C}_E := \{ x_1>  \e^{\gamma/8}, \ x_2\geq 0, \ d(x)\leq \e^{4\gamma/3}\},\\
&& \mathcal{C}_{NE} := \mathcal{C}_O \cup \mathcal{C}_N \cup  \mathcal{C}_E,  
\end{eqnarray*}
while $\mathcal{T}_O$ is the subset of times $t$ when $X^{x,\alpha,\e}(t)$ lies in $\mathcal{C}_O$,
and $\mathcal{T}_N$ (resp.  $\mathcal{T}_E$) is the open subset of times when the trajectory is in $\mathcal{C}_N$
(resp. $\mathcal{C}_E$). Note that all these subsets depend on $\alpha, \e$ (though we do not indicate
it for the sake of notation).

\begin{figure}[ht]
  \begin{center}
\includegraphics[width=6cm]{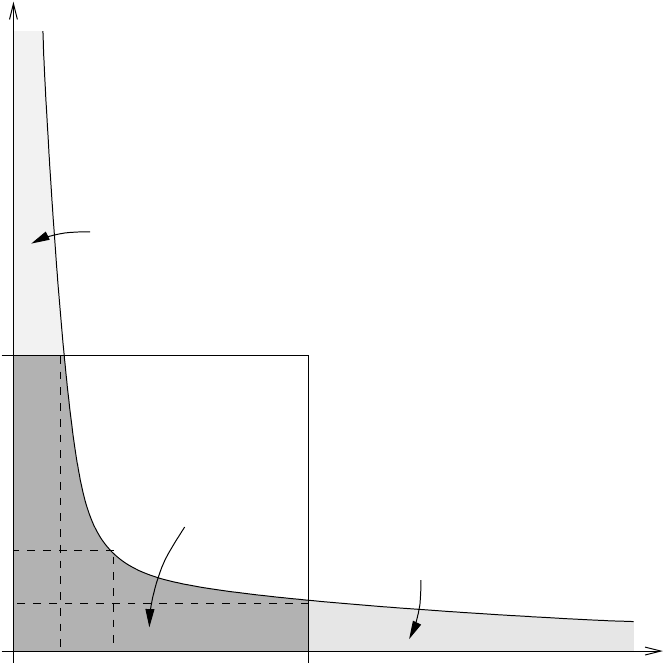}
\caption{}
\label{dess-suivi-traj}
\unitlength=1pt
\begin{picture}(0,0)
\put(-90,34){$\scriptstyle O$}
\put(87,42){$\scriptstyle x_1$}
\put(-83,211){$\scriptstyle x_2$}
\put(-58,148){$\mathcal{C}_N$}
\put(20,64){$\mathcal{C}_E$}
\put(-39,78){$\mathcal{C}_O$}
\put(-8,30){$\scriptstyle \e^{\gamma/8}$}
\put(-55,30){$\scriptstyle \e^{\gamma/3}$}
\put(-82,30){$\scriptstyle \e^{13\gamma\!/\!24}$}
\put(-101,65){$\scriptstyle \e^{\gamma/3}$}
\put(-102,115){$\scriptstyle \e^{\gamma/8}$}
\put(-111,50){$\scriptstyle \e^{13\gamma/24}$}
\put(-69,178){$d=\e^{4\gamma/3}$}
\end{picture}
\end{center}
\end{figure}

The interest of considering $\mathcal{C}_N$ and $\mathcal{C}_E$ is that
the tangential speed of $X^{x,\alpha,\e}$ is under control inside these sets.

\begin{lemma}
\label{controle-vitesse-123}
For all $t\geq   4d(x)^{1/4} \e^{1-\gamma}$, we have $X^{x,\alpha,\e}(t)\in \{d\leq \e^{4\gamma /3}\}$.
Moreover
\begin{eqnarray*}
  && |\dot{X}_2^{x,\alpha,\e}(t)| \leq |f|_\infty (1+C\e^{5\gamma/24}), \quad \text{for a.e. $t\in \mathcal{T}_N$,}\\
  && |\dot{X}_1^{x,\alpha,\e}(t)| \leq |f|_\infty (1+C\e^{5\gamma/24}), \quad \text{for a.e. $t\in \mathcal{T}_E$.}  
\end{eqnarray*}
\end{lemma}

The proof of Lemma~\ref{controle-vitesse-123} is postponed to Appendix~\ref{sec-proof-lemma123}.

From Lemma~\ref{controle-vitesse-123}, we know that $X^{x,\alpha,\e}(t)\in \{d\leq \e^{4\gamma/3}\}$
for all $t\geq C\e^{1-\gamma}$. By symmetry, it follows that, without loss
of generality, we may choose $x\in \{x_1\geq 0, x_2\geq 0\}$
and assume that $X^{x,\alpha,\e}(t)\in  \mathcal{C}_{NE}$ for all $t\geq C\e^{1-\gamma}$.
Thus, one can write
\begin{eqnarray*}
  && (C\e^{1-\gamma}, \infty)
  = \mathcal{T}_N \cup \mathcal{T}_E \cup \mathcal{T}_O
  = \bigcup_{i\in\N} (s_i^N,t_i^N) \cup  \bigcup_{i\in\N} (s_i^E,t_i^E) \cup \mathcal{T}_O.
\end{eqnarray*}

Define $Y_N, Y_E \in W^{1,\infty}_{\rm loc} ([0,\infty); [0,\infty))$ by
\begin{eqnarray*}
  && Y_N(t)= \max \{ \e^{\gamma/8}, X_2^{x,\alpha,\e}(t) \},\\
  && Y_E(t)= \max \{ \e^{\gamma/8}, X_1^{x,\alpha,\e}(t) \}.
\end{eqnarray*}
For all $t\in \mathcal{T}_N$, $|(0,Y_N(t))- X^{x,\alpha,\e}(t)|\leq C\e^{13\gamma/24}$.
When $t\not\in \mathcal{T}_N$, $(0,Y_N(t))=(0, \e^{\gamma/8})$ is constant.
In this way, we construct a trajectory $(0,Y_N)\in \Gamma$,
which tracks $X^{x,\alpha,\e}$ in $\mathcal{T}_N$,
and, similarly, another trajectory $(Y_E,0)\in \Gamma$, which tracks  $X^{x,\alpha,\e}$ in $\mathcal{T}_E$.
The problem is that, on the one hand, the concatenation of the two trajectories is not a continuous one, and,
on the other hand, these trajectories may go a little bit too fast, see Lemma~\ref{controle-vitesse-123}.
We modify them accordingly by slowing them down, and translating them in order to be able
to define a continuous trajectory. For $t\geq C\e^{1-\gamma}$, set
\begin{eqnarray}
\nonumber  && \underline{Y}_N(t)=  (1 +C\e^{5\gamma/24})^{-1} (Y_N(t)- \e^{\gamma/8}),\\
\nonumber  && \underline{Y}_E(t)=  (1 +C\e^{5\gamma/24})^{-1} (Y_E(t)- \e^{\gamma/8}),\\
\label{def-traj-y}  &&  \underline{Y}=(\underline{Y}_E, \underline{Y}_N).
\end{eqnarray}

We claim that $\underline{Y}$ satisfies suitable properties making it a tracking trajectory on~$\Gamma$,
as stated in Lemma~\ref{traj-y-tilde} below, the proof of which is postponed
to Appendix~\ref{sec-proof-lemma123}.
\begin{lemma}
\label{traj-y-tilde}

The trajectory $\underline{Y}$ defined by~\eqref{def-traj-y} satisfies
\begin{enumerate}
\item $\underline{Y}\in W^{1,\infty}_{\rm loc} ([C\e^{1-\gamma},\infty); \Gamma)$,  
\item $|\dot{\underline{Y}}|\leq \sqrt{2}|f|_\infty$ a.e. on $[C\e^{1-\gamma},\infty)$,
\item Estimate~\eqref{estim-traj-111} holds
  with $ \underline{Y}$ instead of $X^{\overline{x}_\e,\beta,\e}$.
\end{enumerate}
\end{lemma}

The end of the proof of Proposition~\ref{suivi-traj-1} consists in proving that, on $[C\e^{1-\gamma}, \infty)$, we can write  $\underline{Y}$ under the form
$X^{\overline{x}_\e,\beta,\e}$ for suitable $\overline{x}_\e$ and $\beta_{\alpha, \e}\in \mathcal{A}_{\overline{x}_\e}$.

First, we set
\begin{eqnarray*}
\overline{x}_\e
:= \underline{Y}(C\e^{1-\gamma})= (1 +C\e^{5\gamma/24})^{-1}
\left\{
\begin{array}{ll}
(0,  X_2^{x,\alpha,\e}(C\e^{1-\gamma}) - \e^{\gamma/8}), & t\in \mathcal{T}_N,\\  
(0,0), & t\in \mathcal{T}_O,\\
(X_1^{x,\alpha,\e}(C\e^{1-\gamma})- \e^{\gamma/8},0), & t\in \mathcal{T}_E.
\end{array}  
\right.
\end{eqnarray*}
Due to Lemma \ref{traj-y-tilde}\,(1), $\overline{x}_\e\in\Gamma$. We claim that $\overline{x}_\e\to \overline{x}$ as $\e\to 0$.
Arguing as in the proof of Lemma~\ref{saut-sur-reseau}, we can prove 
that $X^{x,\alpha,\e}(C\e^{1-\gamma})=(X_1^{x,\alpha,\e}(C\e^{1-\gamma}), X_2^{x,\alpha,\e}(C\e^{1-\gamma}))\to \overline{x}$
as $\e\to 0$.
If $\overline{x}=0$, then it is obvious from the definition of $\overline{x}_\e$, that the sequence~$\overline{x}_\e$ converges to  $\overline{x}$, as $\e\to0$.
If $\overline{x}\not=0$, without loss of generality, we can assume, \emph{e.g.}, that $\overline{x}\in (0,\infty)e_N$.
It follows that, for $\e$ small enough, $X^{x,\alpha,\e}(C\e^{1-\gamma})\in \mathcal{C}_N$. Therefore,
$\overline{x}_\e=(1 +C\e^{5\gamma/24})^{-1}(0,  X_2^{x,\alpha,\e}(C\e^{1-\gamma})- \e^{\gamma/8})\to (0,\overline{x}_2)=\overline{x}$ as $\e\to 0$.

Then, we define a measurable control $\beta=\beta_{\alpha, \e}\in \mathcal{A}_{\overline{x}_\e}$ by
\begin{eqnarray*}
\beta (t)
:=
\left\{
\begin{array}{ll}
(0,0), & t\in [0,C\e^{1-\gamma}),\\*[.4em]  
\dot{\underline{Y}}(t)\in \overline{B}(0,|f|_\infty) & \text{a.e. $t\in [C\e^{1-\gamma},\infty)$.}
\end{array}  
\right.
\end{eqnarray*}
It follows that the corresponding trajectory is
\begin{eqnarray*}
X^{\overline{x}_\e,\beta,\e} (t)
=
\left\{
\begin{array}{ll}
\overline{x}_\e, & t\in [0,C\e^{1-\gamma}),\\*[.4em]  
\underline{Y}(t) & t\in [C\e^{1-\gamma},\infty).
\end{array}  
\right.
\end{eqnarray*}
This shows that $\underline{Y}=X^{\overline{x}_\e,\beta,\e}$ on the interval $[C\e^{1-\gamma}, \infty)$.
The conclusion of Proposition~\ref{suivi-traj-1} holds.
\end{proof}

We are now able to collect further properties of the value functions,
from which we will deduce easily the proof of Theorem~\ref{limit-vf}.

\begin{proposition}\label{propV-123}
Assume~\eqref{hyp-f-ell},~\eqref{calc-var}-\eqref{cout-indep}.
Let $R>0$, $x\in\overline{B}(0,R)$ and $\overline{x}=\phi_{d_\Gamma}(x)\in\Gamma$ (see~\eqref{x-proj}).
Then, for $\e >0$ small enough (depending only on $R>0$ and the given data), we have
\begin{enumerate}

\item $\overline{V}(\overline{x})\leq V^\e (x)+ m_R(\e)$, where $m_R$ is a modulus of continuity depending on $R$,

\item $V^\e (\overline{x})\leq \overline{V}(x)$,

\item  $V^\e (x)\leq V^\e (\overline{x}) + C \e^{1/4}$.

\end{enumerate}
\end{proposition}

\begin{proof}[Proof of Proposition~\ref{propV-123}] \ \\
(1) Let $R >0, \e >0$, $x\in\overline{B}(0,R)$ and $\alpha\in\mathcal{A}$.
From Proposition~\ref{suivi-traj-1}, for every $0<\gamma <1$,
if $\e >0$ is small enough compared to $R$ and $|f|_\infty$,  
then there exists $\Gamma\ni\overline{x}_\e \to \overline{x}$
and $\beta\in \mathcal{A}_{\overline{x}_\e}$ such that
$|X^{x,\alpha,\e}(t)-X^{\overline{x}_\e,\beta,\e}(t)|\leq  C( \e^{\gamma/8} +  \e^{5\gamma/24}t)$
for $t\geq C\e^{1-\gamma}$.
Since $\beta\in \mathcal{A}_{\overline{x}_\e}$, from Lemma~\ref{techn1}\,(1),
we obtain $\omega (\overline{x}_\e, \beta)=\{ X^{\overline{x}_\e, \beta}\}$. Hence,
\begin{eqnarray}\label{ineg591}
\overline{V}(\overline{x}_\e)
& \leq &
J(\overline{x}_\e , \beta, X^{\overline{x}_\e, \beta}) = J^\e (\overline{x}_\e , \beta)\\
\nonumber
&\leq &
J^\e (x , \alpha) + |J^\e (\overline{x}_\e , \beta) - J^\e (x , \alpha)|.
\end{eqnarray}
The constant $C$ below may vary from line to line but does not depend on neither $\e$, nor~$\alpha$.
For $\delta < 1-\gamma$ and $T>1$ to be fixed later, by choosing $\e$ smaller enough, we have
\begin{eqnarray*}
&& |J^\e (\overline{x}_\e , \beta) - J^\e (x , \alpha)|\\
&\leq& 2|\ell|_\infty \int_0^{\e^{\delta}} e^{-\lambda t}dt
+ \int_{\e^\delta}^{T}  e^{-\lambda t} |\ell (X^{\overline{x}_\e,\beta,\e})-\ell (X^{x,\alpha,\e})| dt
+  2|\ell|_\infty \int_{T}^{\infty}  e^{-\lambda t}dt \\
&\leq & C\e^{\delta} + C e^{-\lambda T}+ \int_{\e^{\delta}}^{T}  e^{-\lambda t} m_\ell (C( \e^{\gamma/8} +  \e^{5\gamma/24}t))dt,
\end{eqnarray*}
where $m_\ell$ is the modulus appearing in~\eqref{hyp-f-ell}.
We now choose $T= -\delta \lambda^{-1}\ln (\e)$ so that~$e^{-\lambda T}\leq \e^\delta$. It follows that
\begin{eqnarray*}
|J^\e (\overline{x}_\e , \beta) - J^\e (x , \alpha)|
& \leq &  C\e^{\delta} +  \int_{\e^{\delta}}^{C|\ln \e|}  e^{-\lambda t} m_\ell ( C(\e^{\gamma/8} +  \e^{5\gamma/24} |\ln \e|))dt\\
& \leq &  C\e^{\delta} + C\,   m_\ell ( C(\e^{\gamma/8} +  \e^{5\gamma/24} |\ln \e|)) =:m_R (\e),
\end{eqnarray*} 
with $m_R (\e)\to 0$ as $\e\to 0$. The modulus $m_R$ depends on $R$ through the various constants $C$.
Plugging this estimate into~\eqref{ineg591} and taking the infimum
over $\alpha\in\mathcal{A}$, we obtain (1).
\\
(2) Let $\eta >0$. There exist $\alpha_\eta\in \mathcal{A}$ and $X^{x,\alpha_\eta}\in \omega (x,\alpha_\eta)$, which
are $\eta$-optimal for $\overline{V}(x)$. From Lemma~\ref{techn1}\,(2), we have
$X^{x,\alpha_\eta}=X^{\overline{x},\overline{\alpha}_\eta}$ on $(0,\infty)$,
with $\overline{\alpha}_\eta\in\mathcal{A}_{\overline{x}}$ and
$\omega(\overline{x},\overline{\alpha}_\eta)=\{X^{\overline{x},\overline{\alpha}_\eta}\}$.
It follows that
\begin{eqnarray*}
&& \overline{V}(x) +\eta \geq J(x,\alpha_\eta, X^{x,\alpha_\eta})
= J(\overline{x},\overline{\alpha}_\eta, X^{\overline{x},\overline{\alpha}_\eta})
=J^\e (\overline{x},\overline{\alpha}_\eta)
\geq V^\e (\overline{x}),
\end{eqnarray*}
where we used~\eqref{cout-indep} in the first equality and Lemma~\ref{techn1}\,(1) in the
second one. Since $\eta$ is arbitrary, we conclude (2).
\\
(3) We only need to prove the result when $x\not\in\Gamma$.
We first construct a control  which drives $x$
to $\overline{x}$ in a short time.
There exists a unique solution $Y$ to the ODE
\begin{eqnarray}\label{edo-accel}
&& \dot{Y}= -\frac{\nabla d(Y)}{|\nabla d(Y)|}-\frac{1}{\e}\nabla d(Y), \quad Y(0)=x,
\end{eqnarray}
on the interval $[0,\tau^\e)$ where $\tau^\e:=\inf \{ t>0 : d(Y(t))=0\}$. 
Note that $Y(t)\in Z^x([0,\infty))$ with $d(Y(t))\leq d(Z^x(t))$ for all $t\in [0,\tau^\e)$,
where $Z^x$ is the gradient trajectory~\eqref{Grad-desc-eq}. It follows that, if
$\tau^\e <\infty$, then  $Y(\tau^\e)=\overline{x}$.  
For any $\alpha\in \mathcal{A}$, we define $\alpha^\e\in \mathcal{A}$ by
\begin{eqnarray*}
\alpha^\e(t)=
\left\{
\begin{array}{ll}
  -\frac{\nabla d(Y(t))}{|\nabla d(Y(t))|}, & 0\leq t <\tau^\e,\\
  \alpha (t-\tau^\e), & t\geq \tau^\e.
\end{array}
\right.
\end{eqnarray*}  
By uniqueness of the solution of~\eqref{edo-accel}, we infer that $X^{x,\alpha^\e,\e}=Y$ on $[0,\tau^\e]$.
Hence, for $0\leq t < \tau^\e$, we have
\begin{eqnarray*}
\frac{d}{dt}d^{1/4}(X^{x,\alpha^\e,\e})
&=&   \frac{1}{4}d^{-3/4}(X^{x,\alpha^\e,\e})
\left\langle \nabla d(X^{x,\alpha^\e,\e}), -\frac{\nabla d(X^{x,\alpha^\e,\e})}{|\nabla d(X^{x,\alpha^\e,\e})|}-\frac{1}{\e}\nabla d(X^{x,\alpha^\e,\e})\right\rangle
\\
&=& - \frac{1}{4}\left(1+ \frac{1}{\e}|\nabla d(X^{x,\alpha^\e,\e})|\right)  d^{-3/4}(X^{x,\alpha^\e,\e}) |\nabla d(X^{x,\alpha^\e,\e})|
\\
&\leq & -C(1+ \frac{d^{3/4}(X^{x,\alpha^\e,\e})}{\e} ),
\end{eqnarray*}
since $|\nabla d|\geq 2\sqrt{2} d^{3/4}$ by~\eqref{propri-d}.

By Proposition~\ref{prop-inv}\,(3) with $\gamma=3/4$, we know that
$0< t^\e:=t^{x,\alpha^\e , \e}(\e)\leq \tau^\e$ satisfies
$t^\e\leq C\e^{1/4}$. Integrating the previous inequality on $[t^\e,\tau^\e]$ yields
\begin{eqnarray*}
&& 0=d^{1/4}(X^{x,\alpha^\e,\e}(\tau^\e))
\leq
d^{1/4}(X^{x,\alpha^\e,\e}(t^\e))-C(\tau^\e - t^\e)\leq \e^{1/4} -C(\tau^\e -\e^{1/4}).
\end{eqnarray*} 
Hence,  $X^{x,\alpha^\e,\e}(\tau^\e)=\overline{x}$ with
$\tau^\e \leq C \e^{1/4} <\infty$, which shows that $X^{x,\alpha^\e,\e}$ reaches $\overline{x}$ in a short time. Therefore,
we have
\begin{eqnarray*}
  V^\e (x) \leq  J^{\e}(x,\alpha^\e)
  &=& \int_0^{\tau^\e}e^{-\lambda t}\ell(X^{x,\alpha^\e ,\e}(t))dt
  + e^{-\lambda \tau^\e} \int_{0}^{\infty} e^{-\lambda t}\ell(X^{x,\alpha^\e ,\e}(\tau^\e+t))dt
\\
&\leq & C\e^{1/4} + e^{-\lambda \tau^\e}  \int_{0}^{\infty} e^{-\lambda t}\ell(X^{\overline{x},\alpha ,\e}(t))dt
\\
&=&  C\e^{1/4} + e^{-\lambda \tau^\e} J^{\e}(\overline{x},\alpha)
\\
&\leq& C\e^{1/4} + J^{\e}(\overline{x},\alpha).
\end{eqnarray*}
Taking the infimum over $\alpha\in \mathcal{A}$ and using the
uniform boundedness of $V^\e$, we obtain~(3). Proposition \ref{propV-123} is proved.
\end{proof}

\subsection{Proof of Theorem~\ref{limit-vf}}\label{proof-theo3}
The first part of the theorem is the statement of Proposition~\ref{barV-Vgamma}.
By chaining Inequalities (1), (3) and then (2) (with $\phi_{d_\Gamma}(x)=\overline{x}$ instead of $x$) of
Proposition~\ref{propV-123}, we obtain, for all $R>0$ and
$x\in \overline{B}(0,R)$, that
\begin{eqnarray*}
  && \overline{V}(\overline{x})\leq V^\e (x)+ m_R(\e) \leq V^\e (\overline{x}) + C \e^{1/4} + m_R(\e)
  \leq \overline{V}(\overline{x}) + C \e^{1/4}  + m_R(\e).
\end{eqnarray*}
It follows that $V^\e$ converges locally uniformly to $\overline{V}\circ\phi_d$.
In particular, thanks to Lemma~\ref{Veps-lip-gamma}, $\overline{V}=V_\Gamma$ is Lipschitz continuous
on $\Gamma$. The proof of Theorem~\ref{limit-vf} is complete.~\hfill$\Box$

\appendix
\section{The line network case}\label{line-junction}

For the reader's convenience, we illustrate our approach to the very simple case
$$
\Gamma = \{O\}\cup (0,\infty) e_E \cup (0,\infty) e_W,
$$
when the ambient optimal control problem~\eqref{vf-classique}  in $\R^2$
satisfies Assumptions~\eqref{calc-var}
and~\eqref{cout-indep}, where $\ell$ satisfies~\eqref{hyp-f-ell}.
In this case, choosing the natural function $d_\Gamma(x_1,x_2)=x_2^2$,
which satisfies~\eqref{cond-d},
all the computations can be done explicitly.

The perturbed optimal control~\eqref{val-eps} is governed by the ODE~\eqref{traj-eps}
which reads
\begin{eqnarray*}
  && 
  \left\{
  \begin{array}{l}
    \dot{X}_1^{x,\alpha, \varepsilon}(t)=\alpha_1(t),\\
    \dot{X}_2^{x,\alpha, \varepsilon}(t)=\alpha_2(t) -\frac{2}{\varepsilon} X_2^{x,\alpha, \varepsilon}(t),
  \end{array}
  \right.
  \quad t> 0,
\end{eqnarray*} 
with $X^{x,\alpha, \varepsilon}(0)=x=(x_1,x_2)\in \R^2$.
The explicit solution is
\begin{eqnarray*}
{X}^{x,\alpha, \varepsilon}(t)= \left( x_1+\int_0^t \alpha_1(s)ds \ , \ e^{-\frac{2}{\e}t}x_2 + \int_0^t e^{\frac{2}{\e}(s-t)}\alpha_2(s)ds\right).
\end{eqnarray*} 
In this case, the set $\omega(x,\alpha)$ (see~\eqref{omega-lim}) is reduced to a single element since ${X}^{x,\alpha, \varepsilon}$ converges pointwisely, as $\e\to 0$, to
\begin{eqnarray*}
{X}^{x,\alpha}(t)=
 \left\{
 \begin{array}{ll}
    (x_1,x_2), & t=0,\\
    \left( x_1+\int_0^t \alpha_1(s)ds \ , \ 0 \right), & t>0.
  \end{array}
 \right.
\end{eqnarray*}
The convergence is uniform on every $[\eta, +\infty)$, $\eta >0$. In particular,
\begin{eqnarray*}
 |X^{x,\alpha, \varepsilon}(t)-X^{x,\alpha}(t)| \leq e^{-\frac{2}{\sqrt{\e}}}|x_2|+\frac{\e}{2},
 \quad \text{for $t\in [ \sqrt{\e}, +\infty )$.}
\end{eqnarray*}
Moreover, $\phi_{d_\Gamma}(x)=\overline{x}$ is simply $(x_1,0)$,
$k^{x,\alpha}(t)=(0, x_2+\int_0^t \alpha_2(s)ds)$ for $t>0$,
and $X^{x,\alpha}=X^{\overline{x},\alpha}$ on $(0,\infty)$. The last fact is not clear
when $\omega(x,\alpha)$ is not a singleton, and complicates the general study in the paper.

It follows that, for all $x\in\R^2$,
\begin{eqnarray*}
\overline{V}(x) &:=& \inf_{\alpha \in \mathcal{A}} \int_0^\infty e^{-\lambda t} \ell (X^{x,\alpha}(t))dt
\\
&=&\inf_{\alpha \in \mathcal{A}} \int_0^\infty e^{-\lambda t} \ell \left( x_1+\int_0^t \alpha_1(s)ds \; , \; 0\right)dt
\\
&=&\inf_{\alpha \in \mathcal{A}_{\overline{x}}} \int_0^\infty e^{-\lambda t} \ell ( Y^{\overline{x},\alpha}(t))dt
= V_\Gamma(\overline{x}) = V_\Gamma \circ \phi_d(x),
\end{eqnarray*}
where $\mathcal{A}_x$ is defined by~\eqref{def-Ax} and $Y^{\overline{x},\alpha}$ is the solution of the ODE
$\dot Y^{\overline{x},\alpha}=\alpha$, $Y^{\overline{x},\alpha}(0)=\overline{x}$.

Taking into account~\eqref{hyp-f-ell}, we have
\begin{eqnarray*}
|V^\e(x)- \overline{V}(x)|
&\leq &  \sup_{\alpha \in \mathcal{A}} \int_0^\infty e^{-\lambda t} |\ell (X^{x,\alpha, \varepsilon}(t))-\ell (X^{x,\alpha}(t))|dt\\
&\leq & \int_0^{\sqrt{\e}} 2M dt +
 \sup_{\alpha \in \mathcal{A}} \int_{\sqrt{\e}}^{\infty} e^{-\lambda t} m_\ell (|X^{x,\alpha, \varepsilon}(t)-X^{x,\alpha}(t)|) dt\\
&\leq&
2M \sqrt{\e}+ \frac{1}{\lambda} m_\ell (e^{-\frac{2}{\sqrt{\e}}}|x_2|+\frac{\e}{2}),
\end{eqnarray*}
which proves that $V^\e$ converges locally uniformly in $\R^2$, as $\e\to 0$, to $\overline{V}=V_\Gamma \circ \phi_{d_\Gamma}$.

Let's end by noting that, due to the very simple case considered here (with continuous running cost and dynamics, and
aligned edges), one does not see the junction point and $V_\Gamma$ is nothing but the value function
of the one-dimensional control problem on $\R$ satisfying the one-dimensional Eikonal equation
$$
\lambda V_\Gamma ((x_1,0))+ |\nabla_{x_1} V_\Gamma ((x_1,0))|=\ell (x_1,0), \quad x_1\in\R.
$$

\section{Properties of $d_\Gamma$}\label{prop-dist}

In this section, we state and prove some properties of a function $d_\Gamma:\R^2\to [0,\infty)$
satisfying~\eqref{cond-d}
with the additional assumption that $d_\Gamma$ satisfies a {\L}ojasiewicz inequality on the
whole space. These conditions are trivially satisfied for our function $d_\Gamma$ given by~\eqref{choix-d},
see Proposition~\ref{traj-d1}, but we want to keep some generality on $d_\Gamma$ in order to be able to deal
with general junctions~\eqref{jonction-generale}
(see discussions in~\cite[Chapter 6]{chuberre23}).

\begin{theorem}\label{thm-loja}
Suppose that $d_\Gamma\in C_{\rm loc}^{1,1}(\R^2; [0,\infty))$ satisfies~\eqref{cond-d}
and that there exists $\nu >0$, $\theta\in (0,1)$ such that
\begin{eqnarray}\label{hyp-loja}
&& |\nabla d_\Gamma(x)|\geq \nu d_\Gamma(x)^\theta \text{ for all $x\in\R^2$ \quad ({\L}ojasiewicz inequality).}
\end{eqnarray}
Then,  for any initial position $x\in\R^2$, the unique solution $Z^x : [0,\infty)\to\R^2$ of 
\begin{equation}\label{Grad-desc-eq}
\dot{X}= -\nabla d_\Gamma(X), \ X(0)=x
\end{equation}
has a limit $\overline{x}\in \Gamma$, as $t\to\infty$. Moreover, the map
\begin{equation}\label{x-proj}
\phi_{d_\Gamma}: x\in\R^2 \mapsto
\overline{x}:=\lim_{t\to\infty}Z^x(t)\in\Gamma
\end{equation}
is continuous.
\end{theorem}

We refer to {\L}ojasiewicz~\cite{lojasiewicz84} for seminal references on the subject.
\smallskip

\begin{proof}
When there is no ambiguity, we write $d$, $\phi$, $Z$ for $d_\Gamma$, $\phi_{d_\Gamma}$ and $Z^x$, respectively.
Since $\nabla d$ is locally Lipschitz continuous, there exists a unique maximal solution
$Z$ to~\eqref{Grad-desc-eq} defined on $[0,T)$, for all $T>0$.

Using~\eqref{hyp-loja} and \eqref{Grad-desc-eq},  we get, for all $t\in [0,T),$
\begin{eqnarray*}
  && \frac{1}{1-\theta} \frac{d}{dt} d^{1-\theta}(Z)
=  \langle \nabla d(Z), \dot{Z}\rangle d^{-\theta}(Z)
=  - |\nabla d(Z)|^2 d^{-\theta}(Z)
\leq -\nu |\dot{Z}|.
\end{eqnarray*}
We deduce two consequences from this inequality. First,
$t\in [0,T)\mapsto d(Z(t))\in [0,\infty)$ is nonincreasing,
thus $\lim_{t\to T}d(Z(t))=:\ell \geq 0$. Secondly, by integration, for all $0\leq t_1\leq t_2 <T$, 
\begin{eqnarray}\label{ineg-cauchy}
  && |Z(t_2)-Z(t_1)|\leq \int_{t_1}^{t_2}|\dot{Z}(t)|dt
  \leq \frac{1}{\nu(1-\theta)}\left(d^{1-\theta}(Z(t_1))- d^{1-\theta}(Z(t_2))\right).
\end{eqnarray}
Notice that this inequality ensures the finiteness of the length of the trajectories $Z$,
which is a nontrivial result. More simply, we have, for all $0\leq t<T$,
\begin{eqnarray}\label{Zborne}
  && |Z(t)|\leq |Z(0)|+ \frac{1}{\nu(1-\theta)}d^{1-\theta}(Z(0)),
\end{eqnarray}
from which we infer that $Z$ is bounded and the solution is global in time, \emph{i.e.}, $T=\infty$.

We claim that $\ell=0$. Otherwise, there exist $\eta >0$ such that $d(Z(t))\geq \eta$ for all~$t\geq 0$. 
Thanks to~\eqref{hyp-loja} and~\eqref{ineg-cauchy}, we get
\begin{eqnarray}\label{temps-Z}
  && \ \nu \eta^\theta t \leq \nu\! \int_{0}^{t} d^\theta (Z(s))ds
  \leq \! \int_{0}^{t}|\nabla d(Z(s))|ds
  = \!\int_{0}^{t}|\dot{Z}(s)|ds\leq  \frac{1}{\nu(1-\theta)} d^{1-\theta}(Z(0)),
\end{eqnarray}
which is in contradiction with $t$ large enough.

It follows from the claim and~\eqref{ineg-cauchy} that $(Z(t))_{t\geq 0}$ is a Cauchy sequence in $\R^2$
as~$t \to \infty$. Therefore, there exists some $\overline{x}\in\R^2$ such that $Z(t)\to \overline{x}$, as $t \to \infty$.
Since~$\ell=0$ and~\eqref{cond-d} holds, we conclude that $\overline{x}\in \Gamma$
and $\phi$ is well-defined.

We turn to the proof of the regularity of the map $\phi$.
Let $\overline{B}(0,R)\subset \R^2$. Our goal is to bound $|\phi(x)-\phi(y)|$
for  $x,y\in \R^2$ in a suitable way. We start with two preliminaries.

On the one hand, thanks to~\eqref{Zborne},
for all $x\in \overline{B}(0,R)$ and $t\geq 0$, we have
\begin{eqnarray*}
  && |Z^x(t)| \leq  \overline{R}:=R+\frac{1}{\nu(1-\theta)} \sup_{\overline{B}(0,R)}d^{1-\theta}.
\end{eqnarray*}

On the other hand, for $x\in \R^2$ and $\varepsilon \geq 0$, we
define $t^{x,\varepsilon}$ to be the time needed for~$Z^x$ to reach the $\varepsilon$-level set of $d$, namely 
$d(Z(t^{x,\varepsilon}))= \varepsilon$.
If $\varepsilon\geq d(x)$, we set $t^{x,\varepsilon}=0$. When~$x\in \R^2\setminus \Gamma$, $0<\varepsilon < d(x)$,
from~\eqref{temps-Z}, we obtain
\begin{eqnarray}\label{estim-teps}
  && t^{x,\varepsilon} \leq  \frac{d^{1-\theta}(x)}{\nu^2(1-\theta)} \varepsilon^{-\theta}.
\end{eqnarray}

{\it Case 1. $x\not\in \Gamma$.} Without loss of generality, we may assume that
$t^{y,\varepsilon}\leq t^{x,\varepsilon}$. We write
\begin{eqnarray*}
&&  |\phi(x)- \phi(y)|
\leq  |\phi(x)- Z^x(t^{x,\varepsilon}) |+ |Z^x(t^{x,\varepsilon}) - Z^y(t^{x,\varepsilon})|+ |\phi(y)- Z^y(t^{x,\varepsilon}) |,
\end{eqnarray*}
and we estimate each term separately.
For the first one, from~\eqref{ineg-cauchy}, we have
\begin{eqnarray*}
  &&  |\phi(x)- Z^x(t^{x,\varepsilon}) | \leq \int_{t^{x,\varepsilon}}^{\infty}|\dot{Z}^x(t)|dt
  \leq \frac{1}{\nu(1-\theta)} d^{1-\theta}(Z^x(t^{x,\varepsilon}))\leq \frac{\varepsilon^{1-\theta}}{\nu(1-\theta)}.
\end{eqnarray*}
Since $t^{y,\varepsilon}\leq t^{x,\varepsilon}$, we obtain the same estimate for the third term.
For the second term, using~\eqref{Grad-desc-eq}, we infer, for all $0\leq t\leq t^{x,\varepsilon}$,
\begin{eqnarray*}
  |Z^x(t) - Z^y(t)|
  &\leq& |x-y| + \int_0^{t} |\nabla d(Z^x)- \nabla d(Z^y)|dt\\
&\leq & |x-y| +  {\rm Lip}_{\overline{B}(0,\overline{R})}(\nabla d)\int_0^{t}  |(Z^x-Z^y)(t)|dt,
\end{eqnarray*}
where ${\rm Lip}_{\overline{B}(0,\overline{R})}(\nabla d)$ is the Lipschitz constant of $\nabla d$ in
the ball of radius $\overline{R}$ centered at~$O$.
From Gr\"onwall inequality and~\eqref{estim-teps}, we obtain
\begin{eqnarray*}
  |Z^x(t^{x,\varepsilon}) - Z^y(t^{x,\varepsilon})| \leq e^{C \varepsilon^{-\theta}} |x-y|,
\end{eqnarray*}
with $C= \nu^{-2}(1-\theta)^{-1}{\rm Lip}_{\overline{B}(0,\overline{R})}(\nabla d) \sup_{B(0,R)}d^{1-\theta}$.
Finally, we arrive at
\begin{eqnarray*}
&&  |\phi(x)- \phi(y)|
\leq  \frac{2}{\nu(1-\theta)} \varepsilon^{1-\theta} +  e^{C \varepsilon^{-\theta}} |x-y|.
\end{eqnarray*}

{\it Case 2. $x\in \Gamma$.}
If $y\in \Gamma$, then  $|\phi(x)- \phi(y)|=|x-y|$.
When $y\not\in \Gamma$, using~\eqref{ineg-cauchy}, we have
\begin{eqnarray*}
  |\phi(x)- \phi(y)|
  &\leq& |x-y|+ |y- \phi(y)|\\
  &\leq& |x-y|+ \frac{1}{\nu(1-\theta)}d^{1-\theta}(y)\\
  &\leq&  |x-y|+ \frac{1}{\nu(1-\theta)}| d^{1-\theta}(y)-d^{1-\theta}(x)|\\
  &\leq&  |x-y|+ \frac{{\rm Lip}_{\overline{B}(0,\overline{R})}(d)^{1-\theta}}{\nu(1-\theta)} |x-y|^{1-\theta},
\end{eqnarray*}
since $b^{1-\theta}-a^{1-\theta}\leq (b-a)^{1-\theta}$ for $0\leq a\leq b$. 

{\it Conclusion.} Putting together the estimates of Cases 1 and 2, we obtain that
there exists a constant $C_R$ such that for all $x,y\in B(0,R)$,
\begin{eqnarray*}
&&  |\phi(x)- \phi(y)|
\leq  C_R |x-y|^{1-\theta}+ \inf_{\varepsilon >0}\left\{ \frac{2}{\nu(1-\theta)} \varepsilon^{1-\theta} +  e^{C_R \varepsilon^{-\theta}} |x-y|\right\}.
\end{eqnarray*}
It follows that there exists a modulus of continuity $\omega_R$ such that
$|\phi(x)- \phi(y)|\leq \omega_R(|x-y|)$. The theorem is proved.
\end{proof}

\begin{proposition}\label{traj-d1}
The function $d_\Gamma(x)=x_1^2 x_2^2$ defined by~\eqref{choix-d} for $\Gamma$ given by~\eqref{reseau}
satisfies the assumptions of Theorem~\ref{thm-loja}. More precisely,
if $Z^x=(Z_1^x,Z_2^x)$ is the solution of~\eqref{Grad-desc-eq}, then $Z^x(t)$ lies on a
hyperbola orthogonal to the level sets of $d_\Gamma$, that is
\begin{eqnarray}\label{hyperb}
&& (Z_2^x(t))^2-(Z_1^x(t))^2 =x_2^2-x_1^2 \quad \text{for all $t\geq 0$,}
\end{eqnarray}
and
\begin{eqnarray}\label{form-proj}
  && \phi_{d_\Gamma}(x)=\left\{
  \begin{array}{ll}
    {\rm sgn}(x_2) \sqrt{x_2^2-x_1^2} e_N & \text{if $|x_2|\geq |x_1|$,}\\[.6em]
    {\rm sgn}(x_1) \sqrt{x_1^2-x_2^2} e_E & \text{if $|x_1|\geq |x_2|$,}
  \end{array}
  \right.
\end{eqnarray}
is $1/2$-H\"older continuous on $\R^2$.
\end{proposition}

\begin{proof}
The function $d_\Gamma$ is polynomial, satisfies~\eqref{cond-d}, and {\L}ojasiewicz inequality~\eqref{hyp-loja}
with $\nu=2\sqrt{2}$ and $\theta=3/4$ (straightforward computation, see~\eqref{propri-d}).
Multiplying the first line of the system~\eqref{Grad-desc-eq} by $Z_1^x(t)$ and the second line by $Z_2^x(t)$
and subtracting, we obtain $\dot{Z}_1^x Z_1^x - \dot{Z}_2^x Z_2^x =0$, from which, we get~\eqref{hyperb}.
Taking the intersection of the hyperbola with the axis, we deduce easily~\eqref{form-proj}.
\end{proof}

\section{Proofs of Propositions~\ref{behav-O} and~\ref{behav-1}}\label{proof-beh-thr-origin}

To perform the proofs, we write the ODE~\eqref{traj-eps}
$\dot{X}^{x,\alpha,\e}=F^\e(X^{x,\alpha,\e},\alpha)$ component-wise
and, since the control $\alpha\equiv e_\theta$ and the starting point are
fixed, we omit the dependence with respect to $x,\alpha$, that is,
\begin{eqnarray}\label{edo-c}
  && 
  \left\{
  \begin{array}{l}
    \dot{X}_1^\varepsilon = F_1^\e(X^\e) = \cos\theta -\frac{2}{\varepsilon} X_1^\varepsilon (X_2^\varepsilon)^2,\\[1.5mm]
     \dot{X}_2^\varepsilon = F_2^\e(X^\e) = \sin\theta -\frac{2}{\varepsilon} (X_1^\varepsilon)^2 X_2^\varepsilon,
  \end{array}
  \right.
  \quad t> 0.
\end{eqnarray}

We recall an easy lemma for linear ODE and state a general scaling property for~\eqref{edo-c}.

\begin{lemma}\label{qualit-lin-ode}
Let  $z: [0,\infty)\to \R$ be the solution of the linear ODE
\begin{eqnarray}\label{edo-l}
&& \dot{z}(t)=\alpha(t) + a(t)z(t), \qquad z(0)=z_0,
\end{eqnarray} 
where $\alpha, a :[0,\infty)\to \R$ are continuous functions and $z_0\in\R$.
\begin{enumerate}
\item If $\alpha(t), z(0)=0$ (respectively $\geq 0$), then $z(t)\equiv 0$  (respectively $\geq 0$).    
\item  If $\alpha (t), a(t)\geq 0$, then $z(t)\geq z(0)+\int_0^t \alpha (s)ds$.
\end{enumerate}
\end{lemma}

\begin{proof}
All the properties can be read directly on the explicit solution
$z(t)= z(0)e^{A(t)}+ \int_0^t e^{A(t)-A(s)}\alpha (s) ds$, where  
$A(t)=\int_0^t a(s) ds$.
%
\end{proof}

\begin{lemma}[Scaling property]\label{lem-scaling}
  For every $\rho >0$, the
  solution of~\eqref{edo-c} satisfies
\begin{eqnarray}\label{prop-scaling}
&& X^{x,e_\theta,\e}(t)= \frac{1}{\rho} X^{\rho x, e_\theta, \rho^3 \e}(\rho t).
\end{eqnarray} 
\end{lemma}

\begin{proof}
Set $Y(t)= \frac{1}{\rho} X^{\rho x, e_\theta, \rho^3 \e}(\rho t)$. Then, since $\nabla d(x)= 2 x_1 x_2 (x_2, x_1)$,
\begin{eqnarray*}
  && \dot{Y}(t)= \dot{X}^{\rho x, e_\theta, \rho^3 \e}(\rho t) = e_\theta -\frac{1}{\rho^3\e}\nabla d(X^{\rho x, e_\theta, \rho^3 \e}(\rho t))
  =  e_\theta -\frac{1}{\e}\nabla d(Y(t)).
\end{eqnarray*} 
Taking into account the fact that $Y(0)=x$, by uniqueness of the solution of~\eqref{edo-c}, we obtain that $Y=X^{x,e_\theta,\e}$.
\end{proof}

\subsection{Proof of Proposition~\ref{behav-O}}

We start with the proof of (1). We recall that $X^\e(0)=O$ and, by symmetry, we may assume, for instance, that
$$
\max_i \langle e_\theta , e_i\rangle = \langle e_\theta , e_E\rangle = \cos\theta > \sin\theta >0. 
$$ 
We first infer, from~\eqref{edo-c},
that $z:= X_1^\varepsilon$ (respectively $z:=X_2^\varepsilon$)
is the solution to the linear ODE~\eqref{edo-l} with $\alpha = \cos\theta >0$
(respectively $\alpha = \sin\theta >0$)
and $a(t)= -2\varepsilon^{-1} (X_2^\varepsilon(t))^2$  (respectively $a(t)=-2\varepsilon^{-1} (X_1^\varepsilon(t))^2$) satisfying $z(0)=0$. Then, 
$
X_1^\varepsilon(t)=(\cos\theta)\int_0^t e^{A(t)-A(s)}ds\geq0$ (respectively $X_2^\varepsilon(t)=(\sin\theta)\int_0^t e^{B(t)-B(s)} ds\geq0)$, where $A(t):=-2\varepsilon^{-1} \int_0^t (X_2^\varepsilon(s))^2 ds$ (respectively $B(t):=-2\varepsilon^{-1} \int_0^t (X_1^\varepsilon(s))^2 ds$). 

Moreover, in~\eqref{edo-c}, subtracting the second line from the first one shows that~$z:=X_1^\varepsilon- X_2^\varepsilon$ is the solution to the linear ODE~\eqref{edo-l} with $\alpha = \cos\theta - \sin\theta >0$
and $a(t)= 2\varepsilon^{-1} X_1^\varepsilon(t) X_2^\varepsilon(t)\geq0$, satisfying $z(0)=0$. Then, 
$
X_1^\varepsilon(t)- X_2^\varepsilon(t)=(\cos\theta - \sin\theta)\int_0^t e^{C(t)-C(s)}ds\geq (\cos\theta - \sin\theta)t $, where $C(t):=2\varepsilon^{-1} \int_0^t X_1^\varepsilon(s)X_2^\varepsilon(s) ds$.

 From Theorem~\ref{cv-traj}, letting $\varepsilon\to 0$,
we obtain that $X_1(t)- X_2(t)\geq  (\cos\theta - \sin\theta)t$, $X_1(t), X_2(t)\geq 0$ and $d(X(t))=0$. It follows that, 
necessarily $X(t)\in (0,\infty)e_E$ for all~$t>0$.
From Theorem~\ref{edo-interieur}, we get the desired result.

We turn to the proof of (2), assuming, for instance, that $\theta= \frac{\pi}{4}$.
From~\eqref{edo-c}, we obtain that $z=X_1^\varepsilon- X_2^\varepsilon$ is the solution
of  the linear ODE~\eqref{edo-l} with $\alpha=0$. By Lemma~\ref{qualit-lin-ode},
we obtain $X_1^\varepsilon= X_2^\varepsilon\geq 0$. It follows that
the first line of~\eqref{edo-c} reads
$\dot{X}_1^\varepsilon + 2\varepsilon^{-1} (X_1^\varepsilon)^3 = \cos\theta$ with
$X_1^\varepsilon(0)=0$. We can solve explicitly this ODE to obtain that
$X_1^\varepsilon(t)$ converges to the equilibrium $s^\varepsilon:= (\cos\theta /2)^{1/3} \varepsilon^{1/3}$
as $t\to \infty$. Since $s^\varepsilon\to 0$ as $\varepsilon\to 0$,
we get, at the limit, that $X_1(t)\equiv 0$ and then $X(t)\equiv O$ as
desired.
\qed

\subsection{Proof of Proposition~\ref{behav-1}}
\
\medskip

From Theorem~\ref{edo-interieur}, we have
\begin{eqnarray}\label{formeN}
X^{e_N, e_\theta}(t)=(1+(\sin\theta)t)e_N \quad  \text{for all $t\in [0,\underline{t}]$,}
\end{eqnarray}
with $\underline{t}=\inf\{s\geq 0: X^{e_N, e_\theta}(s)=O\}$.
In the case of part (1) ($\sin\theta \geq 0$),~\eqref{formeN} holds for all $t\geq 0$
since  $\underline{t}=\infty$.
In the case of parts (2) and (3),~\eqref{formeN} holds for all $t\in [0, (-\sin\theta)^{-1}]$.
Part (4) follows from (2) by symmetry.
\medskip

We now concentrate on the proof of (2),(3) for $t\geq  (-\sin\theta)^{-1}$ (after the trajectory hit $O$).
We recall that, in all the proof, we assume
\begin{eqnarray*}
  && \alpha\equiv e_\theta \quad \text{with $\theta\in (\pi, \frac{5\pi}{4}]$ (case (2))
  or $\theta\in (\frac{5\pi}{4}, \frac{3\pi}{2}]$ (case (3))}.
\end{eqnarray*}
The keystone to obtain the result is to prove that, for $t>(-\sin\theta)^{-1}$,
$X_1^{e_N, e_\theta, \e}(t)\leq -\eta$ (respectively $X_2^{e_N, e_\theta, \e}(t)\leq -\eta$)
for some $\eta >0$ independent of $\e$. This proves
that the trajectory necessarily enters the branch $W$ (respectively the branch $S$).

The proof consists in a succession of
lemmas.

\begin{lemma}\label{tepsilon}
Let $\theta\in (\pi, \frac{3\pi}{2}]$.
There exists $t_\e \in (0,  (-\sin\theta)^{-1}]$ such that
\begin{eqnarray*}
  &&   \eta_\e := X_1^{e_N, e_\theta, \e}(t_\e) \leq 0
  \quad \text{and} \quad X_2^{e_N, e_\theta, \e}(t_\e)=0,
\end{eqnarray*}
and, for all $t\geq t_\e$,  $X^{e_N, e_\theta, \e}(t)\in \{ x_1 \leq 0\}\cap \{ x_2 \leq 0\}$. 
\end{lemma}

\begin{proof}
Define $t_\e =\inf \{ t\geq 0 : X_2^\e(t) =0 \}$. Since  $X_2^\e(0)=1$,
$0<t_\e \leq \infty$. For all $t\in [0, t_\e)$,
\begin{eqnarray}\label{x2neg}
\dot{X}_2^\e (t)= \sin\theta -\frac{2}{\e} (X_1^\e(t))^2 X_2^\e(t) \leq \sin\theta < 0,
\end{eqnarray}  
thus, by integration, $X_2^\e(t)\leq 1+(\sin\theta)t$.
Since $\sin\theta <0$, it follows that $t_\e\leq -(\sin\theta)^{-1}$.

Next, we prove that $Q_W:= \{ x_1\leq 0\}$ is invariant for~\eqref{traj-eps}, for all $t\geq 0$.
Indeed, for all $x\in \partial Q_W$, $\nabla d(x)=0$ and $n_{\partial Q_W} (x)=e_E$,
leading to $\langle F^\e (x,e_\theta) , n_{\partial Q_W} (x)\rangle = \langle  e_\theta, e_E\rangle =\cos\theta <0$.
As a consequence, we obtain that $\eta_\e :=X_1^\e (t_\e)\leq 0$.

Similarly,  $Q_S:= \{ x_2\leq 0\}$ is invariant for  $X^{\e}$ for $t\geq t_\e$.
\end{proof}

\begin{lemma}\label{etaepsilon}
Let $\theta\in (\pi, \frac{3\pi}{2}]$.
There exists $\gamma >0$ (independent of $\e$) such that
\begin{eqnarray}\label{ineq-etae}
  && - \frac{3}{2} \frac{(-\cos\theta)}{(-\sin\theta)^{2/3}}  \e^{1/3}\leq   \eta_\e = X_1^{e_N, e_\theta, \e}(t_\e) \leq - \gamma \e^{1/3},
\end{eqnarray}
where we recall that $t_\e, \eta_\e$ are introduced in Lemma~\ref{tepsilon}.
\end{lemma}

\begin{proof}
We start by proving the first inequality. Let $A>0$. Then, $z(t)=X_1^{\e}(t)+ A \e^{1/3}$ is solution
to the linear ODE~\eqref{edo-l} with $\alpha(t)= \cos\theta + \frac{2 A}{\e^{2/3}} (X_2^{\e}(t))^2$,
$a(t)= -\frac{2}{\e}(X_2^{\e}(t))^2$ and $z(0)= A \e^{1/3} \geq 0$. One has $z(t)=A \e^{1/3}e^{A(t)}+\int_0^t (\cos\theta + \frac{2 A}{\e^{2/3}} (X_2^{\e}(t))^2)e^{A(t)-A(s)}ds$, where $A(t):=-2\varepsilon^{-1} \int_0^t (X_2^\varepsilon(s))^2 ds$, for all $t\geq0$. 

The function  $\alpha(t)= \cos\theta + \frac{2 A}{\e^{2/3}} (X_2^{\e}(t))^2$
is nonnegative as soon as
\begin{eqnarray}\label{x2grand}
  && X_2^{\e}(t) \geq \sqrt{\frac{-\cos\theta}{2A}}\e^{1/3}.
\end{eqnarray}
By \eqref{x2neg}, $X_2^\e$ is decreasing on $[0, t_\e]$ with $X_2^\e(0)=1$ and $X_2^\e (t_\e)= 0$.
Therefore, for~$\varepsilon>0$ small enough so that $0<\sqrt{\frac{-\cos\theta}{2A}}\e^{1/3}<1$, there exists $0<\tau_\e < t_\e$ such that~$X_2^\e (\tau_\e)=\sqrt{\frac{-\cos\theta}{2A}}\e^{1/3}$
and~\eqref{x2grand} holds for $0\leq t\leq \tau_\e$, which implies that $z\geq 0$ on~$[0,\tau_\e]$, thus
$X_1^\e (\tau_\e)\geq -A \e^{1/3}$.

Integrating~\eqref{x2neg}, we obtain $0\leq X_2^{\e}(t)\leq X_2^\e (\tau_\e) +(\sin\theta)(t -\tau_\e)$
for $t\in [\tau_\e, t_\e]$,
leading to
\begin{eqnarray*}
  &&  t -\tau_\e \leq \frac{1}{- \sin\theta}\sqrt{\frac{-\cos\theta}{2A}}\e^{1/3}.
\end{eqnarray*}
By Lemma~\ref{tepsilon}, $X_1^{\e}\leq 0$ on $[0,t_\e]$. It follows that
$\dot{X}_1^\e = \cos\theta -\frac{2}{\e}  X_1^{\e}(X_2^{\e})^2\geq  \cos\theta$
and, by integration,
\begin{eqnarray}\label{borne-xx}
  && X_1^{\e}(t) \geq X_1^{\e}(\tau_\e) + (\cos\theta) ( t_\e -\tau_\e) \geq
   -A \e^{1/3} - \frac{\cos\theta}{\sin\theta}\sqrt{\frac{-\cos\theta}{2A}}\e^{1/3} =: - g(A) \e^{1/3},
\end{eqnarray}
for all $t\in [\tau_\e, t_\e]$.
A tedious computation yields $\min_{A>0}g(A)= g(\overline{A})= \frac{3}{2} \frac{(-\cos\theta)}{(-\sin\theta)^{2/3}}$
with $\overline{A}= -\frac{\cos\theta}{2(-\sin\theta)^{2/3}}$. Therefore,
maximizing the right-hand side of~\eqref{borne-xx} with respect to $A>0$, we obtain
the first inequality in~\eqref{ineq-etae}.

We turn to the proof of the second inequality in~\eqref{ineq-etae}.
From the previous step, we have
\begin{eqnarray}\label{borne-yy}
&& 0\leq
X_2^{\e}(t) \leq \sqrt{\frac{-\cos\theta}{2\overline{A}}}\e^{1/3} \qquad  \text{on $[\tau_\e,  t_\e]$,}
\end{eqnarray}
which yields
\begin{eqnarray}\label{ODI-1} 
  &&   \dot{X}_1^\e = \cos\theta -\frac{2}{\e} X_1^\e (X_2^\e)^2 \leq \cos\theta +\frac{\cos\theta}{\e^{1/3}\overline{A}} X_1^\e.
\end{eqnarray}
Integrating~\eqref{ODI-1} on $[\tau_\e,  t_\e]$, we obtain
\begin{eqnarray}\label{estim-eta}
  &&   \eta_\e = X_1^\e(t_\e) = \leq -\overline{A}  \left(1- e^{ \frac{\cos\theta}{\e^{1/3}\overline{A}}  (t_\e -\tau_\e)}\right) \e^{1/3}.
\end{eqnarray}

We now derive a lower bound for $t_\e -\tau_\e$. 
Using~\eqref{borne-xx} and~\eqref{borne-yy}, we get
\begin{eqnarray*}
  &&   \dot{X}_2^\e = \sin\theta -\frac{2}{\e} (X_1^\e)^2 X_2^\e
  \geq \sin\theta -2 \sqrt{\frac{-\cos\theta}{2\overline{A}}} g(\overline{A})^2 =: -C
\quad \text{on $[\tau_\e,  t_\e]$,}
\end{eqnarray*}
for some constant $C>0$ independent of $\e$. It follows that
$X_2^\e(t_\e)-X_2^\e(\tau_\e)\geq -C(t_\e -\tau_\e)$, from which we infer
\begin{eqnarray*}
  &&  t_\e -\tau_\e \geq \frac{1}{C} X_2^\e(\tau_\e)= \frac{1}{C}\sqrt{\frac{-\cos\theta}{2\overline{A}}}\e^{1/3}.
\end{eqnarray*}
Taking into account this estimate in~\eqref{estim-eta}, we obtain the second inequality in~\eqref{ineq-etae}
for some constant $\gamma >0$ depending only on $\theta$.
\end{proof}


\begin{lemma}\label{xprimeneg}
Let  $\theta\in (\pi, \frac{5\pi}{4}]$ (case (2)). Define
\begin{eqnarray*}
   && \mathcal{C}_1^{\e ,1}:=\left\{ x=(x_1,x_2)\in\R^2:  F_1^\e(x, e_\theta) = \cos\theta -\frac{2}{\e}x_1x_2^2 \leq 0\right\}.
\end{eqnarray*}
Then, $\mathcal{C}_1^{\e ,1}\cap \{ x_1 -\eta_\e \leq x_2\}\cap \{x_1\leq 0\}$ is invariant for~\eqref{traj-eps}.
In particular, for any $x\in \mathcal{C}_1^{\e ,1}$, $\dot{X}_1^\e(t)\leq 0$ for all $t\geq 0$.
\end{lemma}

\begin{remark}
Note that $\mathcal{C}_1^{\e ,1}$ is defined so that
$X^\e(t)\in \mathcal{C}_1^{\e ,1}$ if and only if $\dot{X}_1^\e(t)\leq 0$
with  $\dot{X}_1^\e(t)= 0$ on $\partial\mathcal{C}_1^{\e ,1}$.
Unfortunately, Lemma~\ref{xprimeneg} is not sufficient to prove
Proposition~\ref{behav-1}. Indeed, the only easy consequence is
$X_1^\e(t)\leq \eta (t,\e) <0$ for all $t\geq t_\e$, but we may have
$\eta (t,\e)\to 0$ as $\e\to 0$. To overcome this difficulty,
we will need to refine the subset $\mathcal{C}_1^{\e ,1}$ above.
\end{remark}

\begin{figure}[ht]
  \begin{center}
\includegraphics[width=7cm]{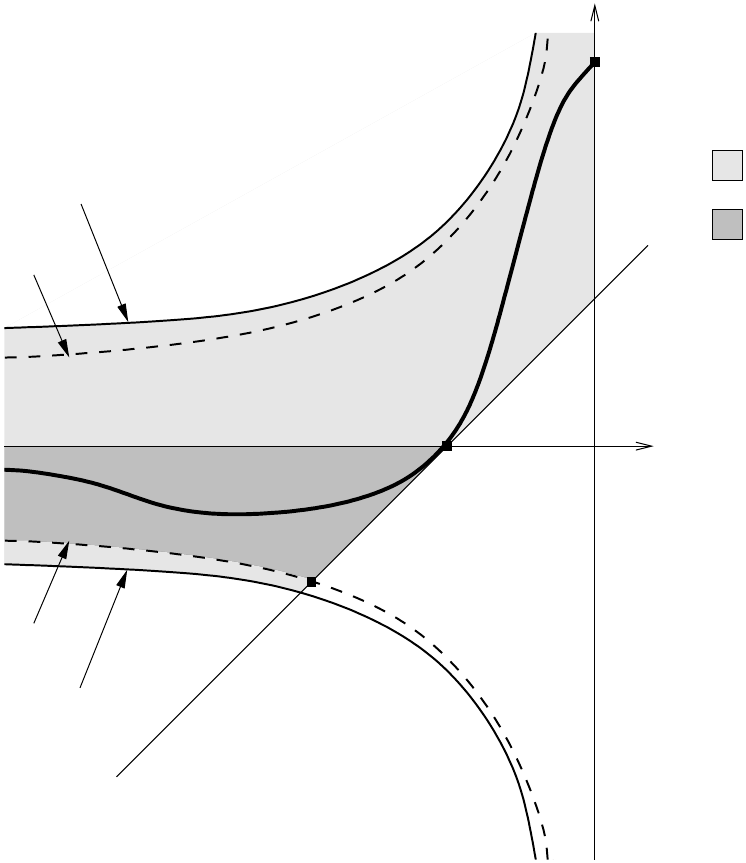}
\caption{}
\label{dess_traj-zones}
\unitlength=1pt
\begin{picture}(0,0)
\put(62,152){$\scriptstyle O$}
\put(63,251){$e_N=(0,1)$}
\put(58,271){$\scriptstyle x_2$}
\put(78,147){$\scriptstyle x_1$}
\put(20,140){$(\eta_\e, \!0)$}
\put(58,186){$\cdot\hspace{.1cm}(0,-\eta_\e)$}
\put(-8,113){$x^{\e, \nu}$}
\put(105,206){$\mathcal{C}_1^{\e ,\nu}\scriptstyle \cap \{ x_1 - \eta_\e\leq x_2\}\cap \{ x_2\leq 0\}$}
\put(105,221){$\mathcal{C}_1^{\e ,1}\scriptstyle \cap \{ x_1 -\eta_\e \leq x_2\}\cap \{x_1\leq 0\}$}
\put(-12,168){$X^{e_N, e_\theta, \e}$}
\put(-125,227){$\scriptstyle  \displaystyle x_2=\sqrt{\frac{\e \cos\theta}{2 x_1}}$}
\put(-165,195){$\scriptstyle \displaystyle x_2=\nu \sqrt{\frac{\e \cos\theta}{2 x_1}}$}
\put(-159,65){$\scriptstyle \displaystyle x_2=-\sqrt{\frac{\e \cos\theta}{2 x_1}}$}
\put(-176,95){$\scriptstyle \displaystyle x_2=-\nu \sqrt{\frac{\e \cos\theta}{2 x_1}}$}
\put(-60,60){$\scriptstyle \displaystyle x_2= x_1-\eta_\e$}
\end{picture}
\end{center}
\end{figure}

\begin{proof} Notice first that
\begin{eqnarray*}
  && \mathcal{C}_1^{\e ,1}:=\left\{ (x_1,x_2)\in\R^2:
  -\sqrt{\frac{\e \cos\theta}{2x_1}} \leq x_2\leq \sqrt{\frac{\e \cos\theta}{2x_1}}, x_1 <0\right\}
  \cup\left\{ (0,x_2), x_2\in\R\right\}
\end{eqnarray*}
is in between two hyperbola branches in $\{x_1 <0\}$ (we refer the reader to Figure~\ref{dess_traj-zones}
for an illustration of the proof).

Let $x$ be a point of the upper-branch, \emph{i.e.},
$$
x\in \{ x_2= \sqrt{\frac{\e \cos\theta}{2x_1}}, x_1 <0 \}= \partial\mathcal{C}_1^{\e ,1}\cap \{x_2>0\} \subset \{ x_1 <x_2\}.
$$
We have $\langle n_{\partial\mathcal{C}_1^{\e ,1}}(x), e_E\rangle <0$ and  $\langle n_{\partial\mathcal{C}_1^{\e ,1}}(x), e_N\rangle >0$,
whereas $F_1^\e(x)=0$ (by definition) and  $F_2^\e(x)<0$ (since $\sin\theta <0$ and $x_2>0$).
Hence $\langle F^\e(x), n_{\partial\mathcal{C}_1^{\e ,1}}(x)\rangle <0$, which proves that the
upper-boundary of $\mathcal{C}_1^{\e ,1}$ is invariant.

The lower-boundary of $\mathcal{C}_1^{\e ,1}\cap \{ x_1-\eta_\e \leq x_2\}$ consists of two parts.
The first one is a piece of $\{x_1-\eta_\e=x_2\}$.
The second one is a part of the lower-branch of the hyperbola, more precisely
\begin{eqnarray}\label{piece-hyp-low}
\{ x_2= -\sqrt{\frac{\e \cos\theta}{2x_1}}, x_1 -\eta_\e\leq x_2 \}.
\end{eqnarray}

We first prove that $\{ x_1 -\eta_\e \leq x_2\}$ is invariant for $t\geq t_\e$.
Notice that $z=X_2^\varepsilon- X_1^\varepsilon$  is the solution
of the linear ODE~\eqref{edo-l} with $a(t)= 2\e^{-1} X_1^\varepsilon(t)X_2^\varepsilon(t)$,
$z(t_\e)= -\eta_\e \geq 0$ and $\alpha = \sin\theta - \cos\theta \geq 0$ (since $\theta\in (\pi, \frac{5\pi}{4}]$).
Hence, by Lemma~\ref{qualit-lin-ode}, we obtain $z(t)\geq  -\eta_\e$ for all $t\geq t_\e$, which proves
the result.

Now, let $x$ be on the piece of curve given by~\eqref{piece-hyp-low}.
We have $\langle n_{\partial\mathcal{C}_1^{\e ,1}}(x), e_E\rangle <0$, $\langle n_{\partial\mathcal{C}_1^{\e ,1}}(x), e_N\rangle <0$,
and $F_1^\e(x)=0$. It follows that $\langle F^\e(x), n_{\partial\mathcal{C}_1^{\e ,1}}(x)\rangle \leq 0$
holds if and only if $F_2^\e(x)\geq 0$. But, since $(\sin\theta - \cos\theta) \geq 0$ for $\theta\in (\pi, \frac{5\pi}{4}]$ and
$x_1\leq x_1-\eta_\e\leq x_2\leq 0$,
\begin{eqnarray*}
&& F_2^\e(x)-F_1^\e(x) = (\sin\theta - \cos\theta) +\frac{2}{\e} x_1 x_2 (x_2 -x_1) \geq 0
\end{eqnarray*}
Recalling that $F_1^\e(x)=0$, this
concludes the proof.
\end{proof}

We now refine the invariant subset obtained in Lemma~\ref{xprimeneg} in order to have $\dot{X}_1^\e <0$
uniformly with respect to $\e$.

\begin{lemma}\label{xprime-strict-neg}
Let  $\theta\in (\pi, \frac{5\pi}{4}]$ (case (2)). For $0<\nu \leq 1$, define
\begin{eqnarray*}
   && \mathcal{C}_1^{\e ,\nu}:=\left\{ x=(x_1,x_2)\in\R^2:  F_1^\e(x, e_\theta) = \cos\theta -\frac{2}{\e}x_1x_2^2 \leq (1-\nu^2)\cos\theta\right\}.
\end{eqnarray*}
Then, there exists $\overline\nu < 1$ such that 
$\mathcal{C}_1^{\e ,\overline\nu}\cap \{ x_1 - \eta_\e\leq x_2\}\cap \{ x_2\leq 0\}$ is invariant for~\eqref{traj-eps} on $[t_\e,\infty)$.
In particular, for any $x\in\mathcal{C}_1^{\e ,\overline\nu}\cap \{ x_1 - \eta_\e\leq x_2\}\cap \{ x_2\leq 0\}$, one has $\dot{X}_1^\e(t)\leq  (1-\overline\nu^2)\cos\theta$, for all $t\geq t_\e$.
\end{lemma}

\begin{proof}
Due to Lemma~\ref{tepsilon}, it is sufficient to prove that, for all
$$
x\in \partial\mathcal{C}_1^{\e ,\nu}\cap \{ x_1 - \eta_\e\leq x_2\}\cap \{ x_2\leq 0\},
$$
we have $\langle F^\e (x), n_{\partial\mathcal{C}_1^{\e ,\nu}}(x) \rangle \leq 0$.  
The proof is more involved than the one of Lemma~\ref{xprimeneg} because
we do not have $F_1^\e (x)=0$ anymore.

We can rewrite $\partial\mathcal{C}_1^{\e ,\nu}\cap \{ x_1 - \eta_\e\leq x_2\}\cap \{ x_2\leq 0\}$
as the piece of curve
\begin{eqnarray}\label{courbe-bas}
&& \{ x_2= -\nu\sqrt{\frac{\e \cos\theta}{2x_1}}, \ x_1 \leq x^{\e, \nu}_1 \},
\end{eqnarray}
where $x^{\e, \nu}=(x^{\e, \nu}_1,x^{\e, \nu}_2)$ is the unique intersection point
of the hyperbola branch $\{x_2= -\nu\sqrt{\frac{\e \cos\theta}{2x_1}}\}$ and the line segment 
$\{ x_1 - \eta_\e= x_2, \ x^{\e, \nu}_1\leq x_1\leq\eta_\e \}$ (see Figure~\ref{dess_traj-zones}).

We first compute  $x^{\e, \nu}$. Recalling $\eta_\e = - \gamma \e^{1/3}$, 
we obtain the equation
$$
x^{\e, \nu}_1+\gamma \e^{1/3} = -\nu \sqrt{\frac{\e \cos\theta}{2 x^{\e, \nu}_1}}.
$$
Setting $r=\e^{-1/6}\sqrt{- x^{\e, \nu}_1}$, the equation reads
$$
P(r):= r^3 -\gamma r -\nu \sqrt{\frac{-\cos\theta}{2}}.
$$
The polynomial $P$ has a unique positive root $r_\nu >\sqrt{\gamma/3}$ and we eventually
obtain $x^{\e, \nu}_1 =- r_\nu^2 \e^{1/3}$, and $x^{\e, \nu}_2$ is determined accordingly.

We now take $x$ on the piece of curve given by~\eqref{courbe-bas} and we compute
$\langle F^\e (x), n_{\partial\mathcal{C}_1^{\e ,\nu}}(x) \rangle$. A straightforward computation
gives that
$$
n_{\partial\mathcal{C}_1^{\e ,\nu}}(x)=\lambda \left( -\nu \sqrt{\frac{\e \cos\theta}{8 x_1^3}}, -1\right)
$$
for some $\lambda >0$. Then, plugging $x_2 = -\nu\sqrt{\frac{\e \cos\theta}{2x_1}}$ in the formulas of $F^\e_1$ and $F^\e_2$
given by~\eqref{edo-c}, a tedious computation leads to
\begin{eqnarray*}
  Q^{\e}(x,\nu)&:=& \langle F^\e (x), n_{\partial\mathcal{C}_1^{\e ,\nu}}(x) \rangle\\
  &=& \frac{(-\cos\theta)^{3/2}}{2\sqrt{2}}\sqrt{\frac{\e}{(-x_1)^3}}(\nu -\nu^3)
  -(-2\cos\theta)^{1/2}\sqrt{\frac{(-x_1)^3}{\e}} \nu -\sin\theta. 
\end{eqnarray*}
Noticing that the above quantity is  nondecreasing with respect to $x_1\in  (-\infty,x^{\e, \nu}_1]$,
we obtain 
\begin{eqnarray*}
&&  Q^{\e}(x,\nu)\leq  Q^{\e}(x^{\e, \nu}, \nu)
  \quad \text{for all $x\in \partial\mathcal{C}_1^{\e ,\nu}\cap \{ x_1 - \eta_\e\leq x_2\}\cap \{ x_2\leq 0\}$.}
\end{eqnarray*}
Plugging $x^{\e, \nu}_1 =- r_\nu^2 \e^{1/3}$ and recalling that $r_\nu^3 =\gamma r_\nu +\nu \sqrt{(-\cos\theta)/2}$, we obtain
\begin{eqnarray*}
Q^{\e}(x^{\e, \nu}, \nu)&=& \frac{(-\cos\theta)^{3/2}}{2\sqrt{2}  r_\nu^3}(\nu -\nu^3)
  -\sqrt{2}(-\cos\theta)^{1/2} r_\nu^3 \nu -\sin\theta\\
  &=&   \frac{(-\cos\theta)^{3/2}}{2\sqrt{2}  r_\nu^3}(\nu -\nu^3)
  + \nu^2 \cos\theta - \sin\theta -\sqrt{2}(-\cos\theta)^{1/2}\gamma \nu r_\nu,
\end{eqnarray*}
which is independent of $\e$.
Since $\nu\in (0,1]\mapsto r_\nu$ is continuous and $r_\nu >\sqrt{\gamma/3}>0$, the mapping 
$\nu\in (0,1]\mapsto  Q^{\e}(x^{\e, \nu}, \nu)$ is obviously continuous. Moreover,
$$
Q^{\e}(x^{\e, 1}, 1) = \cos\theta - \sin\theta -\sqrt{2}(-\cos\theta)^{1/2}\gamma r_1 < 0,
$$ 
since $\cos\theta \leq \sin\theta <0$ and $r_1 >0$.
By continuity, there exists $\overline\nu <1$ sufficiently close to 1, independent of $\e$, such that
$Q^{\e}(x^{\e, \overline\nu}, \overline\nu)< 0$. The conclusion follows.
\end{proof}

We are now in a position to give the proof of (2).
\medskip

\noindent{\it Proof of (2).}
From~Lemmas \ref{tepsilon} and \ref{etaepsilon}, sending $\e\to 0$,  we obtain that 
$\eta_\e\to 0$ and~$t_\e\to (-\sin\theta)^{-1}$.
By Lemma~\ref{xprime-strict-neg}, we get~$X_1^\e (t)\leq \eta_\e +  (1-\overline\nu^2)(\cos\theta) (t-t_\e)$ for all~$t\geq t_\e$
and  $X_1 (t)\leq (1-\overline\nu^2)(\cos\theta) (t- (-\sin\theta)^{-1})$ at the limit $\e\to 0$.
It follows that~$X$ enters the branch $W$  at $t=(-\sin\theta)^{-1}$.
Using Theorem~\ref{edo-interieur}, we obtain~\eqref{traj-W}.
\qed
\medskip

The following lemma is needed to prove (3).

\begin{lemma}\label{yprimeneg12}
Let  $\theta\in (\frac{5\pi}{4}, \frac{3\pi}{2}]$ (case (3)). For $0<\nu \leq 1$, define
\begin{eqnarray*}
   && \mathcal{C}_2^{\e ,\nu}:=\left\{ x=(x_1,x_2)\in\R^2:  F_2^\e(x, e_\theta) = \sin\theta -\frac{2}{\e}x_1^2x_2 \leq (1-\nu)\sin\theta\right\}.
\end{eqnarray*}
There exist $0<\nu <1$ sufficiently close to 1, and $\rho >0$ small enough, independent of~$\e$, such that the set
\begin{eqnarray}\label{invariant12}
   && \mathcal{C}_2^{\e ,\nu}\cap \{ x_2\leq x_1 +\rho \e^{1/3}\}\cap \{ x_1 \leq 0\}\cap \{ x_2 \leq 0\}
\end{eqnarray}
is invariant for~\eqref{traj-eps}.
\end{lemma}

\begin{figure}[ht]
  \begin{center}
\includegraphics[width=8cm]{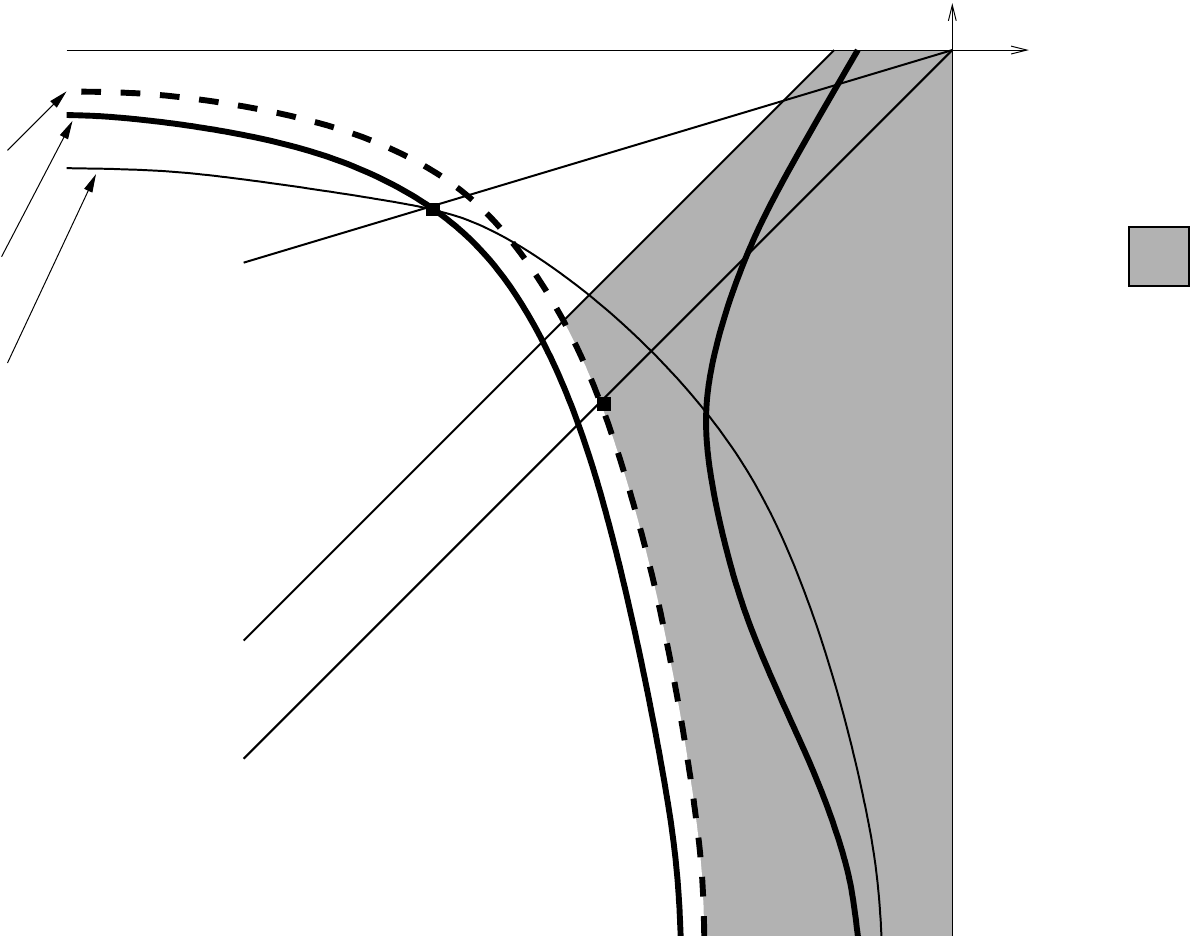}
\caption{}
\label{dess_preuve12}
\unitlength=1pt
\begin{picture}(0,0)
\put(71,210){$\scriptstyle O$}
\put(-38,168){$\scriptstyle x^\theta$}
\put(65,221){$\scriptstyle x_2$}
\put(84,203){$\scriptstyle x_1$}
\put(25,145){$X^{e_N, e_\theta, \e}$}
\put(5,138){$\scriptstyle y^{\theta\nu}$}
\put(-170,188){$x_2=\nu \frac{\e \sin\theta}{2 x_1^2}$}
\put(-165,161){$x_2=\frac{\e \sin\theta}{2 x_1^2}$}
\put(-177,128){$\scriptstyle \displaystyle x_2=-\sqrt{\frac{\e \cos\theta}{2 x_1}}$}
\put(-90,156){$x_2=\frac{\cos\theta}{\sin\theta}x_1$}
\put(-105,70){$x_2=x_1$}
\put(-145,89){$x_2=x_1 +\rho \e^{1/3}$}
\put(118,165){$\mathcal{C}_2^{\e ,\nu} \cap \{ x_2 \leq x_1 + \rho\e^{1/3}\}$}
\end{picture}
\end{center}
\end{figure}

\begin{remark}
We refer the reader to Figure~\ref{dess_preuve12} for an illustration of the proof
of Lemma~\ref{yprimeneg12}, which follows.
In the case (3), $\theta\in (\frac{5\pi}{4}, \frac{3\pi}{2}]$, the equilibrium
point of~\eqref{edo-c} is given by
\begin{eqnarray}\label{point-eq12}
  && x^\theta= ( - (r_1^\theta)^{1/3} , - (r_2^\theta)^{1/3} )\e^{1/3} :=
  \left( -(\frac{(\sin\theta)^2}{2(-\cos\theta)})^{1/3} ,  -(\frac{(\cos\theta)^2}{2(-\sin\theta)})^{1/3}\right)\e^{1/3}.
\end{eqnarray}
\end{remark}

\begin{proof}
Let $S$ be the subset defined by~\eqref{invariant12}. 
From Lemma~\ref{tepsilon}, we know that the ``northeastern'' part $\{ x_1 \leq 0\}\cap \{ x_2 \leq 0\}$ of $S$
is invariant.
The remaining ``western" part of the boundary of $S$ consists of
a line segment and a piece of hyperbola.

Let $x$ be on the  line segment $\{ x_2 =  x_1 +\rho \e^{1/3}\}\cap S$, for some $\rho>0$.
Since $n_{\partial S}(x)=(-1,1)$, we have
\begin{eqnarray*}
  \langle F^\e (x), n_{\partial S}(x) \rangle
  &=& -\cos\theta + \frac{2}{\e} (x_2-\rho \e^{1/3})x_2^2 + \sin\theta -  \frac{2}{\e} (x_2-\rho \e^{1/3})^2 x_2\\
  &=& (\sin\theta -\cos\theta) +2\rho \left( \left(\frac{x_2}{\e^{1/3}}\right)^2  - \rho \left(\frac{x_2}{\e^{1/3}}\right)\right).
\end{eqnarray*}
Noting that, for small $\e>0$,  $x_2=x_1 +\rho \e^{1/3}\geq x_1^\theta+\rho \e^{1/3} \geq -(r_1^\theta)^{1/3}\e^{1/3}$,
where $x_1^\theta$ is the first component in~\eqref{point-eq12}, we obtain
\begin{eqnarray*}
&&  \langle F^\e (x), n_{\partial S}(x) \rangle \leq (\sin\theta -\cos\theta) +2\rho \left( (r_1^\theta)^{2/3} +\rho (r_1^\theta)^{1/3}\right).
\end{eqnarray*}
Recalling that $(\sin\theta -\cos\theta)<0$ for $\theta\in (\frac{5\pi}{4}, \frac{3\pi}{2}]$,
we infer  that $\langle F^\e (x), n_{\partial S}(x) \rangle <0$ for~$\rho >0$  small enough.

Next, we show that the piece of hyperbola is invariant for $0 <\nu <1$ sufficiently close to 1. Let
\begin{eqnarray*}
&& x\in  \partial \mathcal{C}_2^{\e ,\nu}\cap S =\{ x_2= \nu \frac{\e \sin\theta}{2 x_1^2}, x_1\geq x_2 - \rho \e^{1/3}\}.
\end{eqnarray*}
Using $x_2= \nu \frac{\e \sin\theta}{2 x_1^2}$ and setting $r := \e^{-1} (-x_1)^{3}$, tedious computations lead to
$n_{\partial S}(x)= (-\nu\frac{\e \sin\theta}{x_1^3}, -1)$ and
\begin{eqnarray*}
&&  \langle F^\e (x), n_{\partial S}(x) \rangle
= (-\sin\theta)\left( \frac{-\cos\theta}{r}\nu - \frac{(\sin\theta)^2}{2r^2}\nu^3 +1 -\nu\right)=: Q(\nu, r). 
\end{eqnarray*}
We set
\begin{eqnarray*}
  y^{\theta\nu} &:=&\{x_2=x_1\}\cap \{ x_2=\nu \frac{\e \sin\theta}{2 x_1^2} \}\\
  &=&  \left(-(\nu \frac{(-\sin\theta)}{2})^{1/3},  -(\nu \frac{(-\sin\theta)}{2})^{1/3}\right)\e^{1/3} 
  =: \left(-(s_1^{\theta\nu})^{1/3} , -(s_2^{\theta\nu})^{1/3}\right)\e^{1/3}.
\end{eqnarray*}
Since $\sin\theta < \cos\theta <0$ and $\nu <1$, one has
\begin{eqnarray*}
  && s_1^{\theta\nu} = \nu \frac{(-\sin\theta)}{2} = \nu \frac{\cos\theta}{\sin\theta}  r_1^\theta <  r_1^\theta
\end{eqnarray*}
and, for $\rho >0$ small enough, there exists $\overline{r}>0$ satisfying
\begin{eqnarray*}
&& 0 < r=\left(\frac{-x_1}{\e^{1/3}}\right)^{3} \leq ( (s_1^{\theta\nu})^{1/3} +\rho )^3 \leq \overline{r}< r_1^\theta.
\end{eqnarray*}
A straightforward computation yields
$$
\max_{0<r\leq \overline{r}< r_1^\theta} Q(1,r) =  Q(1,\overline{r}) <  Q(1,r_1^\theta) =0.
$$
By continuity of the mapping $(\nu,r)\in (0,1]\times (0,\overline{r}]\mapsto Q(\nu,r)$, 
there exists some positive $\nu$ close to 1, such that $\langle F^\e (x), n_{\partial S}(x) \rangle <0$
for all $x$ on the piece of hyperbola.
Lemma~\ref{yprimeneg12} is proved.
\end{proof}

\noindent{\it Proof of (3).}
By Lemmas~~\ref{tepsilon} and \ref{etaepsilon}, $y^\e:=X^{e_N, e_\theta , \e}(t_\e) = (-\gamma^{e_N,e_\theta}\e^{1/3}, 0)$,
where~$\gamma^{e_N,e_\theta} >0$ is independent of $\e$.
From Lemma~\ref{lem-scaling}, we have
\begin{eqnarray}\label{form-sc}
&& X^{e_N, e_\theta , \e}(t+t_\e)= X^{y^\e , e_\theta, \e}(t)=
\frac{1}{\rho} X^{\rho y^\e , e_\theta, \rho^3 \e}(\rho t) \quad \text{for all $t\geq 0,\ \rho>0$.}
\end{eqnarray}
Setting $\rho =\tilde{\rho} (\gamma^{e_N, e_\theta})^{-1}$ and $\tilde{\e} = \rho^3 \e$,
from Lemma~\ref{yprimeneg12}, for $\tilde{\rho}$ small enough and $\nu$ sufficiently close
to 1, we obtain that $X^{\rho y^\e , e_\theta, \tilde\e}(s)\in  \mathcal{C}_2^{\tilde\e ,\nu}$ for
all $s\geq 0$. In particular, $\dot{X}_2^{\rho y^\e , e_\theta, \tilde\e}(s)\leq (1-\nu)\sin\theta, s\geq 0$.
It follows that, ${X}_2^{\rho y^\e , e_\theta, \tilde\e}(s)\leq (1-\nu)(\sin\theta)s$, for all~$s\geq 0$.
From~\eqref{form-sc}, we have $X_2^{e_N, e_\theta , \e}(t+t_\e)\leq (1-\nu)(\sin\theta)t$
for all $t\geq 0$. Sending $\e\to 0$, we get $X_2^{e_N, e_\theta}(t+(-\sin\theta)^{-1})\leq (1-\nu)(\sin\theta)t$,
which proves that $X^{e_N, e_\theta}$ enters the branch $S$ for $t> (-\sin\theta)^{-1}$.
This completes the proof.
\qed

\section{Proofs of Lemmas~\ref{controle-vitesse-123} and~\ref{traj-y-tilde}}
\label{sec-proof-lemma123}

\begin{proof}[Proof of Lemma~\ref{controle-vitesse-123}]
The first statement is a straightforward application of
Proposition~\ref{prop-inv}\,(3) for $\e$ small enough compared to $\gamma$.

For a.e. $t\in\mathcal{T}_N$, we have
\begin{eqnarray*}
|\dot{X}_{2}^{x,\alpha,\e}|
&=& |f_2( X^{x,\alpha,\e}, \alpha) - 2\e^{-1} (X_1^{x,\alpha,\e})^2 X_2^{x,\alpha,\e}|\\
&\leq & |f|_\infty +   2\e^{-1} \sqrt{d(X^{x,\alpha,\e})} |X_1^{x,\alpha,\e}|\\
&=& |f|_\infty (1+ C\e^{5\gamma/24}),
\end{eqnarray*}
since, for  $t\in\mathcal{T}_N$, we know that $d(X^{x,\alpha,\e})= (X_1^{x,\alpha,\e} X_2^{x,\alpha,\e})^2\leq \e^{4\gamma/3}$
and $X_1^{x,\alpha,\e}\leq \e^{13\gamma/24}$.
We prove the second inequality with similar computations.
\end{proof}

\begin{proof}[Proof of Lemma~\ref{traj-y-tilde}]
By definition of $\underline{Y}$, we have, for all $t\geq C\e^{1-\gamma}$,
\begin{eqnarray*}
&& \underline{Y}(t)=
 (1 +C\e^{5\gamma/24})^{-1}
\left\{
\begin{array}{ll}
(0,  X_2^{x,\alpha,\e}(t)-\e^{\gamma/8}), & t\in \mathcal{T}_N,\\  
(0,0), & t\in \mathcal{T}_O,\\
(X_1^{x,\alpha,\e}(t)-\e^{\gamma/8},0), & t\in \mathcal{T}_E,  
\end{array}  
\right.
\end{eqnarray*}  
which, since $X^{x,\alpha, \varepsilon}\in W_{\rm loc}^{1,\infty}[0,\infty)$, proves (1). The proof of (2) follows from the estimates in Lemma~\ref{controle-vitesse-123}.

We now prove~\eqref{estim-traj-111} for $\underline{Y}$.
When $t\in \mathcal{T}_O$, we obviously have
\begin{eqnarray*}
  && |\underline{Y}(t)-X^{x,\alpha,\e}(t)| \,=\, |X^{x,\alpha,\e}(t)|\leq \sqrt{(\e^{\gamma/8})^2+(\e^{\gamma/8})^2}
  \leq \sqrt{2} \e^{\gamma/8}.
\end{eqnarray*} 

We turn to the case when $t\in \mathcal{T}_N$ (the case $t\in \mathcal{T}_E$ is similar).
Let $(s_i^N,t_i^N) \subset  \mathcal{T}_N$ be the unique interval containing $t$.
For a.e. $s\in (s_i^N,t_i^N)$, using Lemma~\ref{controle-vitesse-123}, we get
\begin{eqnarray}\label{est-eacrt-vitesse-2}
|\dot{\underline{Y}}_2(s)-\dot{X}_2^{x,\alpha,\e}(s)|
\leq |(1 +C\e^{5\gamma/24})^{-1}\dot{X}_2^{x,\alpha,\e}(s) -\dot{X}_2^{x,\alpha,\e}(s)| \leq C\e^{5\gamma/24}.
\end{eqnarray} 
Hence, since, by definition, $\underline{Y}(s_i^N)=(0,0)$ and $X^{x,\alpha,\e}(s_i^N)=(X_1^{x,\alpha,\e}(s_i^N),  \e^{\gamma/8})$, integrating~\eqref{est-eacrt-vitesse-2} over $(s_i^N,t)$, we obtain
\begin{eqnarray*}
|\underline{Y}_2(t)-X_2^{x,\alpha,\e}(t)|
\leq |\underline{Y}_2(s_i^N)-X_2^{x,\alpha,\e}(s_i^N)|+  C\e^{5\gamma/24}(t-s_i^N)
\leq \e^{\gamma/8}+  C\e^{5\gamma/24}(t-s_i^N).
\end{eqnarray*} 
It follows that 
\begin{eqnarray*}
|\underline{Y}(t)-X^{x,\alpha,\e}(t)|
&\leq&  |X_1^{x,\alpha,\e}| + |\underline{Y}_2(t)-X_2^{x,\alpha,\e}(t)|\\
&\leq& \e^{13\gamma/24} + \e^{\gamma/8}+  C\e^{5\gamma/24}(t-s_i^N)
\leq C( \e^{\gamma/8} +  \e^{5\gamma/24}t).
\end{eqnarray*}
This ends the proof of Lemma~\ref{traj-y-tilde}.
\end{proof}


\begin{thebibliography}{10}

\bibitem{acct13}
Yves Achdou, Fabio Camilli, Alessandra Cutr\`\i, and Nicoletta Tchou.
\newblock Hamilton-{J}acobi equations constrained on networks.
\newblock {\em NoDEA Nonlinear Differential Equations Appl.}, 20(3):413--445,
  2013.

\bibitem{aot15}
Yves Achdou, Salom\'{e} Oudet, and Nicoletta Tchou.
\newblock Hamilton-{J}acobi equations for optimal control on junctions and
  networks.
\newblock {\em ESAIM Control Optim. Calc. Var.}, 21(3):876--899, 2015.

\bibitem{at15}
Yves Achdou and Nicoletta Tchou.
\newblock Hamilton-{J}acobi equations on networks as limits of singularly
  perturbed problems in optimal control: dimension reduction.
\newblock {\em Comm. Partial Differential Equations}, 40(4):652--693, 2015.

\bibitem{ab03}
O.~Alvarez and M.~Bardi.
\newblock Singular perturbations of nonlinear degenerate parabolic {PDE}s: a
  general convergence result.
\newblock {\em Arch. Ration. Mech. Anal.}, 170(1):17--61, 2003.

\bibitem{ab10}
Olivier Alvarez and Martino Bardi.
\newblock Ergodicity, stabilization, and singular perturbations for
  {B}ellman-{I}saacs equations.
\newblock {\em Mem. Amer. Math. Soc.}, 204(960):vi+77, 2010.

\bibitem{ag00}
Z.~Artstein and V.~Gaitsgory.
\newblock The value function of singularly perturbed control systems.
\newblock {\em Appl. Math. Optim.}, 41(3):425--445, 2000.

\bibitem{bb98}
Fabio Bagagiolo and Martino Bardi.
\newblock Singular perturbation of a finite horizon problem with state-space
  constraints.
\newblock {\em SIAM J. Control Optim.}, 36(6):2040--2060, 1998.

\bibitem{bcd97}
M.~Bardi and I.~Capuzzo~Dolcetta.
\newblock {\em Optimal control and viscosity solutions of
  {H}amilton-{J}acobi-{B}ellman equations}.
\newblock Birkh\"auser Boston Inc., Boston, MA, 1997.

\bibitem{barles94}
G.~Barles.
\newblock {\em Solutions de viscosit\'e des \'equations de
  {H}amilton-{J}acobi}.
\newblock Springer-Verlag, Paris, 1994.

\bibitem{bbc13}
G.~Barles, A.~Briani, and E.~Chasseigne.
\newblock A {B}ellman approach for two-domains optimal control problems in
  {$\Bbb R^N$}.
\newblock {\em ESAIM Control Optim. Calc. Var.}, 19(3):710--739, 2013.

\bibitem{bbc14}
G.~Barles, A.~Briani, and E.~Chasseigne.
\newblock A {B}ellman approach for regional optimal control problems in
  {$\Bbb{R}^N$}.
\newblock {\em SIAM J. Control Optim.}, 52(3):1712--1744, 2014.

\bibitem{bc24}
Guy Barles and Emmanuel Chasseigne.
\newblock {\em On {M}odern {A}pproaches of {H}amilton-{J}acobi {E}quations and
  {C}ontrol {P}roblems with {D}iscontinuities}, volume 104 of {\em Progress in
  Nonlinear Differential Equations and their Applications}.
\newblock Birkh\"{a}user/Springer, Cham, 2024.
\newblock A Guide to Theory, Applications, and Some Open Problems, PNLDE
  Subseries in Control, 104.

\bibitem{brezis83}
Ha\"{\i}m Brezis.
\newblock {\em Analyse fonctionnelle}.
\newblock Collection Math\'{e}matiques Appliqu\'{e}es pour la Ma\^{i}trise.
  [Collection of Applied Mathematics for the Master's Degree]. Masson, Paris,
  1983.
\newblock Th\'{e}orie et applications. [Theory and applications].

\bibitem{chuberre23}
M\'eriadec Chuberre.
\newblock {\em Th\`ese de doctorat}.
\newblock Institut National des Sciences Appliqu\'ees de Rennes, 2023.

\bibitem{hale80}
Jack~K. Hale.
\newblock {\em Ordinary differential equations}.
\newblock Robert E. Krieger Publishing Co., Inc., Huntington, NY, second
  edition, 1980.

\bibitem{imz13}
Cyril Imbert, R\'{e}gis Monneau, and Hasnaa Zidani.
\newblock A {H}amilton-{J}acobi approach to junction problems and application
  to traffic flows.
\newblock {\em ESAIM Control Optim. Calc. Var.}, 19(1):129--166, 2013.

\bibitem{jz23}
Othmane Jerhaoui and Hasnaa Zidani.
\newblock A general comparison principle for {H}amilton {J}acobi {B}ellman
  equations on stratified domains.
\newblock {\em ESAIM Control Optim. Calc. Var.}, 29:Paper No. 9, 38, 2023.

\bibitem{ls84}
P.-L. Lions and A.~S. Sznitman.
\newblock Stochastic differential equations with reflecting boundary
  conditions.
\newblock {\em Comm. Pure Appl. Math.}, 37:511--537, 1984.

\bibitem{ls17}
Pierre-Louis Lions and Panagiotis Souganidis.
\newblock Well-posedness for multi-dimensional junction problems with
  {K}irchoff-type conditions.
\newblock {\em Atti Accad. Naz. Lincei Rend. Lincei Mat. Appl.},
  28(4):807--816, 2017.

\bibitem{lojasiewicz84}
S.~{\L}ojasiewicz.
\newblock Sur les trajectoires du gradient d'une fonction analytique.
\newblock In {\em Geometry seminars, 1982--1983 (Bologna, 1982/1983)}, pages
  115--117. Univ. Stud. Bologna, Bologna, 1984.

\bibitem{sc13}
Dirk Schieborn and Fabio Camilli.
\newblock Viscosity solutions of {E}ikonal equations on topological networks.
\newblock {\em Calc. Var. Partial Differential Equations}, 46(3-4):671--686,
  2013.

\end{thebibliography}



\end{document}